\documentclass[a4paper,10pt]{article}
\usepackage[vlined]{algorithm2e}

\SetFuncSty{textsc}
\SetKwFunction{frun}{Run}
\SetKwHangingKw{arun}{\frun}
\SetKwFunction{fset}{Set}
\SetKwHangingKw{aset}{\fset}
\SetKwFunction{fselect}{Select}
\SetKwHangingKw{aselect}{\fselect}
\SetKwFunction{fcompute}{Compute}
\SetKwHangingKw{acompute}{\fcompute}
\SetKwFunction{fsolve}{Solve}
\SetKwHangingKw{asolve}{\fsolve}
\SetKwFunction{festimate}{Estimate}
\SetKwHangingKw{aestimate}{\festimate}
\SetKwFunction{fmark}{Mark}
\SetKwHangingKw{amark}{\fmark}
\SetKwFunction{frefine}{Refine}
\SetKwHangingKw{arefine}{\frefine}
\SetKwFunction{compute}{Compute}
\SetKwFunction{set}{Set}

{\algorithm}%
{\endalgorithm}

\usepackage{floatrow}
\newfloatcommand{capbtabbox}{table}[][\FBwidth]

\usepackage{blindtext}

\title{Error analysis of Nitsche's and discontinuous Galerkin
	methods of a reduced Landau-de Gennes problem 
}
\author{Ruma Rani Maity\footnote{
		Department of Mathematics, Indian Institute of Technology Bombay, Powai, Mumbai 400076, India. Email. ruma@math.iitb.ac.in} $\;\;$
	Apala Majumdar* \footnote{* Department of Mathematics And Statistics, University of Strathclyde, 16 Richmond St, Glasgow G1 1XQ, United Kingdom. Visiting Professor,  Indian Institute of Technology Bombay, Powai, Mumbai 400076, India. Email. apala.majumdar@strath.ac.uk}        
	$\;\;$ Neela Nataraj\footnote{Department of Mathematics, Indian Institute of Technology Bombay, Powai, Mumbai 400076, India. Email. neela@math.iitb.ac.in}
}

\usepackage{amsmath,amsthm,amssymb,enumerate, amsbsy}
\usepackage[T1]{fontenc}
\usepackage{gensymb}
\usepackage[sc]{mathpazo}
\usepackage{multirow}
\usepackage{newtxtext,newtxmath}
\usepackage{subfig}
\usepackage{graphicx}
\usepackage{epstopdf}
\usepackage{hyperref}
\usepackage{cancel}
\usepackage[margin=3cm]{geometry}
\usepackage{tikz}
\usepackage[titletoc]{appendix}
\usepackage{caption}
\usetikzlibrary{shapes,calc}
\usepackage{verbatim}
\usepackage{mathrsfs}
\usepackage{accents}
\usepackage[utf8]{inputenc}
\usepackage{float}
\restylefloat{table}

\usetikzlibrary{decorations.pathmorphing}
\usetikzlibrary{decorations.pathreplacing}
\usetikzlibrary{positioning}
\usetikzlibrary{shapes}
\usetikzlibrary{arrows}
\usetikzlibrary{patterns}
\usetikzlibrary{fadings}
\usetikzlibrary{plotmarks}
\usetikzlibrary{calc}
\usetikzlibrary{intersections}
\tikzstyle{every picture}+=[font=\footnotesize]
\usepackage{paralist}
\usepackage{bbm}
\usepackage{latexsym}           
\usepackage{enumerate}
\usepackage{enumitem}

\setlist{noitemsep, topsep=0.8ex, partopsep=0pt
	, leftmargin=3em}
\setlist[1]{labelindent=\parindent}

\newlist{axioms}{enumerate}{1}
\setlist[axioms]{font=\bfseries}

\newlist{alphenum}{enumerate}{1}
\setlist[alphenum]{label=\textbf{(\alph*)}, leftmargin=4em}

\newlist{alphienum}{enumerate}{1}
\setlist[alphienum]{label=\textit{(\alph*)}}

\newlist{romanenum}{enumerate}{1}
\setlist[romanenum]{label=\textit{(\roman*)}}

\newlist{romaninenum}{enumerate*}{1}
\setlist[romaninenum]{label=\textit{(\roman*)}}
\usepackage[vlined]{algorithm2e}
\SetKwIF{If}{ElseIf}{Else}{if}{}{else if}{else}{endif}
\SetKwFor{For}{for}{}{endfor}
\usepackage[noabbrev, capitalise]{cleveref}
\usepackage{tabu}
\crefname{equation}{\unskip}{\unskip}
\creflabelformat{equation}{#2(#1)#3}
\newtheorem{thm}{Theorem}[section]

\newtheorem{lem}[thm]{Lemma}

\theoremstyle{definition}

\theoremstyle{remark}

\newtheorem{rem}[thm]{Remark}

\numberwithin{equation}{section}

\newcommand{\e}{\mathbf{e}}

\newcommand{\V}{\mathbf{V}}
\newcommand{\X}{\mathbf{X}}
\newcommand{\h}{\mathbf{H}}

\newcommand{\abs}[1]{\vert #1\vert}
\newcommand{\dx}{\,{\rm dx}}
\newcommand{\dt}{\,{\rm dt}}
\newcommand{\ds}{\,{\rm ds}}

\newcommand{\E}{\textrm{E}}

\newcommand{\Ichi}{\textrm{I}_{h}\boldsymbol{\chi}}

\newcommand{\Ihpsi}{\textrm{I}_{h} \Psi}
\newcommand{\dg}{{\rm dG}}
\newcommand{\norm}[1]{{\vert\kern-0.25ex\vert #1
		\vert\kern-0.25ex \vert}}
\newcommand{\vertiii}[1]{{\vert\kern-0.25ex\vert\kern-0.25ex\vert #1
		\vert\kern-0.25ex \vert\kern-0.25ex \vert}}
\newcommand{\vertiiih}[1]{{\vert\kern-0.25ex \vert\kern-0.25ex \vert #1
		\vert\kern-0.25ex \vert\kern-0.25ex \vert}_{h}}
\newcommand{\vertiiidg}[1]{{\vert\kern-0.25ex \vert\kern-0.25ex \vert #1
		\vert\kern-0.25ex \vert\kern-0.25ex\vert}_{{\rm dG}}}

\newcommand{\dual}[1]{\langle #1 \rangle}
\newcommand{\duall}[1]{\bigg \langle #1 \bigg \rangle}
\theoremstyle{plain}

\begin{document}
	\maketitle
	\begin{abstract}
		We study a system of semi-linear elliptic partial differential equations with a lower order cubic nonlinear term, and inhomogeneous Dirichlet boundary conditions, relevant for two-dimensional bistable liquid crystal devices, within a reduced Landau-de Gennes framework.
		The main results are (i) a priori error estimates for the energy norm, within the Nitsche's and discontinuous Galerkin frameworks under milder regularity assumptions on the exact solution
		and (ii) a reliable and efficient {\it a posteriori} analysis for a sufficiently large penalization parameter and a sufficiently fine triangulation in both cases. Numerical examples that validate the theoretical results, are presented separately.
	\end{abstract}
	\section{Introduction}
	This paper focuses on the numerical analysis of a system of two second order semi-linear elliptic partial differential equations, with a lower order cubic non-linearity, defined on bounded two-dimensional domains with Lipschitz boundaries and inhomogeneous boundary conditions. Such systems arise naturally in different contexts for two-dimensional systems, our primary motivation being two-dimensional liquid crystal systems \cite{dg}. Liquid crystals are intermediate phases of matter between the conventional solid and liquid states of matter with versatile properties of both phases. Nematic liquid crystals are amongst the most commonly used liquid crystals, for which the constituent rod-like molecules translate freely but exhibit locally preferred directions of orientational ordering, and these locally distinguished directions are referred to as \emph{nematic directors} \cite{dg}. Consequently, nematics are directional or anisotropic materials with direction-dependent physical properties. The Landau-de Gennes (LdG)  theory is perhaps the most celebrated continuum theory for nematic liquid crystals, that describes the nematic state by the LdG order parameter -  $\textit{Q}$-tensor order parameter: a symmetric traceless $3\times 3$ matrix that contains information about the nematic directors and the degree of orientational ordering, within the matrix eigenvectors and eigenvalues, respectively \cite{Majumdar2010}. Reduced two-dimensional LdG approaches have been rigorously derived for two-dimensional domains \cite{golovaty2015, wang_canevari_majumdar_SIAM2019}, for certain model situations. In the reduced case, the order parameter is a symmetric, traceless $2\times 2$ matrix, with simply two independent components, $u$ and $v$. More precisely, if $\mathbf{n} = \left( \cos \theta, \sin \theta \right)$ is the nematic director in the plane, where $\theta \in \left[0, 2\pi \right)$ and $s$ is a scalar order parameter that measures the degree of orientational order, then $u = s \cos 2 \theta$ and v =$s \sin 2 \theta$.
	\medskip
	\noindent Define the two-dimensional vector, $\Psi:=\left(u, v \right)$ on an open, bounded domain, $\Omega \subset \mathbb{R}^2$, with a polygonal boundary, $\partial\Omega$. Then for Dirichlet boundary conditions and in the absence of external fields, the dimensionless reduced LdG free energy  \cite{MultistabilityApalachong} is
	\begin{align} \label{Landau-de Gennes energy functional}
	\mathcal{E}(\Psi_\epsilon) =\int_\Omega (\abs{\nabla \Psi_\epsilon}^2 +
	\epsilon^{-2}(\abs{\Psi_\epsilon}^2  - 1)^2) \dx,
	\end{align}
	where $\Psi_\epsilon = \mathbf{g} $ on $\partial \Omega$, and $\epsilon$ is a material-dependent parameter that depends on the elastic constant, domain size and temperature. 
	Informally speaking, the limit $\epsilon \to 0$ corresponds to macroscopic domains with size much greater than material-dependent characteristic nematic correlation lengths \cite{Majumdar2009EquilibriumOP}. The experimentally observable and physically relevant states are local or global minimizers of the reduced energy, which are weak solutions, $\Psi_\epsilon \in \mathbf{H}^1(\Omega):=H^1(\Omega) \times H^1(\Omega) $,  of the corresponding Euler-Lagrange equations: \begin{align}  \label{continuous nonlinear strong form}
	-\Delta \Psi_\epsilon =2\epsilon^{-2}(1-\abs{\Psi_\epsilon}^2)\Psi_\epsilon \text{ in } \Omega \,\, \text{ and }\,\, \Psi_\epsilon = \mathbf{g} \text{ on } \partial \Omega.
	\end{align}
	In what follows, we work with fixed but small values of $\epsilon$, which describe large domains \cite{henaomajumdarpisante2017}.
	The non-linear system \eqref{continuous nonlinear strong form} is a Poisson-type equation for a two-dimensional vector with a non-linear (cubic) lower order term. In fact, these equations are the celebrated Ginzburg-Landau partial differential equations with a rescaled $\epsilon$, which have been extensively studied in \cite{Bethuel, pacard2000linear}. 
	In this paper, we apply Nitsche's finite-element approximation method to this system \eqref{continuous nonlinear strong form} and
	our main contribution is to relax regularity assumptions on $\Psi_\epsilon$, as will be explained below. 
	The Nitsche's method is well-studied in the literature; see \cite{Gudi_2010_A_new_error_analysis}, \cite{Stenberg2008}, \cite{Strenburg2018},  \cite{Nitsche1971} for applications of Nitsche's method to Poisson's equation with different boundary conditions, {\it a priori} and {\it a posteriori} error analysis for such problems and applications of medius analysis. Further results for the {\it a posteriori} analysis of the Poisson problem are also given in \cite{Braess_Verfurth_a_posteriori} with a {\it saturation assumption}, that can be stated as the approximate solution of the problem in a finer mesh constitutes a better approximation to the exact solution than the approximate solution on a coarser mesh, in the energy norm.
	In \cite{BCD2004} and \cite{KimKwang_nonlinear}, the authors derive
	\emph{ a posteriori} error bounds for the finite-element analysis of the Poisson's equation with $C^0$- Dirichlet boundary conditions, without the saturation assumption. Further, {\it  a priori} and {\it a posteriori} error analysis of dGFEM (Discontinuous Galerkin Finite Element Methods) for the von
	{K}\'{a}rm\'{a}n equations are studied in \cite{Gaurang_apriori_aposteriori}. The references are not exhaustive and the techniques in these papers are adapted to deal with the novel aspects of our problem, as will be outlined below.
	\medskip
	
	The specific model \eqref{continuous nonlinear strong form} has been applied with success to the planar bistable nematic device reported in \cite{Tsakonas}, where the authors study nematic liquid crystals-filled shallow square wells, with experimentally imposed tangent boundary conditions. They report the existence of six distinct stable solutions: two diagonal and four rotated solutions. The nematic director is aligned along the square diagonals in the diagonal solutions, and rotates by $\pi$ radians for the rotated solutions. In \cite{MultistabilityApalachong}, the authors study the numerical convergence of the diagonal and rotated solutions, in a conforming finite-element set-up, as a function of $\epsilon$. In \cite{DGFEM}, the authors carry out a rigorous {\it a priori} error analysis for the dGFEM approximation of  $\h^2$-regular solutions of \eqref{continuous nonlinear strong form} in {\it convex polygonal} domains. Optimal linear (resp. quadratic) order of convergence  in energy (resp. $\mathbf{L}^2$) norm  for solutions, $\Psi_\epsilon \in \h^2(\Omega)$, accompanied by an analysis of the $h-\epsilon$ dependency, where $h$ is the mesh size or the discretization parameter, along with some numerical experiments are discussed. However, there are no {\it a posteriori} error estimates in \cite{DGFEM}.
	\noindent This paper builds on the results in \cite{DGFEM} with several non-trivial generalisations and extensions. Define the admissible space
	$\boldsymbol{\mathcal{{X}}}=\left \{ \mathbf{w} \in \mathbf{H}^1(\Omega) :
	\mathbf{w}= \mathbf{g} \,\, \text{on} \,\, \partial \Omega  \right \}$ and we restrict attention to solutions $\Psi_{\epsilon} \! \in \! \boldsymbol{\mathcal{{X}}}\cap \mathbf{H}^{1+\alpha}(\Omega)$ for $0<\alpha \leq 1$, where $\alpha$ is the index of elliptic regularity in this manuscript. The main contributions of this article are summarised below:
	\begin{itemize}
		\item \noindent an {\it a priori} finite element error analysis for  \eqref{continuous nonlinear strong form},  using the Nitsche's method to incorporate the non-homogeneous boundary conditions,  along with a proof of $h^\alpha$ (resp. $h^{2\alpha}$)-convergence of the energy (resp. $\mathbf{L}^2$) norm where $h$ is the discretization parameter;
		\item   a reliable and efficient residual type {\it a posteriori} error estimate for \eqref{continuous nonlinear strong form} with the assumption that the boundary function, $\mathbf{g} \in \h^{\frac{1}{2}}(\partial\Omega)$ belongs to $\mathbf{C}^0(\overline{\partial\Omega})$;
	\item  {\it a priori} and {\it a posteriori} error estimates for dGFEM under the mild regularity assumption, $\Psi_{\epsilon} \in \boldsymbol{\mathcal{{X}}}\cap \mathbf{H}^{1+\alpha}(\Omega)$ for $0<\alpha \leq 1$;
		\item numerical experiments for uniform as well as adaptive refinement that  validate the theoretical estimates above. 
	\end{itemize}
	There are new technical challenges in this manuscript, compared to  \cite{DGFEM}. The analysis in \cite{DGFEM} is restricted to solutions in $\mathbf{H}^2(\Omega)$, or equivalently $\alpha=1$. The first challenge with a less regular solution, $\Psi_{\epsilon}$ as above, is to handle the normal derivatives $\nabla \Psi_{\epsilon} \nu \notin \mathbf{L}^2(E)$ across the element boundaries. In this paper, the medius analysis \cite{Gudi_2010_A_new_error_analysis} that  combines ideas of both {\it a posteriori}  and {\it a priori} analysis is employed to overcome this. New local efficiency results are proved to establish the stability of a perturbed bilinear form. For the {\it a posteriori} analysis, a lifting operator $\Psi_\mathbf{g}$ is used such that $\Psi_\mathbf{g} = \mathbf{g}$ on $\partial \Omega$, and this technique requires the additional continuity constraints on $\mathbf{g}$. These continuity constraints might be relaxed by using saturation techniques, and will be pursued in future work. The generalisations to dGFEM with $\alpha <1$, involve additional jump and average terms, due to the lack of inter-element continuity in the dGFEM discrete space. For $\alpha=1$, there are $\epsilon$-independent estimates for the $H^2$-norm of solutions, $\Psi_\epsilon$ (see\cite{Brezis_Bethual_Book}) which allows for a $h-\epsilon$ dependency study in \cite{DGFEM}; this is not easily possible for $\alpha<1$ and hence, the value of $\epsilon$ is fixed in this manuscript, unlike the study in \cite{DGFEM}.  
	\medskip
	
	\noindent The reduced regularity assumption for the exact solution is relevant for non-convex polygons, domains with re-entrant corners or slit edges,  with $\alpha<1$ \cite{grisvard1992}. Further, {\it a posteriori} error estimators provide a systematic way of controlling errors for adaptive mesh refinements \cite{Oden,Verfurth}, as is illustrated by means of several numerical experiments in Section~\ref{Numerical results for a priori error analysis}. 
	The numerical estimates confirm the \emph{a priori} and
	\emph{a posteriori} estimates and establishes the advantages of adaptive mesh refinements in terms of computational cost and rates of convergence, captured by informative convergence plots for the estimators in Section \ref{Numerical results for a priori error analysis}.
	\medskip
	
	The standard notation for Sobolev spaces $H^s(\Omega)$ $(\text{resp.} \,W^{s,p}(\Omega))$ with $s,p$ positive real numbers, equipped with the usual norm $\norm{\cdot}_s$ $(\text{resp. } \norm{\cdot}_{s,p} )$ is used throughout the paper and the space $\mathbf{H}^s(\Omega)$ (resp.\,$\mathbf{L}^p(\Omega)$) is defined to be the product space $H^s(\Omega) \times H^s(\Omega)$ $(\text{resp. } L^p(\Omega) \times L^p(\Omega))$ equipped with the corresponding norm $\vertiii{\cdot}_s$ ($\text{resp. }\vertiii{\cdot}_{s,p}$) defined by $\vertiii{\Phi}_s\!=(\norm{\varphi_1}_s^2+\norm{\varphi_2}_s^2)^{\frac{1}{2}}$ for all $ \Phi\!=\!(\varphi_1, \varphi_2) \in \!\mathbf{H}^s(\Omega) $ $ ( \text{resp.  }\vertiii{\Phi}_{s,p}\!\!=(\norm{\varphi_1}_{s,p}^2+\norm{\varphi_2}_{s,p}^2)^{\frac{1}{2}}$ for all $\Phi\!=\!(\varphi_1, \varphi_2) \!\in \!\mathbf{W}^{s,p}(\Omega)).$
  	 The norm on $\mathbf{L}^2(\Omega)$ space is simply $\vertiii{\Phi}_0\!=(\norm{\varphi_1}_0^2+\norm{\varphi_2}_0^2)^{\frac{1}{2}}$ for all $ \Phi\!=\!(\varphi_1, \varphi_2) \in \!\mathbf{L}^2(\Omega). $ 
	 Set $V:=H_0^1(\Omega)= \big\{\varphi \in L^2(\Omega): \frac{\partial \varphi}{\partial x}, \frac{\partial \varphi}{\partial y}
	\in L^2(\Omega) ,$ $ \varphi|_{\partial {\Omega}}=0 \big\} $  and $ \V\!= \,\!\mathbf{H}_0^1(\Omega)=H_0^1(\Omega) \times H_0^1(\Omega)$.
 Throughout the manuscript, $C_s$ denotes a generic constant.
	\medskip
	
	The paper is organized as follows. In the next section, the weak formulation and the Nitsche's method are introduced. Section \ref{Main results} is devoted to the main results for both {\it a priori} and {\it a posteriori} error analysis for Nitsche's method. Section \ref{A priori error estimate} contains some auxiliary results followed by the rigorous {\it a priori} error estimates for Nitsche's method. 
	In Section \ref{A posteriori error estimate}, a reliable and efficient  {\it a posteriori}  error analysis for Nitsche's method is presented. 
Section~\ref{Extension to dG method}  focuses on the the generalisations to dGFEM and is followed by numerical experiments that confirm the theoretical findings in Section \ref{Numerical results for a priori error analysis}. Section~\ref{sec:conclusions} concludes with some brief perspectives. The proofs of the local efficiency results are given in the Appendix.

	\section{Preliminaries and main results} \label{Preliminaries and main results}
	The weak formulation of  the non-linear system \eqref{continuous nonlinear strong form},  the Nitsche's method and the main results in the Nitsche framework are given in this section.
	\subsection{Weak formulation}\label{weak formulation}
	The weak formulation of \eqref{continuous nonlinear strong form} seeks $\Psi_\epsilon \in \boldsymbol{\mathcal{{X}}}$ such that for all  $\Phi \in \V$,
	\begin{align} \label{continuous nonlinear}
	N(\Psi_\epsilon;\Phi):=A(\Psi_\epsilon,\Phi)+B(\Psi_\epsilon,\Psi_\epsilon,\Psi_\epsilon,\Phi)+C(\Psi_\epsilon,\Phi)=0. \,\,\,\, 
	\end{align}
	Here for all $ \Xi=(\xi_1,\xi_2), \boldsymbol \eta= (\eta_1, \eta_2), \Theta=(\theta_1,\theta_2), \Phi=(\varphi_1,\varphi_2)  \in  \X:=\mathbf{H}^1(\Omega) ,$ 
	\begin{align*}
	&A(\Theta,\Phi):=a(\theta_1,\varphi_1)+ a(\theta_2,\varphi_2),\,\,
	C(\Theta,\varphi):=c(\theta_1,\varphi_1)+ c(\theta_2,\varphi_2),
	\\&
	B(\Xi,\boldsymbol \eta,\Theta,\Phi):=\frac{2}{3\epsilon^{2}}\int_\Omega \left((\Xi \cdot \boldsymbol \eta)(\Theta \cdot \Phi)+2(\Xi \cdot \Theta)(\boldsymbol \eta \cdot \Phi)\right) \dx =\frac{1}{3}(3b(\xi_1,\eta_1,\theta_1,\varphi_1)+3b(\xi_2,\eta_2,\theta_2,\varphi_2) \notag \\& 
	\qquad \qquad \qquad \qquad + 2b(\xi_2,\eta_1,\theta_2,\varphi_1)+ 2 b(\xi_1,\eta_2,\theta_1,\varphi_2)+ b(\xi_2,\eta_2,\theta_1,\varphi_1)+b(\xi_1,\eta_1,\theta_2,\varphi_2)),
	\\&
	\text{and for } \xi,\eta,\theta,\varphi \in H^1(\Omega), 
	\:  a(\theta,\varphi):=\int_{\Omega}  \nabla \theta \cdot \nabla \varphi \dx,  \: b(\xi,\eta,\theta,\varphi):= 2\epsilon^{-2}\int_\Omega \xi\eta\theta \varphi \dx \nonumber \\ 
	& \text{ and }  c(\theta,\varphi):=-2\epsilon^{-2} \int_\Omega \theta \varphi \dx. \notag
	\end{align*}
See \cite{MultistabilityApalachong, DGFEM}	for a proof of existence of minimizers of the Landau-de Gennes energy functional \eqref{Landau-de Gennes energy functional}  that are  solutions to \eqref{continuous nonlinear}.
The analysis of this article is applicable to the cases where the exact solution $\Psi_\epsilon$ of \eqref{continuous nonlinear} belongs to $\boldsymbol{\mathcal{{X}}} \cap \h^{1+\alpha}(\Omega),  0 < \alpha <1,$   for example in non-convex polygons.
	When $\Omega$ is a convex polygon, $\alpha=1$; that is, the solution of \eqref{continuous nonlinear} belongs to $\boldsymbol{\mathcal{{X}}} \cap \h^{2}(\Omega)$. 
	The regular solutions (also referred to as non-singular solutions in literature; see  \cite{Gaurang_apriori_aposteriori}  and the references therein) $\Psi_\epsilon$ of \eqref{continuous nonlinear strong form} for \textbf{a fixed $\epsilon$} are approximated.
	This implies that the linearized operator $\dual{DN(\Psi_\epsilon)\cdot, \cdot}$ is invertible in the Banach space and the following inf-sup conditions \cite{Ern} hold: 
	\begin{align}\label{2.9}
	0< \beta := \inf_{\substack{\Theta \in \V \\  \vertiii{\Theta}_1=1}} \sup_{\substack{\Phi \in \V \\  \vertiii{\Phi}_1=1}}\dual{DN(\Psi_\epsilon)\Theta, \Phi}, \text{ and } 0< \beta = \inf_{\substack{\Phi \in \V \\  \vertiii{\Phi}_1=1}} \sup_{\substack{\Theta \in \V \\  \vertiii{\Theta}_1=1}}\dual{DN(\Psi_\epsilon)\Theta, \Phi},
	\end{align}
	where $\dual{DN(\Psi_\epsilon)\Theta, \Phi}:=A(\Theta,\Phi)+3B(\Psi_\epsilon,\Psi_\epsilon,\Theta,\Phi)+C(\Theta,\Phi)$ and  the inf-sup constant $\beta$ depends on $\epsilon$. Here and throughout the paper,  $\dual{\cdot, \cdot}$  denotes the duality pairing between $\V^*$ and $\V$.
	The parameter $\epsilon$ in $\Psi_\epsilon$ is suppressed  for notational brevity and is chosen fixed in the sequel. 
	\subsection{Nitsche's method}
	 Consider a shape-regular triangulation  $\mathcal{T}$ of $\Omega$ into triangles \cite{ciarlet}. Define the mesh discretization parameter $h= \max_{T \in \mathcal{T}} h_T,$ where $h_T= diam(T) $. Denote $\mathcal{E}_h^{i} $( resp. $  \mathcal{E}_h^{\partial}$) to be the interior (resp. boundary) edges of $\mathcal{T}$ and let $\mathcal{E}:=\mathcal{E}_h^{i} \cup  \mathcal{E}_h^{\partial}$. The length of an edge $E$ is denoted by $h_E.$ Define the finite element subspace of $\X$ by $\X_h:=X_h \times X_h$ with $X_h:= \{v \in H^1(\Omega)|\,\, v|_T \in P_1(T) \text{ for all } T \in \mathcal{T} \}$ and let the discrete norm be defined by $\displaystyle 	\norm{v}^2_h:=\int_{\Omega} \abs{\nabla v}^2 \dx + \sum_{ E \in \mathcal{E}_h^{\partial}} \frac{\sigma}{h_E} \int_{E}  v^2 \ds \text{ for all } v \in X_h.$ 
	Here $\sigma>0$ is the penalty parameter and $P_1(T)$ denotes affine polynomials defined on $T$. The space $\X_h $ is equipped with the product norm $\vertiiih{\Phi_{h}}:= \left( \norm{\varphi_{1}}_h^2 + \norm{\varphi_{2}}_h^2\right)^{1/2} $ for all $\Phi_{h}=(\varphi_{1},\varphi_{2}) \in \X_h$.  Define $\vertiii{\Phi_{h}}_{0, E}^2:= \norm{\varphi_{1}}_{0, E}^2 + \norm{\varphi_{2}}_{0, E}^2 $ and $\vertiii{\Phi_{h}}_{0, T}^2:= \norm{\varphi_{1}}_{0, T}^2 + \norm{\varphi_{2}}_{0, T}^2 $ for  $\Phi_{h} \in \X_h$  such that for $v \in X_h,$ $\norm{v}_{0, E}^2:= \int_{ E} v^2 \ds$ and $\norm{v}_{0, T}^2:= \int_{ T} v^2 \dx$, respectively. 
	 Define $H^1(\mathcal{T}):= \{v \in L^2(\Omega)| \, v \in  H^1(T) \text{ for all } T \in \mathcal{T}\}$ and $\mathbf{H}^1(\mathcal{T}):=H^1(\mathcal{T}) \times H^1(\mathcal{T}). $ For an interior edge $E$ shared by the triangles $T^+$ and $T^-$, define the jump  and average of  $\varphi \in H^1(\mathcal{T})$ across $E$ as $[\varphi]_E:= \varphi|_{T^+} - \varphi|_{T^-}$ and  $\{\varphi\}_E:=\frac{1}{2}(\varphi|_{T^+} + \varphi|_{T^-} )$, respectively, and for an   boundary edge $E $ of the triangle $T$, $[\varphi]_E:= \varphi|_{T} $ and  $\{\varphi\}_E:=\varphi|_{T}$, respectively. For a vector function, jump and average are defined component-wise. 
	For  $\theta, \varphi \in H^1(\Omega)$, $\mathbf{g}=(g_1, g_2)$ and  the penalty parameter $\sigma>0$, let
	 \begin{align*}
	&a_h(\theta,\varphi):= \int_\Omega \nabla \theta \cdot \nabla \varphi \dx- \dual{\frac{\partial \theta}{\partial \nu},  \varphi  }_{\partial \Omega} -  \dual{ \theta,  \frac{\partial \varphi}{\partial \nu}  }_{\partial \Omega} + \sum_{E \in  \mathcal{E}_h^{\partial}}  \frac{\sigma}{h_E} \dual{\theta, \varphi }_E,\notag \\	
&	\text{  and }\,\,
	l^i_h(\varphi):=
	- \dual{  g_i,  \frac{\partial \varphi}{\partial \nu}  }_{\partial \Omega} + \sum_{E \in  \mathcal{E}_h^{\partial}} \frac{\sigma}{h_E} \dual{g_i, \varphi }_E \text{ for } 1\leq i \leq 2,
	\end{align*}	
	\noindent where $\dual{\cdot, \cdot}_{\partial \Omega}$ denotes the duality pairing between ${H}^{-\frac{1}{2}}(\partial \Omega)$ and ${H}^{\frac{1}{2}}(\partial \Omega)$ and $\nu$ denotes  the outward unit normal  associated to $\partial \Omega$. In the sequel, $\dual{\cdot, \cdot}_E$  is the duality pairing between 
	 ${H}^{-\frac{1}{2}}(E)$ and ${H}^{\frac{1}{2}}(E)$ for $E \in \mathcal{E}.$ For $ \Theta=(\theta_1,\theta_2),\, \Phi=(\varphi_1,\varphi_2)
	\in \X $, let $A_{h}(\Theta,\Phi) :=a_{h}(\theta_1,\varphi_1)+a_{h}(\theta_2,\varphi_2),$
		and 
	  $L_h(\Phi_{h})=l^1_h(\varphi_1)+l^2_h(\varphi_2)$.	
	 
	 \noindent The  Nitsche's method corresponding to \eqref{continuous nonlinear strong form} seeks $ \Psi_{h}\!\in\!\X_h$, such that for all $ \Phi_{h} \in  \X_h,$
	\begin{align}\label{discrete nonlinear problem}
	N_h(\Psi_h;\Phi_h):=A_{h}(\Psi_{h},\Phi_{h})+B(\Psi_{h},\Psi_{h},\Psi_{h},\Phi_{h})+C(\Psi_{h},\Phi_{h})-L_h(\Phi_{h})=0. 
	\end{align}
	\begin{rem}
		The restrictions of the bilinear and quadrilinear forms $C(\cdot, \cdot),$ $B(\cdot, \cdot,\cdot, \cdot)$ to $T \in \mathcal{T}$ are denoted as $C_T(\cdot, \cdot), $ $B_T(\cdot, \cdot,\cdot, \cdot)$, respectively. Define the bilinear form $A_T(\Theta, \Phi):= \int_T \nabla \Theta \cdot \nabla \Phi \dx$ for all $\Theta, \Phi \in \X$.
 For $\Phi_h=(\varphi_1,\varphi_2) \in \X_{h}$, let  $\nabla\Phi_h\nu_E: =(\frac{\partial \varphi_1}{\partial \nu_E}, \frac{\partial \varphi_2}{\partial \nu_E})$ on an edge $E$ with  outward unit normal $\nu_E$ to $E$.
		\end{rem}
	\subsection{Main results} \label{Main results}
 The main results in this manuscript for Nitsche's method are stated in this sub-section. Theorems \ref{energy and $L^2$ norm error estimate} and \ref{reliability and efficiency} establish  {\it a priori} error estimates 
 in energy and $\mathbf{L}^2(\Omega)$ norms,  and {\it a posteriori} error estimates for {\it Nitsche's method}, respectively, when the exact solution $\Psi$ of \eqref{continuous nonlinear} has the regularity $\boldsymbol{\mathcal{{X}}} \cap \h^{1+\alpha}(\Omega) , \; 0< \alpha \le 1$. Throughout the sequel, $0 < \alpha \le 1$ denotes the index of elliptic regularity. The results are extended for dGFEM  and are presented in Section \ref{Extension to dG method}.
	\begin{thm}({\it A priori} error estimate)\label{energy and $L^2$ norm error estimate}
		Let $\Psi$ be a regular solution of \eqref{continuous nonlinear}. 
		For a sufficiently large  penalty parameter $\sigma>0$, and a sufficiently small  discretization parameter $h$, there exists a unique solution $\Psi_{h}$ to the discrete problem \eqref{discrete nonlinear problem}  that approximates $\Psi$ such that	
		 	\begin{align*}
			(i)\,\,\,	\vertiii{\Psi-\Psi_h}_h \lesssim h^{\alpha},\quad
			(ii)\,\,\,	\vertiii{\Psi-\Psi_h}_0 \lesssim h^{2\alpha},
			\end{align*}
		where $0 < \alpha \le 1$ denotes the index of elliptic regularity. 	As per standard convention  $a \lesssim b$ $\iff$ $a \leq Cb$ where the constant $C$ is independent of the discretization parameter $h$.
	\end{thm}
\noindent 	A reliable and efficient {\it a posteriori} error estimate for \eqref{discrete nonlinear problem} is the second main result of the paper. For each element $T \!\in\! \mathcal{T}$ and edge $E \!\in\! \mathcal{E}$, define the volume and edge contributions to the estimators  by
\begin{align}
&\vartheta_T^2: = h_T^2 \vertiii{ 2\epsilon^{-2}  (\abs{\Psi_h}^2-1)\Psi_h}^2_{0,T}, \,\,\,(\vartheta_E^i)^2 := h_E \vertiii{[\nabla \Psi_h \nu_E]_E}_{0,E}^2 \text{ for all } E \in \mathcal{E}_h^{i}, \label{estimator_volume_interior_edge} \\& \text{and } (\vartheta_E^{\partial })^2:= \frac{1}{h_E} \vertiii{\Psi_{h} - \mathbf{g}}_{0,E}^2 \text{ for all } E \in \mathcal{E}_h^{\partial}. \label{estimator_boundary_edge}
\end{align}
Define the estimator
$ \displaystyle
\vartheta^2:= \sum_{ T \in \mathcal{T}}  \vartheta_T^2 +\sum_{ E \in \mathcal{E}_h^{i}} (\vartheta_E^i)^2 + \sum_{ E \in \mathcal{E}_h^{\partial} }(\vartheta_E^{\partial})^2.
$
\begin{thm}{(A posteriori error estimate)}\label{reliability and efficiency}
	Let $\Psi$ be a regular solution of \eqref{continuous nonlinear} and $\Psi_{h}$ solve \eqref{discrete nonlinear problem}. For a sufficiently large penalty parameter $\sigma>0$, and a sufficiently small  discretization parameter $h$,  there exist $h$-independent positive constants $C_{\text{rel}}$ and $C_{\text{eff}}$ such that
	\begin{align*}
	C_{\text{eff}}	\vartheta \leq	\vertiii{\Psi - \Psi_{h}}_{h} \leq  C_{\text{rel}} \big( \vartheta+  h.o.t \big),
	\end{align*}  
	where $h.o.t$ expresses one or several terms of higher order (as will be explained in Section \ref{A posteriori error estimate}).	
\end{thm}

	\section{A priori error estimate}\label{A priori error estimate}
This section is devoted to the proof of Theorem \ref{energy and $L^2$ norm error estimate}. Some auxiliary results are presented first. This is followed by a discrete inf-sup condition and the construction of a non-linear map for the application of Brouwer's fixed point theorem. The energy and  $\mathbf{L}^2$- norm estimates  follow as a consequence of the fixed point and duality  arguments. 
	\subsection{Auxiliary results}\label{auxiliary results}
		\begin{lem}{(Poincar\'e type inequalities)}\cite{Poincare_Suli}  \label{Poincare type inequality}
		Let $\Omega$ be a bounded open subset of $\mathbb{R}^2$ with Lipschitz continuous boundary $\partial \Omega$. 
		\begin{enumerate}
			\item For $\varphi \in H^1_0(\Omega)$, there exists a positive constant $\alpha_0 =\alpha_0(\Omega)$ such that $\alpha_0 \norm{\varphi}_{0} \leq \norm{\nabla\varphi}_{0} .$
			\item 	For $\varphi \in H^1(\mathcal{T})$, there exists a constant $C_P>0$ independent of $h $ and $\varphi$ such that for $1 \leq r < \infty$,	$\norm{\varphi}_{L^r(\Omega)} \leq C_P \norm{\varphi}_{h}.$
		\end{enumerate}		
	\end{lem}	
\begin{lem}[{\it Trace inequalities}]\cite{Ern} \label{discrete trace inequality} 
	$\displaystyle (i)  \; \text{For  }{\rm v}\!\in\! H^1(T), T \!\in\! \mathcal{T}, \norm{{\rm v}}_{0,\partial T}^2 \lesssim (h_T^{-1} \norm{{\rm v}}_{0,T}^2 + \norm{{\rm v}}_{0,T}\norm{\nabla {\rm v}}_{0,T}).$ $\displaystyle (ii)  \; 
	\text{For all }\Phi_{h}\!\in\! \X_{h}, 
	\sum_{E \in \mathcal{E}_h^{\partial}} h_E \vertiii{\nabla\Phi_{h}\nu_E}_{0,E}^2 \lesssim \vertiii{\nabla\Phi_{h}}_{0}^2. $
\end{lem}
	\begin{lem}\label{Interpolation estimate}{(Interpolation estimate)}\cite{ciarlet}
			For ${\rm v} \!\in\! H^{1+\alpha}(\Omega) \text{ with } \alpha \!\in\! (0,1]$, there exists  ${\rm{I}}_{h}{\rm v} \!\in\! X_h$ such that
		\begin{align*}
	\norm{{\rm v}-{\rm{I}}_{h}{\rm v} }_{0} +h	\norm{{\rm v}-{\rm{I}}_{h}{\rm v} }_{1} \leq C_I h^{1+\alpha}  \abs{{\rm v}}_{H^{1+ \alpha}(\Omega)},
		\end{align*}	
	where $C_I$  is a positive constant independent of $h$.
	\end{lem}
\begin{rem} \label{remark interpolation}
	Trace inequality stated in Lemma \ref{discrete trace inequality}$(i)$ yields $	\norm{{\rm v}-{\rm{I}}_{h}{\rm v} }_{h} \leq C_I h ^{\alpha} \abs{{\rm v}}_{H^{1+\alpha}(\Omega)}$ for a positive constant $C_I$ independent of $h$.
\end{rem}
\begin{lem}({Extension operator})\label{enrichment Nitsche}
	\cite{ Poincare_Brenner,Karakashian,Strenburg2018} Define the operator $\Pi_h: X_h \rightarrow V_h:=X_h \cap H^1_0(\Omega) $ using nodal values of freedom:
	$\displaystyle \Pi_h{\rm v}(n) = \begin{cases} 
	0 \quad  \quad \text{ for a node  } n \text{ on } \partial \Omega \\
	{\rm v}(n) \quad  \quad \text{ for a node  } n \text{ on } \Omega \setminus \partial \Omega
	\end{cases}$.
	 For all ${\rm v}\in X_h$, it holds that
	\begin{align}
	\big(\sum_{T \in \mathcal{T}} h_T^{-2} \norm{{\rm v}-\Pi_h {\rm v}}^2_{0,T} + \sum_{ E \in \mathcal{E}_h^{i}} h_E^{-1}  \norm{{\rm v}-\Pi_h {\rm v}}^2_{0,E}\big)^{\frac{1}{2}} \leq C_{e_1} \norm{{\rm v}}_h \label{enrichment2},  \\
	\norm{{\rm v}-\Pi_h {\rm v}}_h \leq  C_{e_2}(\sum_{ E \in \mathcal{E}^{\partial}_h}h_E^{-1}\int_{ E}{\rm v}^2{ \rm ds} )^{\frac{1}{2}} \leq C_{e_2} \norm{{\rm v}}_h, \quad \norm{\Pi_h {\rm v}}_h \leq C_{e_3} \norm{{\rm v}}_h \label{enrichment1},
	\end{align}
	where the constants $C_{e_1}$, $C_{e_2}$ and $C_{e_3} $ are independent of $h$.
\end{lem}
\noindent The next lemma states boundedness and coercivity results for $A(\cdot,\cdot)$ (resp. $A_h(\cdot,\cdot)$), boundedness results for  $B(\cdot,\cdot,\cdot,\cdot)$ and $C(\cdot,\cdot)$.  These results are a consequence of  H\"older's inequality, Lemma \ref{Poincare type inequality}, and the Sobolev embedding results \cite{ciarlet} $H^{1}(\Omega) \hookrightarrow L^{4}(\Omega)$ and $H^{1+\alpha}(\Omega) \hookrightarrow L^{\infty}(\Omega)$ for $\Omega \subset \mathbb{R}^2$ and $\alpha>0$. For detailed proofs, see \cite{DGFEM}.

\begin{lem}\label{2.3.17}  {(Properties of bilinear and quadrilinear forms)}\cite{ DGFEM} \label{boundedness}
	The following properties hold. 
	\begin{enumerate}[label=(\roman*)]
\item For all $\Theta$, $\Phi \!\in\! \V$, $\displaystyle A(\Theta,\Phi)\leq \vertiii{\Theta}_1 \vertiii{\Phi}_1,\text{ and } A(\Theta,\Theta) \gtrsim \vertiii{\Theta}_1^2. $
\item  For the choice of a sufficiently large parameter $\sigma$, there exists a 
	positive constant $C_{\rm ell} > 0$ such that $\text{ for all } \Theta_h, \Phi_{h} \in \X_h$, $\displaystyle A_{h}(\Theta_{h},\Phi_{h})\lesssim \vertiii{\Theta_{h}}_{h} \vertiii{\Phi_{h}}_{h}, \text{ and }  A_{h}(\Phi_{h}, \Phi_{h})  \geq C_{\rm ell} \; \vertiii{\Phi_{h}}_{h}^2.$
\item   For all $\Theta$, $\Phi \!\in\! \X$,
	$C(\Theta,\Phi)\lesssim  \epsilon^{-2} \vertiii{\Theta}_0 \vertiii{\Phi}_0$ and 
	  $C(\Theta,\Phi)\lesssim \epsilon^{-2} \vertiii{\Theta}_{h} \vertiii{\Phi}_{h}.$
\item  	For $\Xi,\boldsymbol \eta,\Theta,\Phi \!\in\! \X$,  (resp. $\Xi , \boldsymbol \eta \!\in\! \mathbf{H}^{1+\alpha}(\Omega)$ with $ 0<\alpha\le 1$, $\Theta$, $\Phi \!\in\!  \X$),
\begin{align*}
	& \int_{\Omega}(\Xi \cdot \boldsymbol \eta)(\Theta \cdot \Phi) {\dx} \lesssim \vertiii{\Xi}_1\vertiii{\boldsymbol \eta}_1\vertiii{\Theta}_1 \vertiii{\Phi}_1  \text{  and } 
	 \int_\Omega(\Xi \cdot \boldsymbol \eta)(\Theta \cdot \Phi) \dx \lesssim \vertiiih{\Xi}\vertiiih{\boldsymbol \eta}\vertiiih{\Theta} \vertiiih{\Phi}, \\ 
	&  B(\Xi,\boldsymbol \eta,\Theta,\Phi)\lesssim \epsilon^{-2} \vertiii{\Xi}_1\vertiii{\boldsymbol \eta}_1 \vertiii{\Theta}_1 \vertiii{\Phi}_1 
	\text{  and } 
	B(\Xi,\boldsymbol \eta,\Theta,\Phi ) \lesssim \epsilon^{-2} \vertiii{\Xi}_{h} \vertiii{\boldsymbol \eta}_{h} \vertiii{\Theta}_{h} \vertiii{\Phi}_{h}. \\
&
(\text{resp. } 
	B(\Xi,\boldsymbol \eta,\Theta,\Phi)\lesssim \epsilon^{-2}  \vertiii{\Xi}_{1+\alpha} \vertiii{\boldsymbol \eta}_{1+\alpha} \vertiii{\Theta}_0 \vertiii{\Phi}_0). 
	\end{align*}
	\item 	For $\boldsymbol \eta \!\in\! \X$ and for all $ \boldsymbol \eta_h,  \Phi_{h} \!\in\! \X_h$ (resp.   ${\boldsymbol \eta} \!\in\!  \h^{1+\alpha}(\Omega) \text{ with } 0<\alpha \le 1$),
	\begin{align*}
	&
	B(\boldsymbol \eta_h,\boldsymbol \eta_h, \boldsymbol \eta_h, \Phi_{h})- B(\boldsymbol \eta,\boldsymbol \eta, \boldsymbol \eta, \Phi_{h})   \lesssim \epsilon^{-2}(\vertiii{\boldsymbol \eta_h-\boldsymbol \eta}^2_h (\vertiii{\boldsymbol \eta_h}_h +  \vertiii{\boldsymbol \eta}_1) + \vertiii{\boldsymbol \eta_h-\boldsymbol \eta}_h \vertiii{\boldsymbol \eta}_{1}^2) \vertiii{\Phi_{h}}_h	\\
	&	(\text{resp. }B({\rm I}_{h}\boldsymbol \eta,{\rm I}_{h}\boldsymbol \eta, {\rm I}_{h}\boldsymbol \eta, \Phi_{h})- B(\boldsymbol \eta,\boldsymbol \eta, \boldsymbol \eta, \Phi_{h})   \lesssim  \epsilon^{-2} h^{2\alpha} \vertiii{\boldsymbol \eta}_{1+\alpha}^3\vertiiih{\Phi_{h}},  \\ 
	&
 B(\boldsymbol \eta,\boldsymbol \eta, \Theta_{h}, \Phi_{h}) - 	B({\rm I}_{h}\boldsymbol \eta,{\rm I}_{h}\boldsymbol \eta, \Theta_{h}, \Phi_{h}) \lesssim \epsilon^{-2} h^\alpha \vertiii{\boldsymbol \eta}_{1+\alpha}^2\vertiiih{\Theta_{h}} \vertiiih{\Phi_{h}}), \notag	
\end{align*}
	\end{enumerate}
		where the hidden constant in $"\lesssim"$ depends on the constants from $C_P,   C_S$ and 
		$\alpha_0$, and  are independent of $h.$
\end{lem}
\noindent We now state the well-posedness and regularity of solutions of a second-order linear system of equations \eqref{2.4.2.1} and a perturbation result \eqref{2.4.3.1.1} that is important to prove the discrete inf-sup condition in the next section. The proof follows analogous to the proof in Theorem 4.7 of \cite{DGFEM} and is skipped.
\begin{lem} (Linearized systems)\label{linear equation1 for discrete inf-sup} Let $\Psi $ be a regular solution of \eqref{continuous nonlinear}. For a given $\Theta_{h} \!\in\! \X_h$ with $\vertiiih{\Theta_{h}}=1$, there exist  $\boldsymbol{\xi}$ and $\boldsymbol{\eta}\!\in\! \mathbf{H}^{1+\alpha}(\Omega)\cap \V $ that solve the linear systems 
	\begin{align} 
	A(\boldsymbol \xi,\Phi)=3B(\Psi,\Psi,\Theta_{h},\Phi)+C(\Theta_{h},\Phi) \,\, \text{ for all } \Phi \in \V \,\,\,
	\text{and}\label{2.4.2.1}\\
	A(\boldsymbol \eta,\Phi)=3B(\Psi,\Psi, \Pi_h\Theta_{h},\Phi)+C(\Pi_h\Theta_{h},\Phi)  \,\, \text{ for all }  \Phi \in \V \,\,\,\,\,\,\,\,\label{2.4.3.1.1}
	\end{align}	
	such that
	\begin{align}\label{2.4.4.0}
	\vertiii{\boldsymbol \xi}_{1+\alpha} \lesssim \epsilon^{-2}(1+\vertiii{\Psi}_{1+\alpha}^2)\quad \text{ and } \quad \vertiii{\nabla (\boldsymbol\eta-\boldsymbol \xi)}_{0}
	\lesssim \epsilon^{-2}h(1+\vertiii{\Psi}_{1+\alpha}^2),
	\end{align}
	where the constant hidden in  $ "\lesssim" $ depends on $C_S$, $\alpha_0$ and $C_{e_1}$.
\end{lem}
The next three lemmas concern  local efficiency type estimates that yield  lower bounds for the errors, and are necessary for the medius analysis. The proofs follow from standard bubble function techniques extended to the non-linear system considered in this paper and is sketched in the Appendix.
\begin{lem} (Local efficiency I)\label{lem 3.1} \label{local efficiency for ball to ball}	Let $\Psi \!\in\! \boldsymbol{\mathcal{{X}}}$ be a regular solution of \eqref{continuous nonlinear}. For $\Phi_h \!\in\! \X_h$,   define  $\boldsymbol{ \eta}_T := (2\epsilon^{-2}(\abs{\Phi_{h}}^2 -1)\Phi_{h})|_T$, where    $T \!\in\! \mathcal{T}$ and $\boldsymbol{ \eta}_E := [\nabla \Phi_h \nu_E]_E$, where   $E$ is an edge  of $T$. Then the following estimates hold.
	\begin{align*}
(i)	\sum_{ T \in \mathcal{T}}	h_T^2 \vertiii{\boldsymbol{ \eta}_T}_{0,T}^2 + \sum_{ E \in \mathcal{E}_h^{i}} h_E  \vertiii{\boldsymbol{ \eta}_E}_{0,E}^2 \lesssim	\vertiii{\Psi - \Phi_{h}}^2_{h} (1+ \epsilon^{-2} (\vertiii{\Psi-\Phi_{h}}_{h} (\vertiii{\Phi_h}_{1}+ \vertiii{\Psi}_{1}) + \vertiii{\Psi}_{1}^2 +1))^2.
	\end{align*}
$(ii)$  For $\Psi \!\in\! \boldsymbol{\mathcal{{X}}}\cap \h^{1+\alpha}(\Omega)$, $0 < \alpha \le 1$, $\Phi_h:= {\rm I}_h \Psi$ in the definitions of $\boldsymbol{ \eta}_T$ and $\boldsymbol{ \eta}_E$,
	\begin{align*} 
	\sum_{ T \in \mathcal{T}}	h_T^2 \vertiii{\boldsymbol{ \eta}_T}_{0,T}^2 + \sum_{ E \in \mathcal{E}_h^{i}} h_E  \vertiii{\boldsymbol{ \eta}_E}_{0,E}^2 	\lesssim h^{2\alpha} (1+ \epsilon^{-2}h^{\alpha}(1+ \vertiii{ \Psi}_{1+\alpha}^2))^2\vertiii{ \Psi}_{1+\alpha}^2.
	\end{align*}
\end{lem}

\noindent {The next lemma is a local efficiency type result for \eqref{2.4.2.1} that helps to prove the discrete inf-sup condition for a linear problem in the next section.}
\begin{lem} (Local efficiency II) \label{local efficiency for discrete infsup}
	Let $\boldsymbol\xi$ be the solution of \eqref{2.4.2.1} with  interpolant ${\rm{I}}_h \boldsymbol\xi \!\in\! \V_h:= \X_h\cap \h^1_0(\Omega)$.  If the exact solution $\Psi \!\in\! \boldsymbol{\mathcal{{X}}}\cap \h^{1+\alpha}(\Omega)$, $0 < \alpha \le 1$, then
	\begin{align*} 
	\sum_{ T \in \mathcal{T}}	h_T^2 \vertiii{\boldsymbol{ \eta}_T}_{0,T}^2 + \sum_{ E \in \mathcal{E}_h^{i}} h_E  \vertiii{\boldsymbol{ \eta}_E}_{0,E}^2 
	\lesssim \epsilon^{-4}h^{2\alpha}(1+\vertiii{\Psi}_{1+\alpha}^2)^2,
	\end{align*}	
	where $\boldsymbol{ \eta}_T:=(2 \epsilon^{-2}(\abs{{\rm{I}}_h \Psi}^2  \Theta_{h} + 2 ({\rm{I}}_h \Psi \cdot  \Theta_{h}) {\rm{I}}_h \Psi -\Theta_{h}))|_T$ is defined on a triangle $T \!\in\! \mathcal{T}$, $ {\rm{I}}_h \Psi\!\in\! \X_h$ is the interpolant of $\Psi$ and $\boldsymbol{ \eta}_E= [\nabla({\rm{I}}_h\boldsymbol{\xi}) \nu_E]_E$ on the edge $E$ of $T$ and $\Theta_h \in \X_h$.
\end{lem}
\noindent For $G \!\in\! \mathbf{L}^2(\Omega)$, the well-posed dual problem admits a unique   $\boldsymbol{\chi} \!\in\! \V$  \cite{DGFEM} such that
\begin{align} \label{dual problem}
\dual{DN(\Psi)\Phi, \boldsymbol{\chi}}= (G, \Phi) \,\,\quad \text{for all } \Phi \in \V,
\end{align}
that satisfies
\begin{align}\label{xi regularity bound}
\quad \vertiii{\boldsymbol{\chi}}_{1+\alpha}\lesssim (1+ \epsilon^{-2}(1+ \vertiii{\Psi}_{1+\alpha}^2)) \vertiii{G}_0,
\end{align}
where $0 < \alpha \le 1$ denotes the index of elliptic regularity.

\medskip
\noindent A local efficiency type result for \eqref{dual problem} is needed to establish  $\mathbf{L}^2 $- norm error estimates and is stated below.
\begin{lem}(Local efficiency III) \label{local efficiency for L2 estimate}
 Let $\Psi $ be a regular solution of \eqref{continuous nonlinear} and ${\rm I}_h\Psi \!\in\! \X_h$ be its interpolant.	For a given $G \!\in\! \mathbf{L}^2(\Omega)$, let $\boldsymbol{\chi}$ solve \eqref{dual problem} and let its interpolant be  ${\rm{I}}_{h}\boldsymbol{\chi} \!\in\! \V_h$. Then, the following result holds. 
	\begin{align*}
	\sum_{ T \in \mathcal{T}}	h_T^2 \vertiii{\boldsymbol{ \eta}_T}_{0,T}^2 + \sum_{ E \in \mathcal{E}_h^{i}} h_E  \vertiii{\boldsymbol{ \eta}_E}_{0,E}^2 
	\lesssim h^{2\alpha}(1+ \epsilon^{-2}(1+\vertiii{\Psi}^2_{1+\alpha}))^4 \vertiii{G}_0^2+ (Osc(G))^2,
	\end{align*}
	where $\boldsymbol{ \eta}_T :=( G  -2\epsilon^{-2} (\abs{{\rm{I}}_{h}\Psi}^2 {\rm{I}}_{h}\boldsymbol{\chi} +2 ({\rm I}_h\Psi \cdot {\rm I}_h{ \boldsymbol{\chi}}){\rm{I}}_{h}\Psi- {\rm{I}}_{h}\boldsymbol{\chi}))|_T$ is defined on a triangle $T \in \mathcal{T} $, $\boldsymbol{ \boldsymbol{ \eta}}_E:=[ \nabla({\rm{I}}_{h}\boldsymbol{\chi}) \nu_E ]_E $ on edge $E$ of $T$ and $\displaystyle 	Osc(G)=\big(\sum_{ T \in \mathcal{T}}h_T^2 (\inf_{G_h \in P_1(T)} \vertiii{G- G_h}_{0,T}^2)\big)^{\frac{1}{2}}.$
\end{lem}
\begin{rem}
In this article, we consider the case when exact solution belongs to $\boldsymbol{\mathcal{{X}}} \cap \h^{1+\alpha}(\Omega), \; 0< \alpha \le 1.$ Hence globally continuous piece-wise affine polynomials in $X_h$ lead to optimal order estimates. However, if the solution belongs to $\h^s(\Omega )$ for $\frac{3}{2} < s \leq p+1, $ $p \in \mathbb{N},$ then choose $X_h= \{v_h \in {C}^0(\overline{\Omega}), v_h|_T \in P_p(T), \text{ for all } T \in \mathcal{T}\}$ \cite{Stenberg2008}. In this case, the local efficiency terms $\boldsymbol{ \eta}_T:=-\Delta \Phi_h+ (2\epsilon^{-2}(\abs{\Phi_{h}}^2 -1)\Phi_{h})|_T$ in Lemma \ref{local efficiency for ball to ball} and $\boldsymbol{\eta}_T$ will include $\Delta (\textrm{I}_h \boldsymbol{\xi})$ (resp. $\Delta (\textrm{I}_h \boldsymbol{\chi})$)  in Lemma \ref{local efficiency for discrete infsup} (resp. Lemma \ref{local efficiency for L2 estimate}).
\end{rem}
\subsection{Proof of  \textbf{\textit a priori} estimates}\label{Proof of main results}
This subsection focuses on the proof of {\it a priori} error estimates in Theorem \ref{energy and $L^2$ norm error estimate}. The key idea is to establish a discrete inf-sup condition that corresponds to a perturbed bilinear form  defined for all $\Theta_{h}, \Phi_{h} \!\in\! \X_h$ as
	\begin{align}\label{perturbed bilinear form}
	\dual{DN_h(\textrm{I}_h\Psi)\Theta_{h}, \Phi_{h}} := A_{h}(\Theta_{h},\Phi_{h})+3B(\textrm{I}_h\Psi,\textrm{I}_h\Psi,\Theta_{h},\Phi_{h})+C(\Theta_{h},\Phi_{h}).
	\end{align} 	
\noindent	 in Theorem \ref{discrete inf sup} when the exact solution $\Psi$ of \eqref{continuous nonlinear} belongs to $\boldsymbol{\mathcal{{X}}} \cap \h^{1+\alpha}(\Omega)$ with $0< \alpha \le 1$.  The  proofs in \cite[Theorem 4.7, Lemma 4.8]{DGFEM} assume that the exact solution belongs to $ \boldsymbol{\mathcal{{X}}} \cap \h^2(\Omega)$. The non-trivial modification of the proof techniques appeal to a clever re-grouping of the terms that  involve the boundary terms  and an application of Lemma \ref{local efficiency for discrete infsup}. Moreover,  in \cite[Lemma 4.8]{DGFEM}, the stability of the perturbed bilinear form $\dual{DN_h({{\rm I}_{h}\Psi})\cdot,\cdot }$ is established  by first proving the stability of $\dual{DN_h({\Psi})\cdot,\cdot }$ (see  \cite[Theorem 4.8]{DGFEM}).   In this article,  we provide an {\it alternate simplified proof} that {\it directly} establishes the stability  of the perturbed bilinear form using  Lemma \ref{Lemma for discrete infsup}.
	\begin{lem}\label{Lemma for discrete infsup}
	Let $\Psi $ be a regular solution of \eqref{continuous nonlinear} and ${\rm I}_{h}\Psi$ be its interpolant. For $\Theta_{h} \!\in\! \X_h$ with $\vertiiih{\Theta_{h}}=1$, and the interpolant ${\rm{I}}_h \boldsymbol \xi \in \V_{h}$ of the solution $\boldsymbol \xi$ of \eqref{2.4.2.1}, it holds that
		\begin{align*}
		\vertiiih{\Theta_{h} + {\rm{I}}_h\boldsymbol \xi}\lesssim  \dual{DN_h({\rm I}_{h}\Psi)\Theta_{h},\Phi_{h} }+   \epsilon^{-2}h^\alpha (1+\vertiii{\Psi}^2_{1+\alpha}).
		\end{align*}
	\end{lem}
\begin{proof}
		Since $\Theta_{h} + \textrm{I}_h \boldsymbol \xi \!\in\! \X_h,$ the discrete coercivity condition in Lemma \ref{boundedness}$(ii)$ implies that there exists $ \Phi_{h}\!\in\! \X_h$ with $\vertiiih{\Phi_{h} }=1$ such that $\displaystyle 	\vertiiih{\Theta_{h} + \textrm{I}_h\boldsymbol \xi} \lesssim A_h(\Theta_{h} + \textrm{I}_h\boldsymbol \xi,\Phi_{h}). $
	This inequality with \eqref{2.4.2.1}, \eqref{perturbed bilinear form} and a regrouping of terms  yields
	\begin{align}\label{2.4.8.1}
	\vertiiih{\Theta_{h} + \textrm{I}_h\boldsymbol \xi} &\lesssim \dual{DN_h({\rm I}_{h}\Psi)\Theta_{h},\Phi_{h} } + 
	(A_h(\textrm{I}_h\boldsymbol \xi,\Phi_{h})-A(\boldsymbol \xi,\Pi_h\Phi_{h}))+
	(3B({\rm I}_{h}\Psi,{\rm I}_{h}\Psi,\Theta_{h},\Pi_h\Phi_{h}-\Phi_{h}) \nonumber \\
	& \quad + C(\Theta_{h},\Pi_h\Phi_{h}-\Phi_{h})) +
	3(B(\Psi,\Psi,\Theta_{h},\Pi_h\Phi_{h}) 
	- B({\rm I}_{h}\Psi,{\rm I}_{h}\Psi,\Theta_{h},\Pi_h\Phi_{h}) ).		
	\end{align}
	The definition of $A_h(\cdot,\cdot)$ and $\textrm{I}_h\boldsymbol \xi =0$ on $\partial\Omega$  lead to
	\begin{align} \label{first_term}
	A_h(\textrm{I}_h\boldsymbol \xi,\Phi_{h})-A(\boldsymbol \xi,\Pi_h\Phi_{h})
	& = (A(\textrm{I}_h\boldsymbol \xi,\Phi_{h}-\Pi_h\Phi_{h})- \dual{\nabla (\textrm{I}_h\boldsymbol \xi) \nu, \Phi_{h}}_{\partial \Omega}) +
	A(\textrm{I}_h\boldsymbol \xi-\boldsymbol \xi,\Pi_h\Phi_{h}) .
	\end{align}
	\noindent  An  integration by parts element-wise, and $\Delta (\textrm{I}_h\boldsymbol \xi)=0$ in $T$,  $\Pi_h\Phi_{h} =0 $ on $\partial \Omega$, and $ [\Phi_{h}-\Pi_h \Phi_{h} ]_E =0 $ for all $E\in \mathcal{E}^i_h$  lead to an estimate for the first term in the right-hand side of \eqref{first_term} as
	\begin{align*}
	A(\textrm{I}_h\boldsymbol \xi,\Phi_{h}-\Pi_h\Phi_{h})- \dual{\nabla(\textrm{I}_h\boldsymbol \xi) \nu, \Phi_{h}-\Pi_h\Phi_{h}}_{\partial \Omega}
	= \sum_{E \in \mathcal{E}_h^i} \dual{[\nabla(\textrm{I}_h\boldsymbol \xi) \nu_E ]_E, \Phi_{h}-\Pi_h \Phi_{h} }_{E}.
	\end{align*}
	Note that the above term can be combined with the third term on the right-hand side of \eqref{2.4.8.1} to rewrite the expression  with the help of local term $\boldsymbol{ \eta}_T=(2 \epsilon^{-2}(\abs{{\rm{I}}_h \Psi}^2  \Theta_{h} + 2 ({\rm{I}}_h \Psi \cdot  \Theta_{h}) {\rm{I}}_h \Psi -\Theta_{h}))|_T$ on a triangle $T$ and $\boldsymbol{ \eta}_E= [\nabla({\rm{I}}_h\boldsymbol{\xi}) \nu_E]_E$ on the edge $E$ can be rewritten as 
	\begin{align*}
	& A(\textrm{I}_h\boldsymbol \xi,\Phi_{h}-\Pi_h\Phi_{h})- \dual{\nabla(\textrm{I}_h\boldsymbol \xi) \nu, \Phi_{h}-\Pi_h\Phi_{h}}_{\partial \Omega}+ (3B({\rm I}_{h}\Psi,{\rm I}_{h}\Psi,\Theta_{h},\Pi_h\Phi_{h}-\Phi_{h}) +C(\Theta_{h},\Pi_h\Phi_{h}-\Phi_{h})) \notag	
	\\	& = -\sum_{T \in \mathcal{T}}\int_T \boldsymbol{ \eta}_T \cdot  (\Phi_{h}-\Pi_h\Phi_{h})\dx+\sum_{E \in \mathcal{E}^i_h} \dual{\boldsymbol{ \eta}_E, \Phi_{h}-\Pi_h \Phi_{h} }_{E}.
	\end{align*}
	A  Cauchy-Schwarz inequality, 
	Lemma \ref{local efficiency for discrete infsup} and the inequality \eqref{enrichment2} applied to the right-hand side of the last equality yield
	\begin{align}\label{Term4}
	\sum_{T \in \mathcal{T}}\int_T \boldsymbol{ \eta}_T \cdot  (\Pi_h\Phi_{h}-\Phi_{h})\dx+\sum_{E \in \mathcal{E}^i_h} \dual{\boldsymbol{ \eta}_E, \Phi_{h}-\Pi_h \Phi_{h} }_{E} 
	&	\lesssim \vertiiih{\Phi_{h}} (\sum_{T \in \mathcal{T}} h_T^2 \vertiii{\boldsymbol{ \eta}_T}_{0,T}^2+\sum_{E \in \mathcal{E}^i_h} h_E \vertiii{\boldsymbol{ \eta}_E}_{0,E}^2)^{\frac{1}{2}}\nonumber\\&\lesssim \epsilon^{-2}h^\alpha (1+\vertiii{\Psi}^2_{1+\alpha}).
	\end{align}
	Next we proceed to estimate the terms that remain on the right-hand side of  \eqref{first_term} and \eqref{2.4.8.1}. 
	Lemma \ref{2.3.17}$(i)$, Lemma \ref{Interpolation estimate}, \eqref{enrichment1}, $\vertiiih{\Phi_{h}}=1$ and \eqref{2.4.4.0} lead to 
	\begin{align}\label{Term1}
	A(\textrm{I}_h\boldsymbol \xi-\boldsymbol \xi,\Pi_h\Phi_{h}) \lesssim \vertiii{\nabla(\textrm{I}_h\boldsymbol \xi-\boldsymbol \xi)}_0 \vertiiih{\Phi_{h}} \lesssim h^\alpha \vertiii{\boldsymbol \xi}_{1+\alpha} \vertiiih{\Phi_{h}} \lesssim \epsilon^{-2}h^\alpha (1+\vertiii{\Psi}^2_{1+\alpha}).
	\end{align}
	Lemma \ref{2.3.17}$(v)$, \eqref{enrichment1}, $\vertiiih{\Theta_{h}}=1$ and  $\vertiiih{\Phi_{h}}=1$ lead to
	\begin{align}\label{Term5}
	3(B(\Psi,\Psi,\Theta_{h},\Pi_h\Phi_{h}) - B({\rm I}_{h}\Psi,{\rm I}_{h}\Psi,\Theta_{h},\Pi_h\Phi_{h}) )\lesssim   \epsilon^{-2}h^\alpha\vertiii{\Psi}^2_{1+\alpha}.
	\end{align}
	A substitution of \eqref{Term4}- \eqref{Term5} in \eqref{2.4.8.1} concludes the proof of Lemma \ref{Lemma for discrete infsup}.
\end{proof}
\begin{thm} \label{discrete inf sup} \emph{(Stability of perturbed bilinear form).}
		Let $\Psi$ be a regular solution of
		\eqref{continuous nonlinear} and ${{\rm I}_{h}}\Psi$ be its interpolant. For a sufficiently large $\sigma$, and a sufficiently small discretization parameter $h$, there exists a constant $\beta_0 $ such that the perturbed bilinear form in \eqref{perturbed bilinear form} satisfies the following discrete inf-sup condition: 
		\begin{align*}
		0< \beta_0 \leq \inf_{\substack{\Theta_{h} \in \X_h \\  \vertiii{\Theta_{h}}_{h}=1}} \sup_{\substack{\Phi_{h} \in \X_h \\  \vertiii{\Phi_{h}}_{h}=1}}\dual{DN_h({{\rm I}_{h}\Psi})\Theta_{h},\Phi_{h} }.
		\end{align*}
	\end{thm}
	\begin{proof} The inf-sup condition in \eqref{2.9}, \eqref{2.4.3.1.1} and Lemma \ref{boundedness}$(i)$  imply  that there exists $\Phi \!\in\! \V$ with $\vertiii{\Phi}_1 =1$ such that
		$$ \displaystyle \beta \vertiii{\Pi_h \Theta_{h}}_1\leq \dual{DN(\Psi)\Pi_h \Theta_{h},\Phi}=A(\Pi_h \Theta_{h} + \boldsymbol \eta,\Phi) \leq \vertiiih{\Pi_h \Theta_{h} + \boldsymbol \eta}. $$
		
	\noindent Recall that  $\boldsymbol \xi$  is the solution of \eqref{2.4.2.1} and ${\rm{I}}_h \boldsymbol \xi$ is its  interpolant. A triangle inequality followed by an application of the last displayed inequality yields
		\begin{align}\label{2.4.5.1}
		1=\vertiiih{\Theta_{h}} &\leq \vertiiih{ \Theta_{h}-\Pi_h \Theta_{h}} + \vertiii{\Pi_h \Theta_{h}}_1 \lesssim  \vertiiih{ \Theta_{h}-\Pi_h \Theta_{h}} + \vertiiih{\Pi_h \Theta_{h} + \boldsymbol \eta}\notag\\& \lesssim \vertiiih{ \Theta_{h}-\Pi_h \Theta_{h}} + \vertiiih{\Theta_{h} + \textrm{I}_h\boldsymbol \xi} +\vertiiih{\textrm{I}_h\boldsymbol \xi -\boldsymbol \xi}     +  \vertiii{\nabla({\boldsymbol \xi}-{\boldsymbol \eta})}_0,
		\end{align}
		where ${\boldsymbol \xi}-{\boldsymbol \eta} =0$ on $\mathcal{E}^{\partial}_h$ is used in the last term.
		Since $\boldsymbol \xi=0$ on $\mathcal{E}^{\partial}_h$, \eqref{enrichment1} and a triangle inequality yield
		\begin{align*}
		\vertiiih{ \Theta_{h}-\Pi_h \Theta_{h}}  \leq  C_{e_2}(\sum_{ E \in \mathcal{E}^{\partial}_h}h_E^{-1}\vertiii{ \Theta_{h}+\boldsymbol \xi}^2_{0, E} )^{\frac{1}{2}}\leq C_{e_2} \vertiiih{\Theta_{h}+\boldsymbol \xi} \leq C_{e_2} (\vertiiih{\Theta_{h}+\textrm{I}_h\boldsymbol \xi}+\vertiiih{\boldsymbol \xi-\textrm{I}_h\boldsymbol \xi}).
		\end{align*}
	Use this in \eqref{2.4.5.1} and apply Lemmas \ref{Interpolation estimate}, \ref{Lemma for discrete infsup} and \eqref{2.4.4.0} to obtain $\displaystyle 	1 \leq C_1  (\dual{DN_h({\rm I}_{h}\Psi)\Theta_{h},\Phi_{h} }+   \epsilon^{-2}h^\alpha (1+\vertiii{\Psi}^2_{1+\alpha})),$
		where the constant $C_1$ depends on $\alpha_0, C_S, C_I, C_{e_1} $, $C_{e_2},$ $C_{e_3}$ and is independent of $h$. 
		Therefore, for a given $\epsilon$, the discrete inf-sup condition holds with $\beta_0= \frac{1}{C_1}$ for $h < h_0:=\left(\frac{\epsilon^2}{2C_1(1+\vertiii{\Psi}^2_{1+\alpha})} \right)^{\frac{1}{\alpha}} $.
	\end{proof}	
\begin{rem} In \cite{DGFEM},  under the assumption that exact solution has $\h^2$ regularity, the discrete inf-sup condition is established for a  choice of $h=O(\epsilon^2)$. Though $h-\epsilon$ dependency is not the focus of this paper, for the case $\alpha=1$, where it is well-known\cite{Bethuel} that  $\vertiii{\Psi}_{2}$ is bounded independent of $\epsilon,$   $h-\epsilon$ dependency results can be derived analogous to \cite{DGFEM}.
\end{rem}

The proof of the energy norm error estimate in Theorem \ref{energy and $L^2$ norm error estimate} utilizes the methodology of \cite{DGFEM}.  However, Lemma \ref{Lemma for ball to ball map} establishes the estimate that requires non-trivial modifications of the techniques used in \cite{DGFEM} to prove energy norm error estimates.
\begin{lem}(An intermediate estimate)\label{Lemma for ball to ball map}
	Let $\Psi$ be a regular solution of \eqref{continuous nonlinear} and ${\rm I}_h\Psi \in \X_{h}$ be it's interpolant. Then, for any $\Phi_h \in \X_h$ with $ \vertiiih{\Phi_{h} }=1,$ it holds that
	\begin{align*}
	A_{h}({\rm{I}}_{h}\Psi , \Phi_{h} ) + B({\rm{I}}_{h}\Psi,{\rm{I}}_{h}\Psi,{\rm{I}}_{h}\Psi , \Phi_{h} )
	+ C({\rm{I}}_{h}\Psi , \Phi_{h} )  -L_{h}(\Phi_{h})\lesssim h^{\alpha} (1+ \epsilon^{-2}h^{\alpha}(1+ \vertiii{ \Psi}_{1+\alpha}^2)) \vertiii{ \Psi}_{1+\alpha}.
	\end{align*}
\end{lem}
\begin{proof}
Add and subtract $(A_{h}(\textrm{I}_{h}\Psi ,  \Pi_h\Phi_{h} )-L_{h}( \Pi_h\Phi_{h}))  $  to rewrite the left-hand side of the above displayed inequality as  
		\begin{align}\label{3.22 dG}
	A_{h}(\textrm{I}_{h}\Psi , \Phi_{h} ) + B&(\textrm{I}_{h}\Psi,\textrm{I}_{h}\Psi,\textrm{I}_{h}\Psi , \Phi_{h} )
	+ C(\textrm{I}_{h}\Psi , \Phi_{h} ) 
	-L_{h}(\Phi_{h})=(A_{h}(\textrm{I}_{h}\Psi , \Phi_{h}- \Pi_h\Phi_{h} )-L_{h}(\Phi_{h}- \Pi_h\Phi_{h}))
	  \notag\\&	
+(C(\textrm{I}_{h}\Psi , \Phi_{h} )	+B(\textrm{I}_{h}\Psi,\textrm{I}_{h}\Psi,\textrm{I}_{h}\Psi , \Phi_{h} ))  +(A_{h}(\textrm{I}_{h}\Psi ,  \Pi_h\Phi_{h} )-L_{h}( \Pi_h\Phi_{h})).    
	\end{align}
		The definition of $A_h(\cdot,\cdot)$ and $L_h(\cdot)$, followed by an integration by parts element-wise for the term  $A(\cdot,\cdot)$, $\Delta (\textrm{I}_{h}\Psi) = 0$ and $[\Phi_{h}- \Pi_h\Phi_{h}]_E=0$ for $E \in \mathcal{E}^i_h$ show that 
	\begin{align}\label{3.24}
	& A_{h}(\textrm{I}_{h}\Psi , \Phi_{h}- \Pi_h\Phi_{h} )  -L_{h}(\Phi_{h}- \Pi_h\Phi_{h})= \sum_{ E \in \mathcal{E}_h^{i}}\dual{  [\nabla(\textrm{I}_{h}\Psi)\nu_E ]_E,\Phi_{h}- \Pi_h\Phi_{h}}_{E}  \nonumber \\
	& \qquad + \sum_{E \in \mathcal{E}_h^{\partial}} \frac{\sigma}{h_E}\dual{\textrm{I}_{h}\Psi- \mathbf{g}, \Phi_{h}-  \Pi_h\Phi_{h}}_E 
	  +\dual{ \mathbf{g}-\textrm{I}_{h}\Psi, \nabla (\Phi_{h}- \Pi_h\Phi_{h})\nu }_{\partial \Omega}.
	\end{align}
	Set $\boldsymbol{ \eta}_T=( 2\epsilon^{-2}(\abs{\textrm{I}_{h}\Psi}^2 -1)\textrm{I}_{h}\Psi)|_T \text{ on a triangle } T \text{ and } \boldsymbol{ \eta}_E =  [\nabla(\textrm{I}_{h}\Psi)\nu_E]_E \text{ on the edge } E$ and observe that  
	\begin{align}\label{Term}
	\sum_{ E \in \mathcal{E}_h^{i}}\dual{  [\nabla(\textrm{I}_{h}\Psi)\nu_E ]_E,\Phi_{h}&- \Pi_h\Phi_{h}}_{E} +C(\textrm{I}_{h}\Psi , \Phi_{h} )  
	+B(\textrm{I}_{h}\Psi,\textrm{I}_{h}\Psi,\textrm{I}_{h}\Psi , \Phi_{h} )=(\sum_{T \in \mathcal{T}}\int_T \boldsymbol{ \eta}_T \cdot (\Phi_{h}-  \Pi_h\Phi_{h})\dx \notag \\& + \sum_{ E \in \mathcal{E}_h^{i}}\dual{ \boldsymbol{ \eta}_E ,\Phi_{h}- \Pi_h\Phi_{h}}_{E}) +(B(\textrm{I}_{h}\Psi,\textrm{I}_{h}\Psi,\textrm{I}_{h}\Psi ,  \Pi_h\Phi_{h} ) +C({\rm I}_{h}\Psi , \Pi_h\Phi_{h})).
	\end{align}
	The definition of $A_h(\cdot,\cdot) $, the consistency of the exact solution $\Psi$ given by $ N_h(\Psi, \Pi_h\Phi_{h})= L_{h}( \Pi_h\Phi_{h}) $ and $\Pi_h\Phi_{h} = 0 $ on ${\partial \Omega}$ yield, 
	\begin{align} \label{3.23}
	A_{h}(\textrm{I}_{h}\Psi ,  \Pi_h\Phi_{h} )-L_{h}( \Pi_h\Phi_{h}) 
	& = A(\textrm{I}_{h}\Psi -\Psi,  \Pi_h\Phi_{h} )+\dual{\mathbf{g}-\textrm{I}_{h}\Psi, \nabla(\Pi_h\Phi_{h})\nu}_{\partial \Omega} \nonumber \\
	& \quad - (B(\Psi ,\Psi ,\Psi ,  \Pi_h\Phi_{h}) + C(\Psi ,  \Pi_h\Phi_{h} )).
	\end{align} 
	An application of \eqref{3.24}-\eqref{3.23} in \eqref{3.22 dG}, a cancellation of a boundary term and a suitable re-arrangement of terms leads to 
	\begin{align}\label{3.22_new}
	& A_{h}({\rm{I}}_{h}\Psi , \Phi_{h} ) + B({\rm{I}}_{h}\Psi,{\rm{I}}_{h}\Psi,{\rm{I}}_{h}\Psi , \Phi_{h} )
	+ C({\rm{I}}_{h}\Psi , \Phi_{h} )  -L_{h}(\Phi_{h})  = 
	\sum_{T \in \mathcal{T}}\int_T \boldsymbol{ \eta}_T \cdot (\Phi_{h}-  \Pi_h\Phi_{h})\dx  \nonumber 
	\\
	& + \sum_{ E \in \mathcal{E}_h^{i}}\dual{ \boldsymbol{ \eta}_E ,\Phi_{h}- \Pi_h\Phi_{h}}_{E} + (A(\textrm{I}_{h}\Psi -\Psi,  \Pi_h\Phi_{h} )+ C({\rm I}_{h}\Psi -\Psi, \Pi_h\Phi_{h})) + \dual{\mathbf{g}-\textrm{I}_{h}\Psi, \nabla\Phi_{h} \nu}_{\partial \Omega}
	\nonumber \\
	&  + 
	\sum_{E \in \mathcal{E}_h^{\partial}} \frac{\sigma}{h_E}\dual{\textrm{I}_{h}\Psi- \mathbf{g}, \Phi_{h}-  \Pi_h\Phi_{h}}_E+ (B(\textrm{I}_{h}\Psi,\textrm{I}_{h}\Psi,\textrm{I}_{h}\Psi ,  \Pi_h\Phi_{h} )-B(\Psi ,\Psi ,\Psi ,  \Pi_h\Phi_{h})) . 	 
	\end{align}
	 Now we estimate the terms on the right-hand side of \eqref{3.22_new}.	
	A Cauchy-Schwarz  inequality, {\it Lemma \ref{lem 3.1}$(ii)$} and \eqref{enrichment2} with $\vertiiih{\Phi_{h}}=1
	$ leads to
	\begin{align}\label{term_eta}
	&\sum_{T \in \mathcal{T}}\int_T \boldsymbol{ \eta}_T \cdot (\Phi_{h}-  \Pi_h\Phi_{h})\dx + \sum_{ E \in \mathcal{E}_h^{i}}\dual{\boldsymbol{ \eta}_E ,(\Phi_{h}- \Pi_h\Phi_{h})}_{E} 
	\lesssim h^{\alpha} (1+ \epsilon^{-2}h^{\alpha}(1+ \vertiii{ \Psi}_{1+\alpha}^2)) \vertiii{ \Psi}_{1+\alpha}.
	\end{align}
	A use of Lemma \ref{Interpolation estimate}, \eqref{boundedness}$(ii)$ (resp. $(iii)$), Remark \ref{remark interpolation} and  \eqref{enrichment1} yields 
	\begin{align}
		A(\textrm{I}_{h}\Psi -\Psi,  \Pi_h\Phi_{h} )+C({\rm I}_{h}\Psi -\Psi , \Pi_h\Phi_{h}) &\lesssim (\vertiiih{\textrm{I}_{h}\Psi -\Psi}+ \epsilon^{-2}  \vertiii{\textrm{I}_{h}\Psi -\Psi}_0) \vertiiih{\Phi_{h}  } \notag
	\\&\lesssim  h^{\alpha}(1+\epsilon^{-2} h^{\alpha})\vertiii{\Psi}_{1+\alpha} .
	\end{align}
	The next two estimates are obtained using Cauchy-Schwarz inequality, the definition of $\|.\|_h$, Remark \ref{remark interpolation},  \eqref{enrichment1} and Lemma \ref{discrete trace inequality}$(ii)$. 	 
	\begin{align}
	&
	\dual{\mathbf{g}-\textrm{I}_{h}\Psi, \nabla \Phi_{h} \nu}_{\partial \Omega}  \leq
	\big(\sum_{E \in \mathcal{E}_h^{\partial}} \frac{\sigma}{h_E} \vertiii{\textrm{I}_{h}\Psi- \Psi}_{0,E}^2  \big)^{\frac{1}{2}}  \big(\sum_{E \in \mathcal{E}_h^{\partial}} \frac{h_E}{\sigma} \vertiii{\nabla\Phi_{h}\nu_E}_{0,E}^2 \big)^{\frac{1}{2}} 
	\lesssim h^\alpha \vertiii{\Psi}_{1+\alpha} ,
	\\&
	\sum_{E \in \mathcal{E}_h^{\partial}} \frac{\sigma}{h_E}\dual{\textrm{I}_{h}\Psi- \mathbf{g}, \Phi_{h}-  \Pi_h\Phi_{h}}_E
	\leq  \vertiiih{\textrm{I}_{h}\Psi -\Psi} \vertiiih{\Phi_{h}-  \Pi_h\Phi_{h} }
	\lesssim h^\alpha \vertiii{\Psi}_{1+\alpha}.
	\end{align}
	Lemma \ref{2.3.17}$(v)$, \eqref{enrichment1} and $\vertiiih{\Phi_h}=1$ yield
	\begin{align}\label{term_Bh1}
	&B(\textrm{I}_{h}\Psi,\textrm{I}_{h}\Psi,\textrm{I}_{h}\Psi ,  \Pi_h\Phi_{h} )-B(\Psi ,\Psi ,\Psi ,  \Pi_h\Phi_{h}) \lesssim  \epsilon^{-2} h^{2\alpha}\vertiii{\Psi}_{1+\alpha}^3.
	\end{align}
	A combination of the estimates in \eqref{term_eta}- \eqref{term_Bh1} completes the proof of Lemma \ref{Lemma for ball to ball map}.
\end{proof}
\noindent A use of Lemma \ref{Lemma for ball to ball map} and the methodology of \cite{DGFEM} leads to the  proof of Theorem  \ref{energy and $L^2$ norm error estimate}. An outline is sketched for completeness. 
\begin{proof}[\textbf{Proof of energy norm estimate in Theorem  \ref{energy and $L^2$ norm error estimate}}]
For $\Phi_{h} \!\in\! \X_h$,  let the non-linear map $\mu_{h}:\X_h \rightarrow \X_h$ be defined by 
\begin{align}\label{mu nonlinear map}
\dual{DN_h({{\rm I}_{h}\Psi}) \mu_{h}(\Theta_{h}),\Phi_{h} }
= 3B(\textrm{I}_{h}\Psi, \textrm{I}_{h}\Psi,\Theta_{h},\Phi_{h}) - B(\Theta_{h},\Theta_{h}, \Theta_{h},\Phi_{h}) +L_h(\Phi_{h}),
\end{align}
and let $\mathbb{B}_R(\textrm{I}_{h}\Psi):= \{\Phi_h \!\in\! \X_h: \vertiiih{\textrm{I}_h\Psi-\Phi_{h}} \leq R\}.$ Theorem \ref{discrete inf sup} helps to establish that $\mu_{h}$ is well-defined and any fixed point of $\mu_{h}$ is a solution of the discrete non-linear problem \eqref{discrete nonlinear problem}. Moreover,  for a sufficiently large choice of the penalization parameter $\sigma$, and a sufficiently small choice of discretization parameter $h$, there exists a positive constant $R(h)$ such that  $\mu_{h}$ maps the closed convex ball $\mathbb{B}_{R(h)}({\rm I}_{h}\Psi)$ to itself; that is,
$$\vertiii{\Theta_{h} - {\rm I}_{h}\Psi }_{h} \leq R(h) \implies \vertiii{\mu_{h}(\Theta_{h}) - {\rm I}_{h}\Psi}_{h} \leq R(h) \text{ for all } \Theta_{h} \in \X_{h} . $$
\noindent The definition of $\dual{DN_h({\rm I}_{h}\Psi) \cdot ,\cdot} $ from \eqref{perturbed bilinear form}, \eqref{mu nonlinear map}, followed by simple algebra and a re-arrangement of terms leads to
\begin{align}\label{3.22}
& \dual{DN_h({\rm I}_{h}\Psi) ({\rm I}_{h}\Psi-\mu_{h}(\Theta_{h})) ,\Phi_{h} }
 =(A_{h}(\textrm{I}_{h}\Psi , \Phi_{h} ) + B(\textrm{I}_{h}\Psi,\textrm{I}_{h}\Psi,\textrm{I}_{h}\Psi , \Phi_{h} )
+ C(\textrm{I}_{h}\Psi , \Phi_{h} )-L_{h}(\Phi_{h})) \nonumber \\
& \quad\quad +(2B(\textrm{I}_{h}\Psi,\textrm{I}_{h}\Psi,\textrm{I}_{h}\Psi , \Phi_{h} ) -3B(\textrm{I}_{h}\Psi,\textrm{I}_{h}\Psi, \Theta_{h},\Phi_{h})  + B(\Theta_{h},\Theta_{h}, \Theta_{h},\Phi_{h}) )=:T'_1 +T'_2.
\end{align}
\noindent The term $T'_1$ is estimated using Lemma \ref{Lemma for ball to ball map}. Set $\mathbf{\tilde{e}}:=\Theta_{h} -\textrm{I}_{h}\Psi$. The definition of $B(\cdot,\cdot,\cdot,\cdot)$, the Cauchy-Schwarz inequality and some straight forward algebraic manipulations lead to 
\begin{align}\label{term_Bh}
T'_2 \lesssim  2\epsilon^{-2} \vertiiih{\mathbf{\tilde{e}} }^2 (\vertiiih{\mathbf{\tilde{e}} }+\vertiii{\Psi}_{1+\alpha})\vertiiih{\Phi_{h} }.
\end{align}
Discrete inf-sup condition in Lemma \ref{discrete inf sup} yields that there exists a $\Phi_{h}\!\in\! \X_h$ with $\vertiiih{\Phi_{h}}=1$ such that
\begin{align}\label{infsup}
\beta_0\vertiiih{ {\rm I}_{h}\Psi-\mu_{h}(\Theta_{h})} \leq \dual{DN_h({\rm I}_{h}\Psi)  ({\rm I}_{h}\Psi-\mu_{h}(\Theta_{h})) ,\Phi_{h} }.
\end{align}
Since $\Theta_{h} \!\in\! \mathbb{B}_R(\textrm{I}_{h}\Psi), $ $\vertiiih{\tilde{\e}}= \vertiiih{\Theta_{h} -\textrm{I}_{h}\Psi}\leq R(h)$. For a fixed value of $\epsilon$, a use of  Lemma \ref{Lemma for ball to ball map} and \eqref{term_Bh} in \eqref{3.22}, and \eqref{infsup} with $\vertiiih{\Phi_{h}}=1$ leads to 
\begin{align*}
\vertiiih{ {\rm I}_{h}\Psi-\mu_{h}(\Theta_{h})}  \leq C_2(h^{\alpha}(1+h^{\alpha})+ R(h)^2 ( R(h)+1)),
\end{align*}
where $C_2$ is a constant independent of $h.$
For a choice of $R(h):= 2C_2h^{\alpha}$ and
$h< h_2:= \min(h_0, h_1)$ with $h_1^{\alpha}< \frac{1}{1+ 4C_2^2  ( 2C_2 h_0^{\alpha} +1)} $, a simple algebraic calculation leads to
$\vertiiih{ {\rm I}_{h}\Psi-\mu_{h}(\Theta_{h})}  \leq 2C_2h^{\alpha}= R(h).$
 Analogous ideas as \cite[Lemma 5.3]{DGFEM} establishes that $\text{ for all } \Theta_1,  \Theta_2 \!\in\! \mathbb{B}_{R(h)}({\rm I}_{h}\Psi)$,
\begin{align*}
\vertiii{\mu_{h}(\Theta_1)-\mu_{h}(\Theta_2)}_{h}\lesssim h^{\alpha}(h^{\alpha}+1)\vertiii{\Theta_1-\Theta_2}_{h}.
\end{align*}
\noindent Thus the map $\mu_h$ is well-defined, continuous and maps a closed convex subset $\mathbb{B}_R(\textrm{I}_{h}\Psi)$ of a Hilbert space $\X_{h}$ to itself. Therefore, Brouwer's fixed point theorem  and contraction result stated above establishes the existence and uniqueness of the fixed point, say $\Psi_{h}$ in the ball $\mathbb{B}_R(\textrm{I}_{h}\Psi)$. A triangle inequality, $\vertiiih{\textrm{I}_{h}\Psi- \Psi_h}\lesssim h^{\alpha}$ and Remark \ref{remark interpolation} yield the {\it a priori} error estimate in energy norm.
\end{proof}
\begin{rem} The  proof of the energy norm estimate relies on the techniques of medius analysis \cite{Gudi_2010_A_new_error_analysis} to deal with the milder regularity of the exact solution. This involves a different strategy for the proof using the local efficiency results when compared to 
	 \cite[Theorem 5.1]{DGFEM}, where $\h^2(\Omega)$ regularity is assumed for the exact solution.
\end{rem}
\begin{rem}
	For $\alpha=1$, that is, $\Psi \!\in\! \h^{2}(\Omega),$ it is well-known \cite{Bethuel} that  $\vertiii{\Psi}_2$ is bounded independent of $\epsilon$. In this case, $\displaystyle \vertiiih{ {\rm I}_{h}\Psi-\mu_{h}(\Theta_{h})} \leq C_3(h (1+ \epsilon^{-2}h)+\epsilon^{-2} \vertiiih{\mathbf{\tilde{e}} }^2 (\vertiiih{\mathbf{\tilde{e}} }+1)), $
	where the constant $C_3$ is independent of $h$ and $\epsilon$.
For a sufficiently small choice of the discretization parameter chosen as $h =O(\epsilon^{2+\tau}), \tau > 0$ and $R(h)= 2C_3 h $, $\mu_{h}$ maps the ball $\mathbb{B}_R({\rm I}_h \Psi)$ to itself, and it is a contraction map on $\mathbb{B}_R({\rm I}_h \Psi)$. The modification of the proof in above theorem follows analogous to Theorem $5.1$ in \cite{DGFEM} and yields $h$-$\epsilon$ dependent estimates for this case.
\end{rem}	
 Next,  the $L^2$ norm error estimate is derived using the Aubin-Nitsche \cite{ciarlet} duality  technique. The proof relies on energy norm error bounds that has been established for a fixed $\epsilon$. 
  However, when $\Psi \in \h^2(\Omega)$, the proof can be modified as in Theorem $3.5$ in \cite{DGFEM} to obtain $h-\epsilon$ dependent estimates.
	\begin{proof}[\textbf{Proof of ${\bf L}^2$ estimate in Theorem  \ref{energy and $L^2$ norm error estimate}}]  Set ${\boldsymbol \varphi}_h=\textrm{I}_{h}\Psi- \Psi_h$ and choose  $G= {\boldsymbol \varphi}_h$, $\Phi= \Pi_h {\boldsymbol \varphi}_h $ in the continuous dual linear problem  \eqref{dual problem} to deduce
		\begin{align} \label{Term_0}
		\vertiii{{\boldsymbol \varphi}_h}_0^2= ({\boldsymbol \varphi}_h, {\boldsymbol \varphi}_h)
		= ({\boldsymbol \varphi}_h, {\boldsymbol \varphi}_h- \Pi_h {\boldsymbol \varphi}_h) + \dual{DN(\Psi)  \Pi_h {\boldsymbol \varphi}_h , \boldsymbol \chi}.
		\end{align}
		Let $\textrm{I}_{h}\boldsymbol \chi \!\in\! \V_h \subset \h^1_0(\Omega)$ denotes the interpolant of $\boldsymbol \chi $. A use of  $\textrm{I}_h\boldsymbol \chi =0$ on $\partial \Omega$ implies 
		\begin{align*}
		\dual{DN(\Psi)   {\boldsymbol \varphi}_h , \textrm{I}_{h}\boldsymbol \chi}= \dual{DN_h(\Psi) {\boldsymbol \varphi}_h, \textrm{I}_{h}\boldsymbol \chi}   +   \dual{ {\boldsymbol \varphi}_h, \nabla( \textrm{I}_{h} \boldsymbol \chi) \nu}_{\partial \Omega }.
		\end{align*}
		Add and subtract $\dual{DN(\Psi)  ({\boldsymbol \varphi}_h-\Pi_h {\boldsymbol \varphi}_h ), \textrm{I}_{h}\boldsymbol \chi}$ in the right-hand side of \eqref{Term_0}, use the definition of $\dual{DN(\Psi)\cdot, \cdot}$ and the last displayed identity with $\Pi_h{\boldsymbol \varphi}_h=0$ on $\partial \Omega$, and re-arrange the terms to obtain
		\begin{align} \label{3.4.1}
		\vertiii{{\boldsymbol \varphi}_h}_0^2	
		&
		=({\boldsymbol \varphi}_h, {\boldsymbol \varphi}_h- \Pi_h {\boldsymbol \varphi}_h)+ (-A( \textrm{I}_{h}\boldsymbol \chi,{\boldsymbol \varphi}_h-\Pi_h{\boldsymbol \varphi}_h) 
		+ \dual{  \nabla( \textrm{I}_{h} \boldsymbol \chi) \nu,{\boldsymbol \varphi}_h -\Pi_h{\boldsymbol \varphi}_h }_{\partial \Omega }) +(C(\textrm{I}_{h}\boldsymbol \chi,\Pi_h{\boldsymbol \varphi}_h-{\boldsymbol \varphi}_h )
	 \notag \\
	& \quad 	+3B(\Psi,\Psi, \textrm{I}_{h}\boldsymbol \chi, \Pi_h{\boldsymbol \varphi}_h-{\boldsymbol \varphi}_h ) )+\dual{DN(\Psi)  \Pi_h {\boldsymbol \varphi}_h , \boldsymbol \chi -\textrm{I}_{h}\boldsymbol \chi}  +\dual{DN_h(\Psi) {\boldsymbol \varphi}_h, \textrm{I}_{h}\boldsymbol \chi} 
		\notag \\
		&      
		=:\mathrm{T}_1+\mathrm{T}_2+ \mathrm{T}_3+\mathrm{T}_4+\mathrm{T}_5.
		\end{align}	
A use of H\"older's inequality, \eqref{enrichment2} and the estimate $\vertiiih{{\boldsymbol \varphi}_h}= \vertiiih{\textrm{I}_{h}  \Psi -\Psi_h}\lesssim h^\alpha$ from the proof of Theorem \ref{energy and $L^2$ norm error estimate} leads to  $\displaystyle \mathrm{T}_1=({\boldsymbol \varphi}_h, {\boldsymbol \varphi}_h- \Pi_h {\boldsymbol \varphi}_h) \leq \vertiii{{\boldsymbol \varphi}_h- \Pi_h {\boldsymbol \varphi}_h}_0 \vertiii{{\boldsymbol \varphi}_h}_0 \lesssim h^{2\alpha} \vertiii{{\boldsymbol \varphi}_h}_0. $
		Apply  integration by parts element-wise for the term $A({\boldsymbol \varphi}_h-\Pi_h{\boldsymbol \varphi}_h,\textrm{I}_{h}{\boldsymbol \chi}  )$ in the expression of ${\mathrm T}_2$, use $\Delta ( \textrm{I}_{h}\boldsymbol \chi)=0$ and recall the definitions of the local term $\boldsymbol{ \eta}_E=[\nabla( \textrm{I}_{h} \boldsymbol \chi) \nu_E]_E $ on $E$ from Lemma \ref{local efficiency for L2 estimate} with $G= {\boldsymbol \varphi}_h$ and $Osc({\boldsymbol \varphi}_h)=0$. This with a Cauchy-Schwarz inequality, Lemma \ref{local efficiency for L2 estimate}, \eqref{enrichment2} and  the estimate $\vertiiih{{\boldsymbol \varphi}_h} \lesssim h^\alpha$ leads to
		\begin{align*} 
		{\mathrm T}_2=	-A( \textrm{I}_{h}\boldsymbol \chi,{\boldsymbol \varphi}_h-\Pi_h{\boldsymbol \varphi}_h) 
		+ \dual{  \nabla( \textrm{I}_{h} \boldsymbol \chi) \nu,{\boldsymbol \varphi}_h -\Pi_h{\boldsymbol \varphi}_h }_{\partial \Omega }
		= \sum_{ E \in \mathcal{E}_h^{i}} \dual{ \boldsymbol{ \eta}_E, \Pi_h{\boldsymbol \varphi}_h -{\boldsymbol \varphi}_h }_E 
		\lesssim h^{2\alpha} \vertiii{{\boldsymbol \varphi}_h}_{0}.
		\end{align*}
		Lemma \ref{2.3.17}, \eqref{xi regularity bound}, \eqref{enrichment1} and $\vertiiih{{\boldsymbol \varphi}_h}\lesssim h^\alpha$  
		yield 
		\begin{align*} 
{\mathrm T}_3=	3B(\Psi,\Psi, \textrm{I}_{h}\boldsymbol \chi, \Pi_h{\boldsymbol \varphi}_h-{\boldsymbol \varphi}_h )+C(\textrm{I}_{h}\boldsymbol \chi,\Pi_h{\boldsymbol \varphi}_h-{\boldsymbol \varphi}_h ) 
			\lesssim   \vertiii{\textrm{I}_{h}\boldsymbol \chi}_h  
		\vertiii{\Pi_h{\boldsymbol \varphi}_h -{\boldsymbol \varphi}_h}_0 
	\lesssim   h^{2\alpha}
		\vertiii{{\boldsymbol \varphi}_h}_0. 
		\end{align*}
%
		\noindent  The boundedness and interpolation estimates in Lemmas \ref{boundedness}, \ref{Interpolation estimate} 
		, \eqref{enrichment1} and $\vertiiih{{\boldsymbol \varphi}_h}\lesssim h^\alpha,$ and \eqref{xi regularity bound}, leads to a bound for the fourth term of \eqref{3.4.1} as 
		\begin{align*}
		{\mathrm T }_4 &
		= A(   \Pi_h{\boldsymbol \varphi}_h , \boldsymbol \chi -\textrm{I}_{h}\boldsymbol \chi)+3B(\Psi,\Psi,  \Pi_h{\boldsymbol \varphi}_h , \boldsymbol \chi -\textrm{I}_{h}\boldsymbol \chi)+ C(  \Pi_h{\boldsymbol \varphi}_h , \boldsymbol \chi -\textrm{I}_{h}\boldsymbol \chi) 
		\lesssim h^{2\alpha} \vertiii{\boldsymbol \chi}_{1+\alpha}\lesssim h^{2\alpha} \vertiii{{\boldsymbol \varphi}_h}_0. 
		\end{align*}		
		The discrete nonlinear problem \eqref{discrete nonlinear problem} plus the consistency of the exact solution  $\Psi$ yield $\displaystyle 	N_h(\Psi, \textrm{I}_{h}\boldsymbol \chi)= L_{h}( \textrm{I}_{h}\boldsymbol \chi) = N_h(\Psi_h, \textrm{I}_{h}\boldsymbol \chi).$
		\noindent Recall that ${\boldsymbol \varphi}_h=\textrm{I}_{h}\Psi- \Psi_h$ and re-write the last term in \eqref{3.4.1}  using the above displayed identity, and the definitions of $DN_h$ and $N_h$ as
		\begin{align*}
		{\mathrm T}_5  =& \dual{DN_h(\Psi) {\boldsymbol \varphi}_h, \textrm{I}_{h}\boldsymbol \chi}     + N_h(\Psi_{h}, \textrm{I}_{h}\boldsymbol \chi)- N_h(\Psi, \textrm{I}_{h}\boldsymbol \chi)
		= A_h(\textrm{I}_{h}\Psi- \Psi , \textrm{I}_{h}\boldsymbol \chi)+(C(\textrm{I}_{h}\Psi- \Psi, \textrm{I}_{h}\boldsymbol \chi )\\& +3B(\Psi,\Psi,\textrm{I}_{h}\Psi- \Psi,\textrm{I}_{h}\boldsymbol \chi))+(2 B(\Psi,\Psi,\Psi,\textrm{I}_{h}\boldsymbol \chi) -3 B(\Psi,\Psi, \Psi_h,\textrm{I}_{h}\boldsymbol \chi)+ B(\Psi_h,\Psi_h,\Psi_h,\textrm{I}_{h}\boldsymbol \chi)).
		\end{align*}
		An integration by parts element-wise for the term $A(\Psi-\textrm{I}_{h}\Psi,\textrm{I}_{h} \boldsymbol \chi)$, $\Delta \textrm{I}_{h} \boldsymbol \chi=0,$ a Cauchy-Schwarz inequality, {\it Lemma \ref{local efficiency for L2 estimate}} with $G= {\boldsymbol \varphi}_h$, $Osc({\boldsymbol \varphi}_h)=0$ and Lemma \ref{Interpolation estimate} lead to an estimate for the first term on the right-hand side of ${\mathrm T }_5$ above as
		\begin{align}\label{L2_5}
	 \sum_{ E \in \mathcal{E}_h^{i}} \dual{\boldsymbol{ \eta}_E,\textrm{I}_{h}\Psi -\Psi }_E
		\lesssim h^{2\alpha} \vertiii{{\boldsymbol \varphi}_h}_0 .
		\end{align}
		Here, $\boldsymbol{ \eta}_E$ is the local term as defined in the above estimates.
	Lemma \ref{2.3.17},	Remark \ref{remark interpolation} and  \eqref{xi regularity bound} leads to an estimate for the second term in the expression on the right-hand side for ${\mathrm T}_5$ above as
		\begin{align}\label{L2_4}
	C(\textrm{I}_{h}\Psi- \Psi, \textrm{I}_{h}\boldsymbol \chi )+	3B(\Psi,\Psi, \Psi-\textrm{I}_{h}\Psi ,\textrm{I}_{h}\boldsymbol \chi)
		\lesssim   h^{2\alpha} \vertiii{\boldsymbol \chi}_{1+\alpha} 	\lesssim h^{2\alpha} \vertiii{{\boldsymbol \varphi}_h}_0 .
		\end{align}
		Proceed as in the estimate for $T_2'$ (also see  \cite[ Theorem $3.5$, $T_8$]{DGFEM}),
		and use Remark \ref{remark interpolation} and \eqref{xi regularity bound} to estimate the third term in the expression on the right-hand side for ${\mathrm T}_5$ as
		\begin{align}\label{L2_7}
		&
		2 B(\Psi,\Psi,\Psi,\textrm{I}_{h}\boldsymbol \chi) -3 B(\Psi,\Psi, \Psi_h,\textrm{I}_{h}\boldsymbol \chi)+ B(\Psi_h,\Psi_h,\Psi_h,\textrm{I}_{h}\boldsymbol \chi) \notag\\& 
		\lesssim \vertiiih{\Psi- \Psi_h}^2(\vertiiih{\Psi- \Psi_h}+\vertiii{\Psi}_1)\vertiiih{\textrm{I}_{h}\boldsymbol \chi}
		\lesssim  h^{2\alpha}\vertiii{\boldsymbol \chi}_{1+\alpha}
		\lesssim h^{2\alpha} \vertiii{{\boldsymbol \varphi}_h}_0.
		\end{align}
		A combination of the  estimates in \eqref{L2_5}- \eqref{L2_7} yields ${\mathrm T}_5\lesssim  h^{2\alpha} \vertiii{{\boldsymbol \varphi}_h}_0.$	Substitute the estimates derived for  ${\mathrm T}_1$ to ${\mathrm T}_5$ in \eqref{3.4.1} and cancel the term $\vertiii{{\boldsymbol \varphi}_h}_0$ to obtain $ \vertiii{\Ihpsi - \Psi_{h}}_0 \lesssim h^{2\alpha}. $ This estimate,  a triangle inequality and Lemma \ref{Interpolation estimate} yield
		$\vertiii{\Psi - \Psi_{h}}_0\leq \vertiii{\Psi - \Ihpsi }_0 + \vertiii{\Ihpsi  - \Psi_{h}}_0 \lesssim h^{2\alpha}$ and this concludes the proof.
	\end{proof}
	\section{A posteriori error estimate} \label{A posteriori error estimate}

		In this section, we present some auxiliary results  followed by the {\it a posteriori} error analysis for the Nitsche's method. 
		Note that, to derive the {\it a posteriori} estimates, it is assumed that $\mathbf{g}$ (the inhomogeneous Dirichlet boundary condition) belongs to $ \h^{\frac{1}{2}}(\partial\Omega)\cap \mathbf{C}^0(\overline{\partial\Omega}).$

\medskip

\noindent	The approximation properties of the  Scott-Zhang interpolation operator \cite{ ScottZhang} are introduced first. 
		\begin{lem}(Scott-Zhang interpolation )\cite{ ScottZhang}\label{Scott-Zhang interpolation} 
For $l, m \in \mathbb{N}$ with $1 \leq l < \infty$, there exists an interpolation operator ${\rm I}^{SZ}_h: H^l_0(\Omega) \rightarrow V_h:= X_h \cap H^1_0(\Omega) $  that satisfies the stability and aproximation properties given by:
			(a)  for all $0\leq m \leq \min(1,l)$,
				$\norm{{\rm I}^{SZ}_h v }_{m, \Omega} \leq C_{SZ}\norm{v}_{l,\Omega}  \,\text{ for all } v \in H^{l}_0(\Omega) ,$
			(b)  provided $l \leq 2$, for all $0 \leq m \leq l$,
				 $\norm{v- {\rm I}^{SZ}_h v }_{m, T} \leq C_{SZ}h_T^{l-m}\abs{v}_{l, \omega_T}  \text{ for all } v \in H_0^{l}(\omega_T) \text{ and } T \in \mathcal{T},$
			where the constant $C_{SZ}>0$ is independent of $h$,  and $\omega_T$ is the set of all triangles  in $\mathcal{T}$ that share at least one vertex with $T$. 
	\end{lem}
		\begin{lem}\cite[Page 48]{KimKwang}\label{psi-psi_g estimate}
		Let $\Psi_{\mathbf{g}} \in \boldsymbol{\mathcal{{X}}}$ solve
			$\displaystyle \int_{\Omega} \nabla \Psi_{\mathbf{g}} \cdot \nabla \Phi\, {\rm dx} = \sum_{ T \in \mathcal{T}} \int_T \nabla \Psi_{h} \cdot \nabla \Phi \, {\rm dx}  \text{ for all } \Phi \in \V,$
			where $\Psi_h$ is the solution of \eqref{discrete nonlinear problem}. Then there exists a constant $C>0$, depending only on the minimum angle of $\mathcal{T} $ such that 
		$\displaystyle	\sum_{ T \in \mathcal{T}} \vertiii{\nabla (\Psi_{\mathbf{g}} - \Psi_{h})}_{0, T}^2 \leq C (\vartheta_{hot}^{\partial})^2, $
			where 
			$\displaystyle 
			(\vartheta_{hot}^{\partial})^2:= \sum_{ E \in \mathcal{E}_h^{\partial} }h_E^{-1} \vertiii{ \mathbf{g} - \mathbf{g}_h}_{0,E}^2+ h_E \vertiii{\nabla(\mathbf{g} - \mathbf{g}_h)}_{0,E}^2
			$,
			$ \mathbf{g}_h$ being the standard Lagrange interpolant \cite{ciarlet} of $\mathbf{g}$ from $\mathbf{P}_2(\mathcal{E}^{\partial}_h) \cap \mathbf{C}^0(\overline{\partial\Omega} ).$
		\end{lem}
	\begin{rem}
{ Note that the benchmark liquid crystal example: Example \ref{Luo_exampple} in \cite{MultistabilityApalachong} has Lipschitz continuous boundary conditions. Hence the {\it a posteriori} error analysis of this paper is applicable to this example and the results are illustrated in Section \ref{Numerical results for a priori error analysis}.
}	
\end{rem}
			 The proof of Theorem \ref{reliability and efficiency}, stated in Subsection \ref{Main results}, is presented in this section. An abstract estimate for the case of non-homogeneous boundary conditions and quartic nonlinearity is derived modifying the methodology in \cite{Gaurang_semilinear, Verfurth} first and this result is crucial to prove Theorem \ref{reliability and efficiency}.
			 \begin{thm}({An abstract estimate})\label{thm 4.1}
			 		Let $\Psi $ be a regular solution to  \eqref{continuous nonlinear} and $\Psi_{\mathbf{g}} \in \boldsymbol{\mathcal{{X}}}$. 	 Then, $DN$ is locally Lipschitz continuous at $\Psi$, that is given $R_0>0$, $DN$ restricted to $B(\Psi,R_0)$ is Lipschitz continuous. Moreover, (a) $\displaystyle \gamma:= \sup_{\boldsymbol \eta \in B(\Psi, R_0)} \dfrac{\vertiii{DN(\boldsymbol \eta) - DN(\Psi)}_{\mathcal{L}(\X, \V^*)}}{\vertiii{\boldsymbol \eta- \Psi}_{h}} < \infty,$
			 	and (b) there exists a constant $R>0$ such that for all ${\boldsymbol \eta}_{h} \in B(\Psi, R)$,
			 		\begin{align} \label{thm_4.4}
			 		\vertiii{\Psi- {\boldsymbol \eta}_h}_h \lesssim \vertiii{N({\boldsymbol \eta}_h)}_{\V^*}  +(1+\vertiii{DN({\boldsymbol \eta}_h)}_{\mathcal{L}(\X, \V^*)})\vertiii{\Psi_{\mathbf{g}} - {\boldsymbol \eta}_h}_h,
			 		\end{align}
			 	where the constant in $"\lesssim"$ depends on $\gamma$, continuous inf-sup constant  $\beta$ and Poincar\'e constant $C_P$,
			 	and the nonlinear (resp. linearized) operator $N(\cdot)$ (resp. $DN(\cdot)$) is defined in \eqref{continuous nonlinear} (resp. \eqref{2.9}).
			 \end{thm}
	\begin{proof}		
				In the {\it first step}, it is established that $DN$ is locally Lipschitz continuous at $\Psi$ and $\gamma <\infty$.
Let  $R_0>0$ be given and $\boldsymbol \eta \!\in\! B(\Psi, R_0).$	For $\Theta \!\in\!\X$ and $\Phi \!\in\! \V,$ the definition of $DN(\cdot)$, $B(\cdot, \cdot, \cdot, \cdot)$, a re-grouping of terms and Lemma \ref{2.3.17}$(iv)$ leads to 
				\begin{align} \label{DN bound}
				&\dual{DN(\boldsymbol \eta) \Theta, \Phi} - \dual{DN(\Psi) \Theta, \Phi}= 3B(\boldsymbol \eta,\boldsymbol \eta, \Theta, \Phi) - 3	B(\Psi,\Psi, \Theta, \Phi) \notag\\& = 2\epsilon^{-2} \int_{\Omega} ((\boldsymbol \eta -\Psi)\cdot ( \boldsymbol \eta+\Psi)(\Theta \cdot \Phi)+ 2 (\boldsymbol \eta -\Psi) \cdot \Theta (\boldsymbol \eta \cdot \Phi) +  2 (\Psi \cdot \Theta) (\boldsymbol \eta -\Psi) \cdot \Phi ) \dx \notag \\&
				\lesssim \epsilon^{-2} \vertiii{\boldsymbol \eta -\Psi}_{1}(R_0 + \vertiii{\Psi}_1) \vertiii{\Theta}_1 \vertiii{\Phi}_1.
				\end{align}
				The above displayed inequality  with definition of $\vertiii{DN(\boldsymbol \eta) - DN(\Psi)}_{\mathcal{L}(\X, \V^*)}$ leads to the Lipschitz continuity. This and Lemma \ref{Poincare type inequality} concludes the proof of the first step.
 
				\medskip
\noindent  {\it Step two} establishes  \eqref{thm_4.4}. The continuous formulation \eqref{continuous nonlinear}  and a Taylor expansion lead to  
	\begin{align*}
	0=N(\Psi;\Phi)= N({\boldsymbol \eta}_h; \Phi) + \duall{\int_{0}^{1}DN(\Psi + t( {\boldsymbol \eta}_h-\Psi ))(\Psi - {\boldsymbol \eta}_h)\dt, \Phi}. 
	\end{align*}
	Introduce  $\pm \dual{DN(\Psi)(\Psi - {\boldsymbol \eta}_h), \Phi} $ in the above displayed expression and rearrange the terms  to obtain
	\begin{align}\label{4.4}
	\dual{DN(\Psi)(\Psi - {\boldsymbol \eta}_h), \Phi}= -N({\boldsymbol \eta}_h; \Phi) - \duall{\int_{0}^{1}(DN(\Psi + t({\boldsymbol \eta}_h-\Psi )) -DN(\Psi) )(\Psi- {\boldsymbol \eta}_h )\dt, \Phi} .
	\end{align}
	Rewrite $\Psi - {\boldsymbol \eta}_h$  as $ (\Psi - \Psi_{\mathbf{g}})+ (\Psi_{\mathbf{g}} - {\boldsymbol \eta}_h )$ in the left-hand side of the above term, use linearity of $\dual{DN(\Psi)\cdot, \cdot}$, introduce $\pm \dual{DN({\boldsymbol \eta}_h)(\Psi_{\mathbf{g}} - {\boldsymbol \eta}_h), \Phi}$ in the first step; and bound in the second step below to obtain
	\begin{align}\label{2.4}
	&\dual{DN(\Psi)(\Psi - \Psi_{\mathbf{g}}), \Phi}= - N({\boldsymbol \eta}_h;\Phi) +\dual{(DN({\boldsymbol \eta}_h)-DN(\Psi))(\Psi_{\mathbf{g}} - {\boldsymbol \eta}_h), \Phi} - \dual{DN({\boldsymbol \eta}_h)(\Psi_{\mathbf{g}} - {\boldsymbol \eta}_h), \Phi}\notag\\&\quad - \duall{\int_{0}^{1}( DN(\Psi + t( {\boldsymbol \eta}_h-\Psi )) - DN(\Psi) )(\Psi - {\boldsymbol \eta}_h)\dt, \Phi} \notag\\&\lesssim \big(\vertiii{N({\boldsymbol \eta}_h)}_{\V^*} + \vertiii{DN({\boldsymbol \eta}_h)-DN(\Psi)}_{\mathcal{L}(\X, \V^*)}\vertiii{\Psi_{\mathbf{g}} - {\boldsymbol \eta}_h}_1 + \vertiii{DN({\boldsymbol \eta}_h)}_{\mathcal{L}(\X, \V^*)}\vertiii{\Psi_{\mathbf{g}} - {\boldsymbol \eta}_h}_1 \notag\\& \quad + \int_{0}^{1} \vertiii{( DN(\Psi + t( {\boldsymbol \eta}_h-\Psi ))-DN(\Psi) )}_{\mathcal{L}(\X, \V^*)}\vertiii{\Psi - {\boldsymbol \eta}_h}_1 \dt \big)\vertiii{\Phi}_1. 
	\end{align}	
	Since $\Psi_{\mathbf{g}} \!\in\! \boldsymbol{\mathcal{{X}}}$, $\Psi - \Psi_{\mathbf{g}}\!\in\! \V.$ 
			For $\delta >0 $ small enough, the 
			 continuous inf-sup condition 
			\eqref{2.9} implies that 
			there exists $\Phi \!\in\! \V$ with $\vertiii{\Phi}_1 = 1$ such that $\displaystyle 	(\beta - \delta)\vertiii{ \Psi - \Psi_{\mathbf{g}}}_1 \leq \dual{DN(\Psi)(\Psi - \Psi_{\mathbf{g}}), \Phi}.$
			A triangle inequality, $\Psi - \Psi_{\mathbf{g}}=0$ on $\partial \Omega$  
	and	the last displayed inequality yield
			\begin{align*} 
			(\beta - \delta)	\vertiii{\Psi - {\boldsymbol \eta}_h}_h  \leq (\beta - \delta)(\vertiii{\Psi - \Psi_{\mathbf{g}}}_1 + \vertiii{ \Psi_{\mathbf{g}} - {\boldsymbol \eta}_h}_h) \lesssim  \dual{DN(\Psi)(\Psi - \Psi_{\mathbf{g}}), \Phi} + (\beta - \delta)\vertiii{ \Psi_{\mathbf{g}} - {\boldsymbol \eta}_h}_h. 
			\end{align*}
			Take $\delta \rightarrow 0$ to obtain $\displaystyle 	\beta \vertiii{\Psi - {\boldsymbol \eta}_h}_h  \lesssim  \dual{DN(\Psi)(\Psi - \Psi_{\mathbf{g}}), \Phi} + \beta \vertiii{ \Psi_{\mathbf{g}} - {\boldsymbol \eta}_h}_h. $
			A combination of \eqref{2.4}, the last displayed inequality and the definition of $\gamma$ plus Lemma \ref{Poincare type inequality} for $\vertiii{\Phi}_1=1$ leads to
			\begin{align*}
			C_4 \vertiii{\Psi- {\boldsymbol \eta}_h}_h \leq   \vertiii{N({\boldsymbol \eta}_h)}_{\V^*} + (1+\vertiii{\Psi - {\boldsymbol \eta}_h}_h +\vertiii{DN({\boldsymbol \eta}_h)}_{\mathcal{L}(\X, \V^*)})\vertiii{\Psi_{\mathbf{g}} - {\boldsymbol \eta}_h}_h +  \vertiii{\Psi - {\boldsymbol \eta}_h}^2_h,
			\end{align*}			
			where the constant $C_4$ depends on $\beta$, $\gamma$ and $ C_P. $
				For a choice of $R := \min\{{R_0,  C_4/2}\}$, use  $\vertiii{\Psi - {\boldsymbol \eta}_h}_h < C_4/2$ and $\vertiii{\Psi - {\boldsymbol \eta}_h}^2_h <C_4/2 \vertiii{\Psi - {\boldsymbol \eta}_h}_h $ in the second and third terms, respectively, in the right-hand side of the above inequality to obtain $\displaystyle C_4/2 \vertiii{\Psi- {\boldsymbol \eta}_h}_h \leq   \vertiii{N({\boldsymbol \eta}_h)}_{\V^*} + (1+C_4/2+\vertiii{DN({\boldsymbol \eta}_h)}_{\mathcal{L}(\X, \V^*)} )\vertiii{\Psi_{\mathbf{g}} - {\boldsymbol \eta}_h}_h,$
		 and this leads to the desired conclusion.
		\end{proof}	

		\noindent Next, the main result of this section is proved in the following text.
		\begin{proof}[\textbf{Proof of Theorem  \ref{reliability and efficiency}}]
			Theorem \ref{energy and $L^2$ norm error estimate} guarantees the existence of $R>0$ such that \eqref{thm_4.4} holds for a choice of $\boldsymbol{ \eta}_h=\Psi_h.$ Choose $\Psi_{\mathbf{g}}$ as in Lemma \ref{psi-psi_g estimate}. 
	\textit{A posteriori} reliability (resp. efficiency) estimate provides an upper bound (resp. lower bound) on the discretization error, up to a constant. 
	
	\medskip
	
	\noindent To establish the reliability,	Theorem \ref{thm 4.1} is utilized  and the term $	\vertiii{N(\Psi_h)}_{\V^*}$ is estimated first. 	Since $\V $ is a Hilbert space, there exists a $\Phi \!\in\! \V$ with $\vertiii{\Phi}_1= 1$ such that 
			\begin{align}\label{4.8.1}
			\vertiii{N(\Psi_h)}_{\V^*}= N(\Psi_h; \Phi)= N(\Psi_h; \Phi-{\textrm{I}^{SZ}_h} \Phi)+N(\Psi_h; \textrm{I}^{SZ}_h\Phi),
			\end{align}	
			where $\textrm{I}^{SZ}_h:\V \rightarrow \V_h$ is the Scott-Zhang interpolation in Lemma \ref{Scott-Zhang interpolation}. The second term in \eqref{4.8.1} can be rewritten using \eqref{discrete nonlinear problem} with test function $\textrm{I}^{SZ}_h\Phi$  (that vanishes on $\partial \Omega$) as
			\begin{align}\label{4.7}
			N(\Psi_h; \textrm{I}^{SZ}_h\Phi) = \dual{\Psi_h- \mathbf{g},\nabla (\textrm{I}^{SZ}_h\Phi) \nu }_{\partial \Omega}  &\lesssim \big(\sum_{E \in \mathcal{E}_h^{\partial}}    \frac{\sigma}{h_E} \vertiii{\Psi_h - \mathbf{g}}_{0,E}^2 \big)^{\frac{1}{2}}  \big(\sum_{E \in \mathcal{E}_h^{\partial}} \frac{h_E}{\sigma} \vertiii{\nabla(\textrm{I}^{SZ}_h\Phi) \nu_E}_{0,E}^2 \big)^{\frac{1}{2}}  \notag \\
			&  \lesssim  \big(\sum_{E \in \mathcal{E}_h^{\partial}}    \frac{\sigma}{h_E} \vertiii{\Psi_h - \mathbf{g}}_{0,E}^2 \big)^{\frac{1}{2}} \vertiii{\Phi}_1 = (\sum_{E \in \mathcal{E}_h^{\partial}} (\vartheta_E^{\partial})^2   )^{{1}/{2}},
			\end{align}
			where for $\vertiii{\Phi}_1=1$, a Cauchy-Schwarz inequality, Lemmas \ref{discrete trace inequality}($ii$) and \ref{Scott-Zhang interpolation} are utilized in the second and third steps.
			
			\smallskip
\noindent Apply integration by parts element-wise for $A(\Psi_h, \Phi-\textrm{I}^{SZ}_h\Phi) $ in the expression of $N(\Psi_h; \Phi-\textrm{I}^{SZ}_h \Phi)$, use $[\Phi-\textrm{I}^{SZ}_h \Phi]_E=0$ on $E \!\in\! \mathcal{E}_h^i$,  $\Phi-\textrm{I}^{SZ}_h \Phi = 0$ on $\partial \Omega$, $\Delta \Psi_{h}=0$ and recall the definition of the local terms $\boldsymbol{ \eta}_T := (2\epsilon^{-2}(\abs{\Psi_{h}}^2 -1)\Psi_{h})|_T$   defined on a triangle $T\!\in\! \mathcal{T}$ and $\boldsymbol{ \eta}_E := [\nabla \Psi_h \nu_E]_E$ on the edge $E$ of $T$. 			
	For $\vertiii{\Phi}_1=1$,	the above arguments, Cauchy-Schwarz inequality and  Lemma \ref{Scott-Zhang interpolation} lead to
			\begin{align} \label{4.10}
		&	N(\Psi_h; \Phi-\textrm{I}^{SZ}_h\Phi) =  	A(\Psi_h, \Phi-\textrm{I}^{SZ}_h\Phi)+ B(\Psi_h,\Psi_h,\Psi_h,  \Phi-\textrm{I}^{SZ}_h \Phi)+ C(\Psi_h,  \Phi-\textrm{I}^{SZ}_h \Phi)
			\notag \\  
			&= \sum_{T \in \mathcal{T}} \int_T \boldsymbol{ \eta}_T \cdot (\Phi-\textrm{I}^{SZ}_h \Phi) \dx +\sum_{E \in \mathcal{E}_h^i} \dual{ \boldsymbol{ \eta}_E, \Phi-\textrm{I}^{SZ}_h \Phi}_E 
			 \lesssim \big(\sum_{T \in \mathcal{T}} \vartheta_T^2+\sum_{E \in \mathcal{E}_h^i} (\vartheta_E^{i})^2 \big)^{\frac{1}{2}} \notag
			 \\&\quad \times \big(\sum_{T \in \mathcal{T}} h_T^{-2} \vertiii{\Phi-\textrm{I}^{SZ}_h \Phi}^2_{0,T}+\sum_{E \in \mathcal{E}_h^i} h_E^{-1} \vertiii{\Phi-\textrm{I}^{SZ}_h \Phi}_{0,E}^2 \big)^{\frac{1}{2}}
			 \lesssim \big(\sum_{T \in \mathcal{T}} \vartheta_T^2+\sum_{E \in \mathcal{E}_h^i} (\vartheta_E^{i})^2 \big)^{\frac{1}{2}},			
			\end{align}
		where $\vartheta_T^2 = h_T^2 \vertiii{2\epsilon^{-2}  (\abs{\Psi_h}^2-1)\Psi_h}^2_{0,T},$ and  $(\vartheta_E^i)^2 = h_E \vertiii{[\nabla \Psi_h \nu_E]_E}_{0,E}^2 \text{ for all } E \in \mathcal{E}_h^{i}$.	
	A use of \eqref{4.7}, \eqref{4.10} in \eqref{4.8.1} leads to the estimate of $\vertiii{N(\Psi_h)}_{\V^*}.$
	
	\medskip
\noindent The definition of $\vertiiih{\cdot}$ and Lemma \ref{psi-psi_g estimate}  yield
$$\vertiii{ \Psi_{\mathbf{g}} - \Psi_h}^2_h = \sum_{ T \in \mathcal{T}} \vertiii{\nabla (\Psi_{\mathbf{g}} - \Psi_{h})}_{0, T}^2 + \sum_{ E \in \mathcal{E}_h^{\partial} } \frac{\sigma}{h_E} \vertiii{\Psi_{h}- \mathbf{g} }^2_{0, E} \\ \lesssim  (\vartheta_{hot}^{\partial})^2+\sum_{ E \in \mathcal{E}_h^{\partial} }(\vartheta_E^{\partial })^2,$$ where $
(\vartheta_{hot}^{\partial})^2:= \sum_{ E \in \mathcal{E}_h^{\partial} }h_E^{-1} \vertiii{ \mathbf{g} - \mathbf{g}_h}_{0,E}^2 + h_E \vertiii{\nabla(\mathbf{g} - \mathbf{g}_h)}_{0,E}^2
$ (see Lemma  \ref{psi-psi_g estimate}). This leads to the bound for the second term in \eqref{thm_4.4} by  {\it higher order terms} (h.o.t.) \cite{Carsteen_higher_o_t}  that consist of (i) the errors arising due to the polynomial approximation of the boundary data $\mathbf{g}$ that  depends on the given data smoothness and 
(ii) the terms $\vertiii{DN(\Psi_h)}_{\mathcal{L}(\X, \V^*)} \vertiiih{\Psi_{\mathbf{g}}- \Psi_{h}}$. 
%
\medskip

	\noindent 	To establish the efficiency estimate, set $\Phi_h =  \Psi_h$ and $\boldsymbol{ \eta}_T = (2\epsilon^{-2}(\abs{\Psi_{h}}^2 -1)\Psi_{h})|_T $ on a triangle $T$ and $	\boldsymbol{ \eta}_E=[\nabla \Psi_h \nu_E]_E$ on a edge $E$ in Lemma  \ref{lem 3.1}. A use of the local efficiency estimates in Lemma \ref{lem 3.1}(i)  and $\displaystyle 	\sum_{ E \in \mathcal{E}^{\partial}_h} (	\vartheta_E^{\partial })^2=\sum_{ E \in \mathcal{E}^{\partial}_h} \frac{1}{h_E} \vertiii{\Psi_{h} - \mathbf{g}}_{0,E}^2 \leq \vertiii{\Psi-\Psi_{h}}_{h}^2$
			 establishes the lower bound in Theorem \ref{reliability and efficiency}.
		\end{proof}
	\begin{rem}
	 For $X_h= \{v_h \in {C}^0(\overline{\Omega}), v_h|_T \in P_p(T), \text{ for all } T \in \mathcal{T}\}$, that is, if we use higher order polynomials for the approximation, then $\vartheta_T^2 =  h_T^2 \vertiii{-\Delta \Psi_h+2\epsilon^{-2}  (\abs{\Psi_h}^2-1)\Psi_h}^2_{0,T}$ in \eqref{estimator_volume_interior_edge}.
	\end{rem}
	\section{Extension to discontinuous Galerkin FEM} \label{Extension to dG method}
	In this section, we extend the results  in Section \ref{Main results} to dGFEM.
	\medskip
	
\noindent The discrete space for dGFEM  consists of piecewise linear polynomials defined by $$\displaystyle X_{\dg}:=\{v \in L^2(\Omega): v|_T \in P_1(T) \text{ for all } T \in \mathcal{T}\},$$ 
and the mesh dependent norm 
$ \displaystyle \norm{v}^2_{\dg}:
=\sum_{ T \in \mathcal{T}} \int_T \abs{ \nabla v}^2 \dx + \sum_{E \in \mathcal{E}}  \frac{\sigma_\dg}{h_E} \int_{E} [v]_E^2 \ds,$
where $\sigma_\dg > 0$ is the penalty parameter. Let $\X_{\dg}\!: =X_{\dg} \times X_{\dg} $ be equipped with the product norm defined by $\vertiiidg{\Phi_\dg}^2= \norm{\varphi_1}^2_\dg +  \norm{\varphi_2}^2_\dg$ for all  $\Phi_\dg=(\varphi_1,\varphi_2) \in \X_{\dg}$. 
The dGFEM formulation corresponding to \eqref{continuous nonlinear} seeks $ \Psi_{\dg}\!\in\!\X_{\dg}$ such that for all $ \Phi_{\dg} \!\in\! \X_{\dg},$
\begin{align}\label{discrete nonlinear problem dG}
N_{\dg}(\Psi_\dg;\Phi_\dg):=A_{\dg}(\Psi_{\dg},\Phi_{\dg})+B(\Psi_{\dg},\Psi_{\dg},\Psi_{\dg},\Phi_{\dg})+C(\Psi_{\dg},\Phi_{\dg})-L_\dg(\Phi_{\dg})=0, 
\end{align}
where for $ \Theta=(\theta_1,\theta_2)  ,\, \Phi=(\varphi_1,\varphi_2)
\in \h^1(\mathcal{T}) $,  $A_{\dg}(\Theta,\Phi) :=a_{\dg}(\theta_1,\varphi_1)+a_{\dg}(\theta_2,\varphi_2),  
$  $L_\dg(\Phi_{\dg})=l^1_\dg(\varphi_1)+l^2_\dg(\varphi_2)$,
and for 
$\theta, \varphi \in H^1(\mathcal{T})$,  and for $-1 \leq \lambda \leq 1,$ 
\begin{align*}
&a_{\dg}(\theta,\varphi):=\sum_{ T \in \mathcal{T}} \int_T \nabla \theta \cdot \nabla \varphi \dx- \sum_{E \in \mathcal{E}} \dual{\{\frac{\partial \theta}{\partial \nu_E}\}_E,  [\varphi]_E}_E-\lambda\sum_{E \in \mathcal{E}} \dual{\{\frac{\partial \varphi }{\partial \nu_E}\}_E,  [\theta]_E}_E+ \sum_{E \in \mathcal{E}} \frac{\sigma_\dg}{h_E}  \dual{[\theta]_E, [\varphi]_E }_E\\&
	\text{  and }\,\,
l^i_\dg(\varphi):=
-  \sum_{E \in  \mathcal{E}_h^{\partial}}\dual{  g_i,  \frac{\partial \varphi}{\partial \nu_E}  }_{E} + \sum_{E \in  \mathcal{E}_h^{\partial}} \frac{\sigma_\dg}{h_E} \dual{g_i, \varphi }_E \text{ for } 1\leq i \leq 2. 
\end{align*}
The operators $B(\cdot,\cdot,\cdot,\cdot)$ and $ C(\cdot,\cdot)$ are as defined in Section \ref{weak formulation}. 

\medskip \noindent The proofs of results in this section follow on similar lines to the  results established in Sections \ref{A priori error estimate} and \ref{A posteriori error estimate} for the Nitsche's method. Hence the main resuts and the auxiliary results needed to establish them are stated and  parts of  proofs where ideas differ are highlighted.
\begin{lem}\label{2.3.17 dG}  {(Boundedness and coercivity of $A_{\rm dG}$)}\cite{ Prudhomme} \label{boundedness dG}
 For the choice of a sufficiently large parameter $\sigma_\dg$, there exists a 
	positive constant $\alpha_2 > 0$ such that 
	for	$\Theta_{\dg}, \Phi_{\dg} \!\in\! \X_\dg$,
	\begin{align*}
	A_{\dg}(\Theta_{\dg},\Phi_{\dg})\lesssim \vertiii{\Theta_{\dg}}_{\dg} \vertiii{\Phi_{\dg}}_{\dg}, \text{ and }  A_{\dg}(\Phi_{\dg}, \Phi_{\dg})  \geq \alpha_2 \vertiii{\Phi_{\dg}}_{\dg}^2, 
	\end{align*}
	where the hidden constant in $"\lesssim"$ is independent of $h.$
\end{lem}
	\begin{lem}\emph{(Interpolation estimate)}\cite{Prudhomme_Report} \label{Interpolation estimate dG}
	For ${\rm v} \in H^s(\Omega) \text{ with } s \geq 1$, there exists  ${\rm{I}_{\rm dG}v} \in X_{\rm dG}$ such that for any $T \in \mathcal{T},$ $\displaystyle \norm{\rm v-\rm{I}_{\rm dG}v }_{H^l(T)} \leq C_I h_T ^{s-l} \norm{\rm v}_{H^s(T)}$
	for $l=0,1$ where $C_I$ denotes  a generic interpolation constant  independent of $h$..
\end{lem}
\begin{lem} \emph{(Enrichment operator).} \cite{Poincare_Brenner,Karakashian} \label{enrichment operator dG}
	There exists an enrichment operator ${\rm E}_h : X_{\rm dG} \rightarrow V_h\subset H_0^1(\Omega)$, where $V_h $ is the Lagrange $P_1$ conforming finite element space associated with the triangulation $\mathcal{T}$ that satisfies the following properties. For any $\varphi_{\rm dG} \in X_{\rm dG}$, $\displaystyle (a) \sum_{ T \in \mathcal{T}} h_T^{-2}\norm{{\rm E}_h\varphi_{\rm dG} -\varphi_{\rm dG}}_{0,T}^2 + \norm{{\rm E}_h\varphi_{\rm dG}}_{1}^2 \leq C_{en_1} \norm{\varphi_{\rm dG}}_{\rm dG}^2,$ and $\displaystyle (b)\, \norm{{\rm E}_h\varphi_{\rm dG} - \varphi_{\rm dG}}_{\rm dG}^2 \leq C_{en_2}(\sum_{ E \in \mathcal{E}}\int_{ E}\frac{1}{h_E}[\varphi_{\rm dG}]_E^2 {\rm ds} ) ,$
	where $C_{en_1}$ and $C_{en_2}$ are positive constants independent of $h$.
\end{lem}
\begin{rem}(Modified local efficiency results)\label{modified local efficiency}
	Similar local efficiency results in Lemmas \ref{lem 3.1} -\ref{local efficiency for L2 estimate} hold for dGFEM with $\vertiiih{\cdot}$ is replaced by $\vertiiidg{\cdot}$ and the interpolation operator ${\rm I}_h$ replaced by ${\rm I}_\dg.$
\end{rem}
\noindent 
The discrete inf-sup condition corresponding to the perturbed bilinear form
$$\dual{DN_{\dg}(\textrm{I}_{\dg}\Psi)\Theta_{\dg}, \Phi_{\dg}} := A_{\dg}(\Theta_{\dg},\Phi_{\dg})+3B(\textrm{I}_{\dg}\Psi,\textrm{I}_{\dg}\Psi,\Theta_{\dg},\Phi_{\dg})+C(\Theta_{\dg},\Phi_{\dg})$$ is stated first. This is crucial in establishing the error estimates.
\begin{lem} \label{discrete inf sup dG} \emph{(Stability of perturbed bilinear form).}
Let $\Psi \! \in \! \boldsymbol{\mathcal{{X}}} \cap \mathbf{H}^{1+\alpha}(\Omega), \, 0<\alpha \le 1,$ be a regular solution of
	\eqref{continuous nonlinear} and ${{\rm I}_{\rm dG}}\Psi$ be its interpolant. For a sufficiently large $\sigma_\dg$ and a sufficiently small discretization parameter $h$, there exists a constant $\beta_1 $ such that $\displaystyle 	0< \beta_1 \leq \inf_{\substack{\Theta_{\dg} \in \X_{\dg} \\  \vertiii{\Theta_{\dg}}_{\dg}=1}} \sup_{\substack{\Phi_{\dg} \in \X_{\dg} \\  \vertiii{\Phi_{\dg}}_{\dg}=1}}\dual{DN_{\dg}({{\rm I}_{\dg}\Psi})\Theta_{\dg},\Phi_{\dg} }.$
\end{lem}
\begin{proof}
\noindent The proof follows along similar lines as the proofs of Lemma \ref{Lemma for discrete infsup}  and Theorem \ref{discrete inf sup}, except for the additional terms $$ \displaystyle \sum_{E \in \mathcal{E}^i_h}\dual{\{\nabla (\textrm{I}_{\dg}\boldsymbol \xi) \nu_E\}_E, [\Phi_{\dg}]_E}_{E} , \;  \displaystyle \sum_{E \in \mathcal{E}}\dual{\{\nabla \Phi_{\dg}  \nu_E\}_E, [\textrm{I}_{\dg}\boldsymbol \xi]_E}_{E} , \; \sum_{E \in \mathcal{E}}\frac{\sigma_\dg}{h_E}\dual{[\textrm{I}_{\dg}\boldsymbol \xi]_E , \;[\Phi_{\dg}]_E}_{E} $$ that appear in $A_{\dg}(\textrm{I}_{\dg}\boldsymbol \xi,\Phi_{\dg})-A(\boldsymbol \xi,\E_h\Phi_{\dg})$ (see \eqref{2.4.8.1}). Since $[{\rm E}_h\Phi_{\dg}]_{E}=0$ and $[\boldsymbol \xi]_E=0$ for all $E \in  \mathcal{E}_h^i $, the above displayed  terms are equal to 
$\displaystyle  \sum_{E \in \mathcal{E}^i_h}\dual{\{\nabla (\textrm{I}_{\dg}\boldsymbol \xi) \nu_E\}_E, [\Phi_{\dg} -\E_h\Phi_{\dg}]_E}_{E}, \sum_{E \in \mathcal{E}}\dual{\{\nabla \Phi_{\dg}  \nu_E\}_E, [\textrm{I}_{\dg}\boldsymbol \xi-\boldsymbol \xi]_E}_{E} , \sum_{E \in \mathcal{E}}\frac{\sigma_\dg}{h_E}\dual{[\textrm{I}_{\dg}\boldsymbol \xi-\boldsymbol \xi]_E , [\Phi_{\dg} -\E_h\Phi_{\dg}]_E}_{E},$
respectively. A Cauchy-Schwarz inequality and Lemmas \ref{Interpolation estimate dG},  \ref{enrichment operator dG} yield estimate for the above terms. The rest of the details are skipped for brevity.
\end{proof}
%
\begin{lem}\label{Lemma for ball to ball map dG}
	Let $\Psi$ be a regular solution of \eqref{continuous nonlinear} and ${\rm I}_\dg\Psi \in \X_{\dg}$ be its interpolant. Then, for any $\Phi_\dg \in \X_\dg$ with $ \vertiiidg{\Phi_{\dg} }=1,$ it holds that
	\begin{align*}
	&A_{\dg}({\rm{I}}_{\dg}\Psi , \Phi_{\dg} ) + B({\rm{I}}_{\dg}\Psi,{\rm{I}}_{\dg}\Psi,{\rm{I}}_{\dg}\Psi , \Phi_{\dg} )
	+ C({\rm{I}}_{\dg}\Psi , \Phi_{\dg} )  -L_{\dg}(\Phi_{\dg})\\& \qquad \lesssim h^{\alpha} (1+ \epsilon^{-2}h^{\alpha}(1+ \vertiii{ \Psi}_{1+\alpha}^2)) \vertiii{ \Psi}_{1+\alpha}.
	\end{align*}
\end{lem}
\begin{proof}
The proof of Lemma 3.15 is modified and the steps that are different are detailed.
 The definitions of $A_\dg(\cdot, \cdot) $ (with an integration by parts) and $L_\dg( \cdot) $ will lead to the inter-element jump and average terms in the identities  corresponding to \eqref{3.24} and \eqref{3.23}. Utilize $[\Psi]_E=0$ for all $E \in \mathcal{E}^i$ to rewrite these identities as follows.
	\begin{align*}
	A_{\dg}(\textrm{I}_{\dg}&\Psi , \Phi_{\dg}- \E_{h}\Phi_{\dg} )  -L_{\dg}(\Phi_{\dg}- {\rm E}_h\Phi_{\dg})= \sum_{ E \in \mathcal{E}_h^{i}}\dual{  [\nabla(\textrm{I}_{\dg}\Psi)\nu_E ]_E,\{\Phi_{\dg}- \E_h\Phi_{\dg}\}_E}_{E} \notag\\&+ \lambda\sum_{ E \in \mathcal{E}}\dual{[ \Psi-\textrm{I}_{\dg}\Psi]_E, \{\nabla (\Phi_{\dg}- \E_h\Phi_{\dg})\nu_E \}_E}_{E}  + \sum_{E \in \mathcal{E}} \frac{\sigma_\dg}{h_E}\dual{[\textrm{I}_{\dg}\Psi- \Psi]_E, [\Phi_{\dg}-  \E_h\Phi_{\dg}]_E}_E,\\& 
\hspace{-1.1 cm}	A_{\dg}(\textrm{I}_{\dg}\Psi ,  \E_h\Phi_{\dg} )-L_{\dg}( \E_h\Phi_{\dg}) 
 = A(\textrm{I}_{\dg}\Psi -\Psi,  \E_h\Phi_{\dg} )+\lambda \sum_{ E \in \mathcal{E}}\dual{[{\Psi}-\textrm{I}_{\dg}\Psi]_E, \{\nabla(\E_h\Phi_{\dg})\nu_E\}_E}_{E} \nonumber 
 \\
	& \quad\quad\quad\hspace{3 cm} - (B(\Psi ,\Psi ,\Psi ,  \E_h\Phi_{\dg}) + C(\Psi ,  \E_h\Phi_{\dg} )).
	\end{align*} 
	The inclusion of jump and average terms in the above displayed identities will modify  \eqref{3.22_new}  as
	\begin{align}\label{3.22_new dG}
	& A_{\dg}({\rm{I}}_{\dg}\Psi , \Phi_{\dg} ) + B({\rm{I}}_{\dg}\Psi,{\rm{I}}_{\dg}\Psi,{\rm{I}}_{\dg}\Psi , \Phi_{\dg} )
	+ C({\rm{I}}_{\dg}\Psi , \Phi_{\dg} )  -L_{\dg}(\Phi_{\dg})  \nonumber 
	 = 
	\sum_{T \in \mathcal{T}}\int_T \boldsymbol{ \eta}_T \cdot (\Phi_{\dg}-  \E_h\Phi_{\dg})\dx \\
	& \quad+ \sum_{ E \in \mathcal{E}_h^{i}}\dual{ \boldsymbol{ \eta}_E ,\{\Phi_{\dg}- \E_h\Phi_{\dg}\}_E}_{E} + (A(\textrm{I}_{\dg}\Psi -\Psi,  \E_h\Phi_{\dg} ) + C({\rm I}_{\dg}\Psi -\Psi, \E_h\Phi_{\dg})) \nonumber \\
	& \quad
	+ (B(\textrm{I}_{\dg}\Psi,\textrm{I}_{\dg}\Psi,\textrm{I}_{\dg}\Psi ,  \E_h\Phi_{\dg} )-B(\Psi ,\Psi ,\Psi ,  \E_h\Phi_{\dg}))	
	+\lambda \sum_{ E \in \mathcal{E}}\dual{[{\Psi}-\textrm{I}_{\dg}\Psi]_E, \{\nabla\Phi_{\dg} \nu_E \}_E}_{E}   \notag\\&\quad+
	\sum_{E \in \mathcal{E}} \frac{\sigma_\dg}{h_E}\dual{[\textrm{I}_{\dg}\Psi- \Psi]_E, [\Phi_{\dg}-  \E_h\Phi_{\dg}]_E}_E   :=T_1 + \cdots +T_6,	 
	\end{align}
	where $\boldsymbol{ \eta}_T:=( 2\epsilon^{-2}(\abs{\textrm{I}_{\dg}\Psi}^2 -1)\textrm{I}_{\dg}\Psi)|_T \text{ on } T \text{ and } \boldsymbol{ \eta}_E: =  [\nabla(\textrm{I}_{\dg}\Psi)\nu_E]_E \text{ on } E.$  The terms $T_1$ to $ T_4 $ are estimated in similar lines to the corresponding terms in Lemma \ref{Lemma for ball to ball map}. Apply Cauchy-Schwarz inequality, Lemma \ref{discrete trace inequality}, Lemmas \ref{Interpolation estimate dG}, \ref{enrichment operator dG} and $ \vertiiidg{\Phi_{\dg}}=1$  to  $T_5$ and $T_6$. 
	\begin{align*}
&T_5:=\sum_{E \in \mathcal{E}}\dual{[\Psi-\textrm{I}_{\dg}\Psi]_E, \{\nabla \Phi_{\dg} \nu_E \}_E}_{E}  \leq
\vertiiidg{\textrm{I}_{\dg}\Psi -\Psi} \vertiiidg{\Phi_{\dg}}
\lesssim h^\alpha \vertiii{\Psi}_{1+\alpha},
\\&
T_6:=\sum_{E \in \mathcal{E}} \frac{\sigma_\dg}{h_E}\dual{[\textrm{I}_{\dg}\Psi- \Psi]_E, [\Phi_{\dg}-  \E_h\Phi_{\dg}]_E}_E	
\leq  \vertiiidg{\textrm{I}_{\dg}\Psi -\Psi} \vertiiidg{\Phi_{\dg}-  \E_h\Phi_{\dg} }
\lesssim h^\alpha \vertiii{\Psi}_{1+\alpha}.
\end{align*}
A combination of the estimates lead to the desired result.
\end{proof}
\noindent The next abstract estimate is analogous to Theorem  \ref{thm 4.1} in Section \ref{A posteriori error estimate} and is useful to establish a reliable and  efficient {\it a posteriori} error estimate for dGFEM.  
\begin{lem}\label{Aposteriori abstract theorem}
	Let $\Psi $ be a regular solution to  \eqref{continuous nonlinear} and $\Psi_{\mathbf{g}} \! \in \! \boldsymbol{\mathcal{{X}}}$. 	 Then, $DN$ is locally Lipschitz continuous at $\Psi$, that is given $R_0>0$, $DN$ restricted to $B(\Psi,R_0)$ is Lipschitz continuous. Moreover, (a) $\displaystyle \gamma:= \sup_{\boldsymbol \eta \in B(\Psi, R_0)} \dfrac{\vertiii{DN(\boldsymbol \eta) - DN(\Psi)}_{\mathcal{L}(\X, \V^*)}}{\vertiii{\boldsymbol \eta- \Psi}_{\rm dG}} < \infty,$
	and (b) there exists a constant $R>0$ such that for all ${\boldsymbol \eta}_{\rm dG} \in B(\Psi, R)$, $\displaystyle  \vertiii{\Psi- {\boldsymbol \eta}_{\rm dG}}_\dg \lesssim \vertiii{N({\boldsymbol \eta}_{\rm dG})}_{\V^*}  +(1+\vertiii{DN({\boldsymbol \eta}_{\rm dG})}_{\mathcal{L}(\X, \V^*)})\vertiii{\Psi_{\mathbf{g}} - {\boldsymbol \eta}_{\rm dG}}_{\rm dG},$
	where the constant in $"\lesssim"$ depends on $\gamma$, continuous inf-sup constant  $\beta$ and Poincar\'e constant,
	and the nonlinear (resp. linearized) operator $N(\cdot)$ (resp. $DN(\cdot)$) is defined in \eqref{continuous nonlinear} (resp. \eqref{2.9}).
\end{lem} 
\noindent For each element $T$ and edge $E$, the volume and edge contributions to the estimators for dGFEM are
\begin{align}
&\vartheta_T^2: = h_T^2 \vertiii{-\Delta \Psi_{\rm dG} + 2\epsilon^{-2}  (\abs{\Psi_{\rm dG}}^2-1)\Psi_{\rm dG}}^2_{0,T}, \,\,\, (\vartheta_E^{\partial })^2:= \frac{1}{h_E} \vertiii{\Psi_{\rm dG} - \mathbf{g}}_{0,E}^2 \text{ for all } E \in \mathcal{E}_h^{\partial},\label{estimator_boundary_edge_dG}\\& \text{and }(\vartheta_E^i)^2 := h_E \vertiii{[\nabla \Psi_{\rm dG} \nu_E]_E}_{0,E}^2 +\frac{1}{h_E} \vertiii{[\Psi_{\rm dG}]_E}_{0,E}^2 \text{ for all } E \in \mathcal{E}_h^{i}\label{estimator_volume_interior_edge_dG} . 
\end{align}
Define the estimator
$ \displaystyle
\vartheta^2:= \sum_{ T \in \mathcal{T}}  \vartheta_T^2 +\sum_{ E \in \mathcal{E}_h^{i}} (\vartheta_E^i)^2 + \sum_{ E \in \mathcal{E}_h^{\partial} }(\vartheta_E^{\partial})^2.
$

\noindent The main result of this section is presented now.
\begin{thm}({\it A priori} and a posteriori error estimates)\label{ a priori and a posteriori error estimate dG}
	Let $\Psi$ be a regular solution of \eqref{continuous nonlinear}. 
	For a sufficiently large  penalty parameter $\sigma_\dg>0$ and a sufficiently small  discretization parameter $h$, there exists a unique solution $\Psi_{\rm dG}$ to the discrete problem \eqref{discrete nonlinear problem dG}  that approximates $\Psi$ such that	
	\begin{enumerate}
		\item (Energy norm estimate) $ 	\vertiii{\Psi-\Psi_{\rm dG}}_{\rm dG} \lesssim h^{\alpha}$, where $\Psi \! \in \! \boldsymbol{\mathcal{{X}}}\cap \mathbf{H}^{1+\alpha}(\Omega) ,$ $0 < \alpha \le 1$ is the index of elliptic regularity,
		\item (A posteriori estimates) There exist $h$-independent positive constants ${\rm C}_{\text{rel}}$ and ${\rm C}_{\text{eff}}$ such that
		\begin{align*}
		{\rm C}_{\text{eff}}	\vartheta \leq	\vertiii{\Psi - \Psi_{\rm dG}}_{\rm dG} \leq  {\rm C}_{\text{rel}} \big( \vartheta+  h.o.t \big),
		\end{align*} 
			where $h.o.t$ expresses terms of higher order.
	\end{enumerate}
\end{thm}
\begin{proof}
	The basic ideas of proofs of both {\it a priori} and {\it a posteriori} error estimates follow from Theorems \ref{energy and $L^2$ norm error estimate} and \ref{reliability and efficiency}. The modifications for the case of dGFEM are sketched for the sake of clarity. 
	
	\medskip

	\noindent \textit{1.} \textit{(Energy norm estimate):} 
The energy norm error estimate in {\it a priori } error analysis is proved using Brouwer's fixed point theorem. The non-linear map \cite{DGFEM} $\mu_{\rm dG}: \X_{\dg} \rightarrow \X_{\dg}$ is defined in this case as
	\begin{align*}
		\dual{DN_\dg({{\rm I}_{\rm dG}\Psi}) \mu_{\dg}(\Theta_{\rm dG}),\Phi_{\rm dG} }
	= 3B(\textrm{I}_{\rm dG}\Psi, \textrm{I}_{\rm dG}\Psi,\Theta_{\rm dG},\Phi_{\rm dG}) - B(\Theta_{\rm dG},\Theta_{\rm dG}, \Theta_{\rm dG},\Phi_{\rm dG}) +L_h(\Phi_{\rm dG}).
	\end{align*}
The proof now follows analogous to Theorem \ref{energy and $L^2$ norm error estimate}, using Lemmas \ref{discrete inf sup dG} and \ref{Lemma for ball to ball map dG}.

	\medskip
	
	\noindent \textit{2.} \textit{(A posteriori estimate):}    Lemma \ref{Aposteriori abstract theorem} and  techniques used in proof of Theorem  \ref{reliability and efficiency} lead to {\it a posteriori} estimates. The jump and average terms of $A_\dg(\cdot,\cdot)$ in the expansion of $N_\dg(\Psi_{\rm dG}; \textrm{I}^{SZ}_h\Phi)$ in \eqref{discrete nonlinear problem dG} will modify \eqref{4.7} to
	\begin{align*}
	N(\Psi_{\rm dG}; \textrm{I}^{SZ}_h\Phi)=\lambda \sum_{ E \in \mathcal{E}_h^{i}} \dual{[\Psi_{\rm dG}]_E,\{\nabla(\textrm{I}^{SZ}_h\Phi)  \nu_E\}_E}_E+ \lambda\sum_{ E \in \mathcal{E}_h^{\partial} } \dual{\Psi_{\rm dG}- \mathbf{g},\nabla (\textrm{I}^{SZ}_h\Phi) \nu_E }_{E}.
	\end{align*}
 The Cauchy-Schwarz inequality, Lemmas \ref{discrete trace inequality}($ii$) and \ref{Scott-Zhang interpolation} plus  $\vertiii{\Phi}_1=1$ yield 
	\begin{align*}
		N(\Psi_{\rm dG}; \textrm{I}^{SZ}_h\Phi)  \lesssim  \big(\sum_{E \in \mathcal{E}_h^{i}}    \frac{\sigma_\dg}{h_E} \vertiii{[\Psi_{\rm dG}]_E}_{0,E}^2+\sum_{E \in \mathcal{E}_h^{\partial}}    \frac{\sigma_\dg}{h_E} \vertiii{\Psi_\dg - \mathbf{g}}_0^2 \big)^{\frac{1}{2}} 
	\end{align*}
and	leads to interior edge estimator term $ \sum_{E \in \mathcal{E}_h^{\partial}} \frac{1}{h_E} \vertiii{[\Psi_{\rm dG}]_E}_{0,E}^2$.
	Moreover, a use of $[\Psi]_E=0$ for all $E \in \mathcal{E}_h^{i}$ shows $ \sum_{ E \in \mathcal{E}_h^{i}}\frac{1}{h_E} \vertiii{[\Psi_\dg]_E}_{0,E}^2 \leq \vertiiidg{\Psi_\dg - \Psi}$ and establishes the efficiency bound. The remaining part of the proof uses ideas similar to the proof of Theorem  \ref{reliability and efficiency}. 
\end{proof}

\section{Numerical results}\label{Numerical results for a priori error analysis}
In this section, we present  some numerical experiments that confirm the theoretical results obtained in Sections \ref{A priori error estimate}-\ref{Extension to dG method}, and illustrate the practical performances of the error indicators in adaptive mesh refinement for both Nitsche's method and dGFEM.
	\subsection{ Preliminaries}
\begin{itemize}
	\item The uniform refinement process divides each triangle in the triangulation of  $\bar{\Omega}$  into four similar triangles for subsequent mesh refinements using red refinement. 
	\item Let $e(n)$ and $h(n)$ (resp. $e({n-1})$ and $h({n-1})$) denote the error and the discretization parameter at the $n$-th (resp. $n-1$-th) level of uniform refinements, respectively. The convergence rate at $n$-th level is defined by $\displaystyle \alpha_n:=log(e(n)/e({n-1}))/log(h(n)/h({n-1}))$.
	\item The penalty parameters  $\sigma=\sigma_\dg=10$ is chosen for the numerical experiments.
	\item Numerical experiments are performed for different values of $\epsilon$ to illustrate the efficacy of the methods.
	\item
	Newton's method is employed to compute the approximated solutions of the discrete nonlinear problem \eqref{discrete nonlinear problem}. The Newton's iterates for Nitsche's method (see \cite{DGFEM} for Newton's iterates in dGFEM) are given by $\Psi^n_{h}, \,\, n= 1,2, \ldots$  
	\begin{align}\label{Newtons iterate}
	A_h(\Psi^n_{h}, \Phi_h) + 3B(\Psi^{n-1}_{h},\Psi^{n-1}_{h},\Psi^n_{h},\Phi_{h}) +C(\Psi^n_{h}, \Phi_h)= 2 B(\Psi^{n-1}_{h},\Psi^{n-1}_{h},\Psi^{n-1}_{h},\Phi_{h}) +L_h(\Phi_{h}).
	\end{align}
	The tolerance in the Newton's method is chosen as $10^{-8} $ in the numerical experiments unless mentioned otherwise.
\end{itemize}
\begin{rem}
	It can be established that the Newton iterates in \eqref{Newtons iterate} converge  quadratically to the discrete solution \cite[Theorem $3.6$]{DGFEM}. 
\end{rem} 
\begin{figure}[h]
	\centering
	\subfloat[]{\includegraphics[width=3.1cm,height=3.3cm]{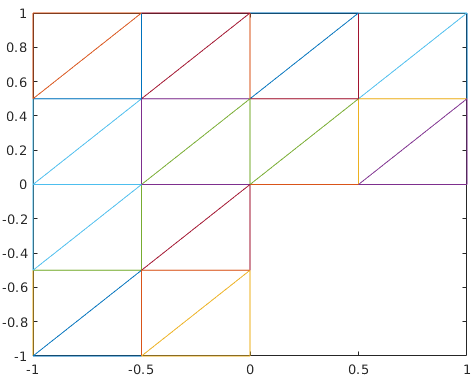}\label{triangulation2}} 
	\hspace{0.5cm}
	\subfloat[]{\includegraphics[width=3.1cm,height=3.3cm]{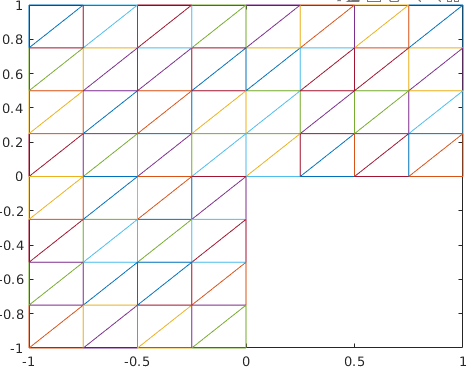}\label{triangulation3}}
	\hspace{0.5cm}
	\subfloat[]{\includegraphics[width=3.1cm,height=3.4cm]{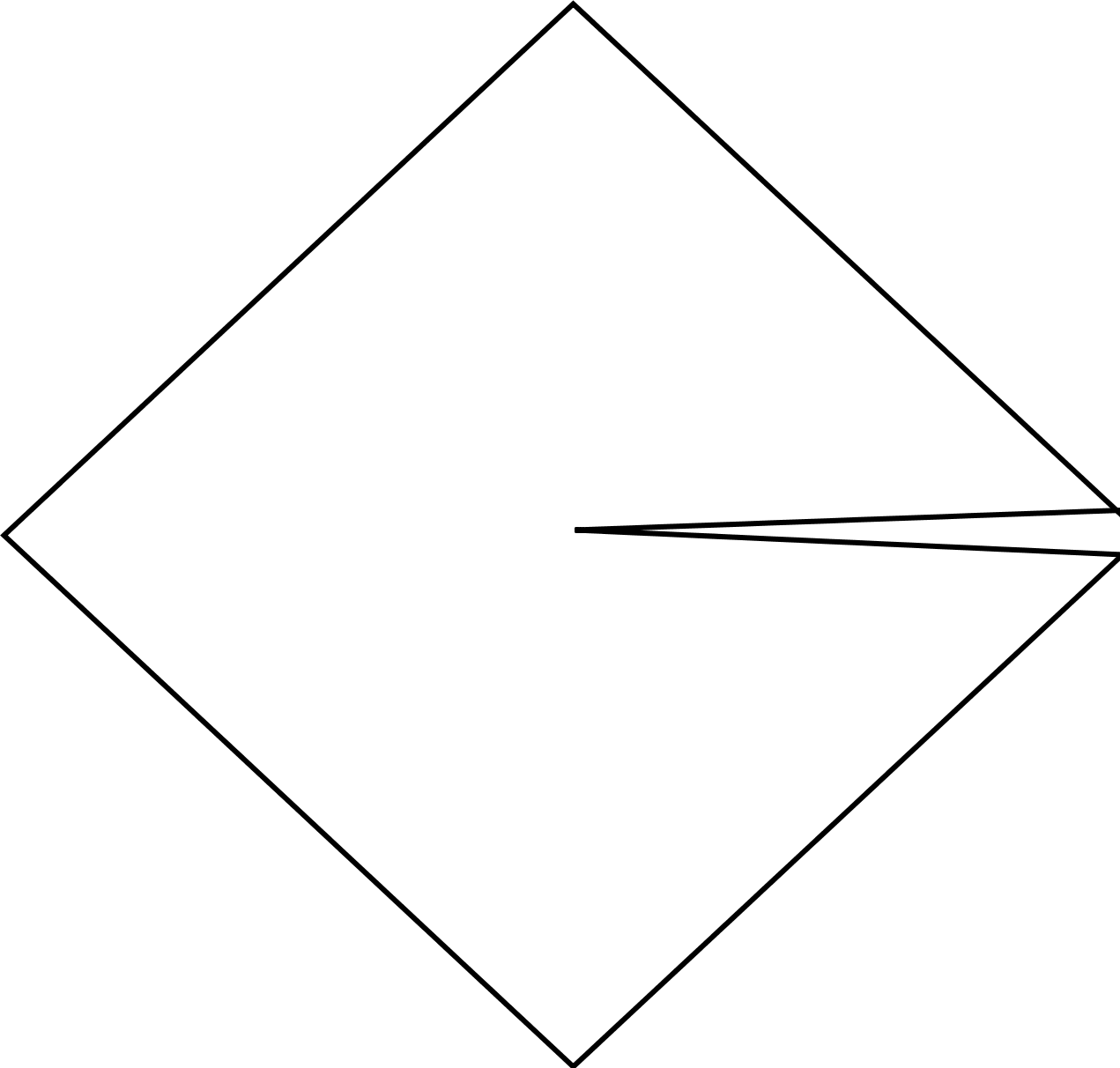}\label{Slit_domain}} 
	\caption{\protect\subref{triangulation2}Initial triangulation $\mathcal{T}_0$ of L-shape domain in Example \ref{Exact_sln_Lshape_apriori} and \protect\subref{triangulation3} its  uniform refinement.   \protect\subref{Slit_domain} Slit domain in Example \ref{Slit_domain_dG}}
	\label{Lshape_triangulations and slit domain discrete solution}
\end{figure}
\begin{figure}[h]
	\centering 
	\subfloat[]{\includegraphics[width=6cm,height=4.5cm]{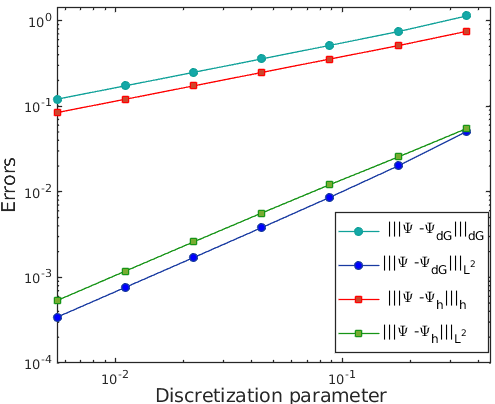}\label{Apriori_result_Lslape_domain}}
	\hspace{0.8cm}
	\subfloat[]{\includegraphics[width=6cm,height=4.5cm]{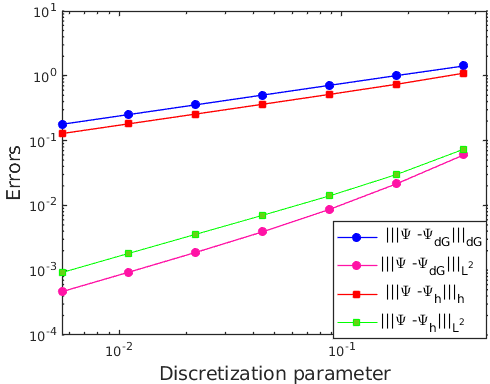}\label{Apriori_result_Slit_domain}}	
	\caption{ Convergence history ({\it a priori} estimates)  for Nitsche's method and dGFEM for \protect \subref{Apriori_result_Lslape_domain} Example \ref{Exact_sln_Lshape_apriori}, with $\epsilon=0.4$, and \protect \subref{Apriori_result_Slit_domain} Example \ref{Slit_domain_dG}, with $\epsilon=0.6$.}
	\label{Apriori_convergence_rates}
\end{figure}  
\subsection{ Example on a $L$-shaped domain}\label{Exact_sln_Lshape_apriori}
Consider \eqref{continuous nonlinear strong form} on a non-convex L-shape domain $\Omega = (-1, 1 )\times (-1,1) \setminus [0,1] \times [-1,0]$.
For the manufactured solution $	u= r^{{2}/{3}}\sin({2 \theta}/{3}),$ $v= r^{1/2}\sin(\theta/2)$,
where $(r, \theta)$ denote the system of polar coordinates, compute the corresponding right-hand side $\mathbf{f}$ and the non-homogeneous Dirichlet boundary condition $\mathbf{g}$. In this case, the exact solutions $\Psi \in \mathbf{H}^{1+1/2-\kappa}(\kappa \!> \! 0)$ \cite{grisvard1992}, and the theoretically expected rate of error reduction is $O(h^{1/2})$ and $O(h)$ in the energy norm and  $\mathbf{L}^2$ norm (see Theorems \ref{energy and $L^2$ norm error estimate} and \ref{ a priori and a posteriori error estimate dG}). 
 Figure \ref{triangulation2} and \ref{triangulation3} display the initial triangulation and its uniform refinement. The initial
guess for the Nitsche's method (resp. dGFEM) is chosen as $\Psi_h^0 \in \X_{h}$ (resp. $\Psi_\dg^0 \in \X_{\dg}$), where $\Psi_h^0$ solves $A_h(\Psi_h^0, \Phi_h)= L_h(\Phi_h) \text{  for all } \Phi_h \in \X_{h}$ (resp. $A_\dg(\Psi_\dg^0, \Phi_\dg)= L_\dg(\Phi_\dg) \text{  for all } \Phi_\dg \in \X_{\dg}$), 
and the linear form $L_h(\cdot)$ (resp. $L_\dg(\cdot)$) is modified to incorporate the information on $\mathbf{f}$. 
The approximations to the discrete solution to \eqref{discrete nonlinear problem} are obtained using the Newton's method defined in \eqref{Newtons iterate}. 
Figure \ref{Apriori_result_Lslape_domain} presents the convergence rates in energy norm and $\mathbf{L}^2$ norm,  with $\epsilon=0.4$, for both Nitsche's method and dGFEM. 
\subsection{ Benchmark example on unit square domain }\label{Luo_exampple}
Consider \eqref{continuous nonlinear} on a convex domain $\Omega=(0,1)\times (0,1)$ with the Dirichlet boundary condition \cite{MultistabilityApalachong} given by   
$\displaystyle 
\mathbf{g}=
\begin{cases} 
(\textit{T}_{d}(x),0) & \text{on}\,\,\,\, y=0 \,\,\,\,\text{and} \,\,\,\, y=1 \\
(- \textit{T}_d(y),0)  & \text{on}\,\,\,\, x=0 \,\,\,\,  \text{and} \,\,\,\,  x=1
\end{cases}
$,
where the parameter $d=3 \epsilon$ with $\epsilon = 0.02$ and the trapezoidal shape function $\textit{T}_d:[0,1]\rightarrow {\mathbb{R}}$ is defined by
$\displaystyle 
\textit{T}_d(t)=
\begin{cases} 
t/d, & 0 \leq t \leq d \\
1, &  d \leq  t \leq 1- d  \\
(1-t)/d, & 1- d \leq t \leq 1
\end{cases}
$. 	See \cite{MultistabilityApalachong, DGFEM} for details of construction of a suitable initial guess of Newton’s iterates in this example.
\begin{table}[H]
	\centering
	\begin{tabular}{c c  c  c  c c} 
		\hline
		$h$    &     Energy    &  $\vertiiih{\Psi_{h}^{n}-\Psi_{h}^{n-1}}$ &  Order &  $\vertiii{\Psi_{h}^{n}-\Psi_{h}^{n-1}}_{\mathbf{L}^2}$ &  Order  \\ [0.5ex] 
		\hline\hline
		0.0220   & 79.2401782   & 1.84603412 &-   & 0.92144678E-2   & -  \\
		0.0110   & 78.2908458  & 0.93391985 & 0.98305855 & 0.29215012E-2 & 1.65719096 \\
		0.0055    & 78.0391327  & 0.46906124 & 0.99352244 & 0.85867723E-3  & 1.76652203  \\
		0.0027    & 77.9747996 & 0.23603408 & 0.99078112 & 0.22993326E-3  & 1.90090075   \\[0.5ex]
		\hline \vspace{-0.4 cm}\\ 
		0.0220      & 87.9041386    & 1.86143892 &-   & 0.95840689E-2   & -  \\
		0.0110   & 86.9334857 &   0.94131381 & 0.98367059  & 0.29869842E-2 & 1.68194864  \\
		0.0055    & 86.6764846 & 0.47271729 & 0.99369812 & 0.87222395E-3  &   1.77591911  \\
		0.0027    & 86.6108341 &0.23784715 & 0.99094285 & 0.23307142E-3  &  1.90392648   \\
		\hline		
	\end{tabular}
	\caption{ Numerical energy, errors and convergence rates for  D1 and R1 solutions, respectively in energy and $\mathbf{L}^2$ norms for $\epsilon=0.02$.
	}
	\label{table:NitscheD1}		
\end{table}
\noindent Table \ref{table:NitscheD1} presents the computed energy, error in energy and $\mathbf{L}^2$ norms for numerical approximation of the diagonal D1 and rotated R1 solutions, respectively obtained using the Nitsche's method in \eqref{discrete nonlinear problem}. The orders of convergence agrees with the theoretical orders of convergence obtained in 
\cite{MultistabilityApalachong, DGFEM, Tsakonas}. For the corresponding results for dGFEM, see \cite[Tables 5, 6]{DGFEM}.
\subsection{ Example on a slit domain}\label{Slit_domain_dG}
Let $\Omega$ be the slit domain $\{(x,y) \in \mathbb{R}^2: \abs{x}+\abs{y}<1 \} \setminus ([0,1] \times \{0\})$ (see Figure \ref{Slit_domain}). Select the non-homogeneous Dirichlet boundary data $\mathbf{g}$ and the right-hand side $\mathbf{f}$ so that the manufactured solution is given by 
	$u(r, \theta)=v(r, \theta)= r^{1/2}\sin(\theta/2)- (1/2)(r\sin(\theta))^2$. 
	Figure \ref{Apriori_result_Slit_domain} presents the convergence history in energy norm, $\mathbf{L}^2$ norm,  with the parameter value $\epsilon=0.6$, for Nitsche's method and dGFEM. The rates are  approximately $0.5004$ (resp. $0.9846$) in energy norm ($\mathbf{L}^2$ norm). 
\subsection{ Adaptive mesh-refinement}
	\begin{itemize}
			\item For the adaptive refinement, the order of convergence of error and estimators are related to total number of unknowns ($\text{Ndof}(l)$). 
			 Let $e(l)$ and $\text{Ndof}(l)$ be the error and total number of unknowns at the $l-$th level refinement, respectively. The convergence rates are calculated as 
	\begin{align*}
	\text{Order}_e(l):= \frac{\log(e(l-1)/e(l))}{\log(\text{Ndof}(l)/\text{Ndof}(l-1))} \qquad \text{and} \qquad	\text{Order}_{\vartheta}(l):= \frac{\log(\vartheta(l-1)/\vartheta(l))}{\log(\text{Ndof}(l)/\text{Ndof}(l-1))}.
	\end{align*}

	\medskip
	
	\item Given an initial triangulation $\mathcal{T}_0,$ run the steps \textbf{SOLVE}, \textbf{ESTIMATE}, \textbf{MARK} and \textbf{REFINE}  successively for different levels $l=0,1,2,\ldots$ 
	
	\medskip
	
	\textbf{SOLVE} Compute the solution $\Psi_l:=\Psi_h$ (resp. $\Psi_l:=\Psi_\dg$) of the discrete problem \eqref{discrete nonlinear problem} (resp. \ref{discrete nonlinear problem dG})  for the triangulation $\mathcal{T}_l$. 
	
	\medskip
	
	\textbf{ESTIMATE} Calculate the error indicator $\displaystyle 	\varXi_{T,l}:=\big(\vartheta_T^2 + \sum_{ E \in \partial T \cap \mathcal{E}_h^{i}} (\vartheta^i_E)^2 + \sum_{E \in   \partial T \cap \mathcal{E}_h^{\partial }}(\vartheta_E^{\partial})^2 \big)^\frac{1}{2} $ for each element $T \in \mathcal{T}_l.$ 
	Recall the volume and edge estimators for Nitsche's method (resp. dGFEM) given by 
\eqref{estimator_volume_interior_edge}-\eqref{estimator_boundary_edge} (resp. \eqref{estimator_boundary_edge_dG}-\eqref{estimator_volume_interior_edge_dG}).
	
\medskip

	\textbf{MARK} 
	 For next refinement, choose the elements $T \in \mathcal{T}_l$  using D\"orfler marking such that  $ 0.3\sum_{ T \in {\mathcal{T}_l}}\varXi_{T, l}^2 \\ \leq \sum_{ T \in \tilde{\mathcal{T}}}\varXi_{T, l}^2$
	and collect those elements to construct a subset $\tilde{\mathcal{T}} \subset \mathcal{T}_l$.
	
	\medskip
	
	\textbf{REFINE} Compute the closure of $\tilde{\mathcal{T}}$ and use newest vertex bisection \cite{Stevenson} refinement strategy to construct the new triangulation $\mathcal{T}_{l+1}$.  
\end{itemize}
	Consider \eqref{continuous nonlinear strong form} in L-shaped domain (Figure \ref{triangulation2})  with the manufactured solution presented in Example \ref{Exact_sln_Lshape_apriori} and apply the adaptive refinement algorithm.
The estimator is modified as 	$\vartheta_T^2:= h_T^2 \vertiii{\mathbf{f}- 2\epsilon^{-2}(\abs{\Psi_h}^2- 1 )\Psi_h}_{0,T}^2$ (resp. $\vartheta_T^2:= h_T^2 \vertiii{\mathbf{f}- 2\epsilon^{-2}(\abs{\Psi_\dg}^2- 1 )\Psi_\dg}_{0,T}^2$) for Nitsche's method (resp. dGFEM) and this takes into account the effect of the non-zero right-hand side $\mathbf{f}$ calculated using the manufactured solution.
	\noindent 
	 Figures \ref{Exact_solution_uh} and \ref{Exact_solution_vh} (resp. Figures \ref{Exact_solution_uh_dG} and \ref{Exact_solution_vh_dG}) plot the discrete solutions, $u_h$ and $v_h$ of the Nitsche's method (resp. dGFEM), respectively, with the parameter $\epsilon=0.4$, and display the adaptive refinement near the vicinity of the re-entrant corner of the L-shaped domain.
	Table \ref{table:Error_estimator} displays the computational error, estimator and convergence rates for uniform and adaptive mesh refinement for $\epsilon=0.4$. 
	\noindent  It
	is observed from Table \ref{table:Error_estimator} that we have a suboptimal empirical convergence rate (calculated with respect to Ndof) of $0.25$ for uniform mesh refinement and an improved optimal empirical convergence rate of $0.5$ for adaptive mesh-refinement.  Further, in the adaptive refinement process,
	the number of mesh points required to achieve convergence is significantly reduced compared to uniform meshes and the convergence is faster than the uniform refinement process. Figure \ref{Error_estimators_ceff} displays the convergence behavior of the error and estimator along with the efficiency constant $C_{eff}$ plot, as a function of the total number of degrees of freedoms for $\epsilon=0.2, 0.8$. Here, $C_{eff}$ is the ratio between computed estimators and errors, which remains constant after the first few refinement levels. 

\medskip

	Figure \ref{Adaptive_solution_D1} displays the discrete solutions (diagonal D1 and rotated R1) and the adaptive mesh refinements in the square domain $\Omega = (0,1) \times (0,1)$, for Example \ref{Luo_exampple}.
Here, we observe adaptive mesh refinements near the defect points \cite{MultistabilityApalachong} of the domain (four corner points). Note that the estimator  tends to zero as the number of degrees of freedom (Ndof) increases. Figure \ref{Estimators} (resp. Figure \ref{Estimators_dG}) is the estimator vs Ndof plot for various values of $\epsilon$ for the diagonal, D1 solution obtained using Nitsche's method (resp. dGFEM). The tolerance used for Newton's method convergence is $10^{-6}$ and it is observed that for a fixed value of the discretization parameter $h$, the number of Newton iterations required for the convergence increases as the value of $\epsilon$ decreases. 
Observe that the rate of decay of the estimators is slower for smaller values of $\epsilon$ .
\begin{rem}
The $h$-$\epsilon$ dependency, discussed in \cite{DGFEM} has been reflected for adaptive refinement in this article, in terms of Ndof-$\epsilon$ dependency. It is observed in \cite{DGFEM} that errors
	are sensitive to the choice of discretization parameter as $\epsilon$ decreases. 
	\end{rem} 
\medskip

Figure \ref{Slit_domain_triangulation_Nitsche} (resp. Figure \ref{Slit_domain_triangulation_dG}) display the discrete solution corresponding to the Example \ref{Slit_domain_dG} and adaptive mesh-refinements, near the singularity at the origin for the parameter value $\epsilon=0.6$ (resp. $\epsilon=1$),  for Nitsche's method (resp. dGFEM). Figure \ref{Slit_domain_Error_estimator_Nitsche} (resp. Figure \ref{Slit_domain_Error_estimator}) shows the convergence history of errors in energy norm and estimators, for both uniform and adaptive refinements, for Nitsche's method (resp. dGFEM). 
 A sub-optimal empirical convergence rate $1/3$ for uniform refinement, and an improved empirical convergence rate $0.5 $, for adaptive mesh refinement, are obtained as a function of degrees of freedom for both Nitsche's method and dGFEM.
  \begin{figure}[H]
 	\centering
 	\subfloat[]{\includegraphics[width=3.5cm,height=3.5cm]{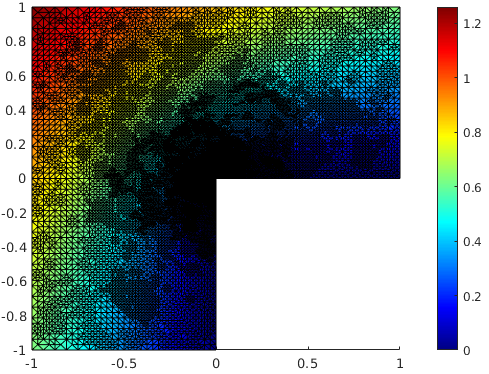}\label{Exact_solution_uh}} 
 	\hspace{0.2 cm}
 	\subfloat[]{\includegraphics[width=3.3cm,height=3.5cm]{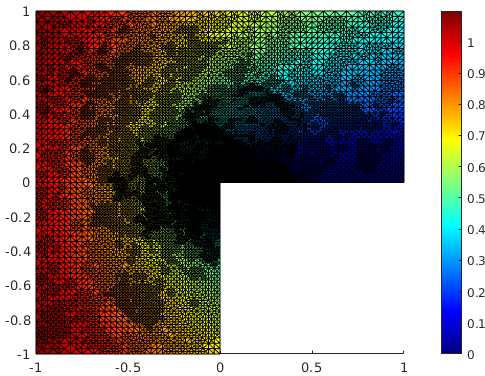}\label{Exact_solution_vh}}
 	\hspace{0.2 cm}
 	\subfloat[]{\includegraphics[width=3.3cm,height=3.5cm]{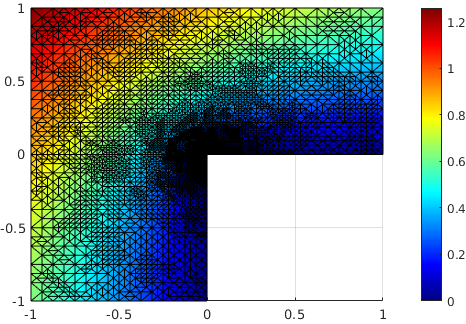}\label{Exact_solution_uh_dG}}
 	\hspace{0.2 cm}	
 	\subfloat[]{\includegraphics[width=3.3cm,height=3.5cm]{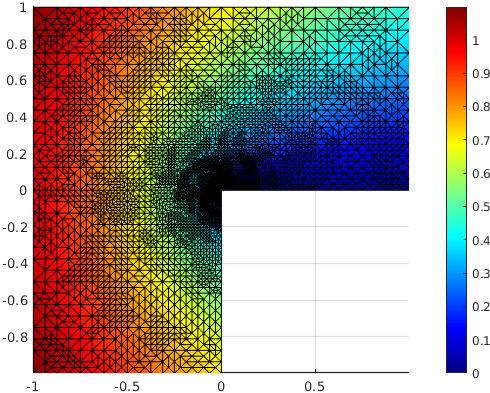}\label{Exact_solution_vh_dG}} 	
 	\caption{ Adaptive mesh refinements: \protect\subref{Exact_solution_uh} $u_h$, \protect\subref{Exact_solution_vh} $v_h$  for Nitsche's method and,  \protect\subref{Exact_solution_uh_dG} $u_h$, \protect\subref{Exact_solution_vh_dG} $v_h$  for dGFEM for  Example \ref{Exact_sln_Lshape_apriori} with $\epsilon=0.4$.  }
 	\label{Lshape_triangulations}
 \end{figure}
\begin{table}[H]
	\begin{tabular}{c c c c c c c c c c c} 
		\hline\\[-3.5mm] 
		\multicolumn{9}{c }{Uniform refinement \quad\quad\quad\quad\quad\quad\quad\quad\quad\quad\quad\quad Adaptive refinement} \\[0.5ex] 
		\hline \\[-4mm]
		Ndof    &  Error & Order$_{e}$ & $\vartheta$ &  Order$_\vartheta$ & 	&	Ndof    &  Error &  Order$_{e}$ &  $\vartheta$ &  Order$_\vartheta$  \\ [0.5ex] 
		\hline\hline\\[-4mm]
		42 &   0.74880  &-      & 2.14850 & - &	& 42  &0.74880  &-       & 2.14850  & - \\
		130 &  0.50988  &0.3401 & 1.42668 & 0.3623 &  &284  & 	0.20381 & 0.6808  & 0.74811 & 0.5519 \\
		450 &  0.35592  &0.2894 & 0.98209 & 0.3007 &  &1298 & 0.07579 & 0.6509 & 0.34115 & 0.5167 \\
		1666 & 0.24796  &0.2761 & 0.68290 & 0.2775 &  &2958 & 0.04662 & 0.5899 &0.22668  &	0.4962 \\
		6402&  0.17274  &0.2685 & 0.47546 & 0.2689 &  &6732 & 0.02982 & 0.5434 &0.15020 &0.5004 \\
		25090& 0.12053  &0.2634 & 0.33121 & 0.2647 &  &14792 &0.01956 & 0.5356  &0.10090 & 0.5053\\
		99330& 0.08426  &0.2601 & 0.23099 & 0.2618 &  &21936 &	0.01597  &0.5146  &0.08303  &0.4945 \\
		395266&0.05902  & 0.2577& 0.16138 & 0.2596 &  &47326 &0.01071  &0.5194  &0.05621  &0.5074 \\
		\hline		
	\end{tabular}
	\caption{Numerical errors, estimators and experimental convergence rates for uniform and adaptive mesh refinement for $\epsilon=0.4.$}
	\label{table:Error_estimator}	
\end{table}	
 \begin{figure}[H]
	\centering 
	\includegraphics[width=9cm,height=6.5cm]{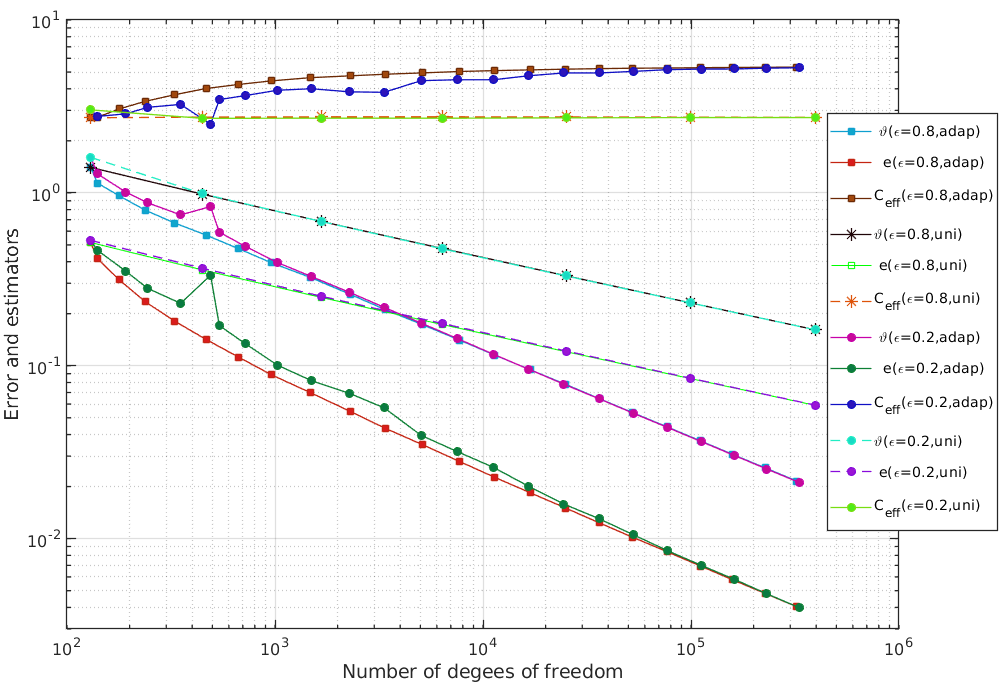}\label{error_estimator_rel_constant} 	
	\caption{ Ndof versus $e$, $\vartheta$ and $C_{eff}$ for L-shape domain in Example \ref{Exact_sln_Lshape_apriori}.
	}
	\label{Error_estimators_ceff}
\end{figure}
\begin{figure}[H]
	\centering
	\subfloat[]{\includegraphics[width=3.7cm,height=3.4cm]{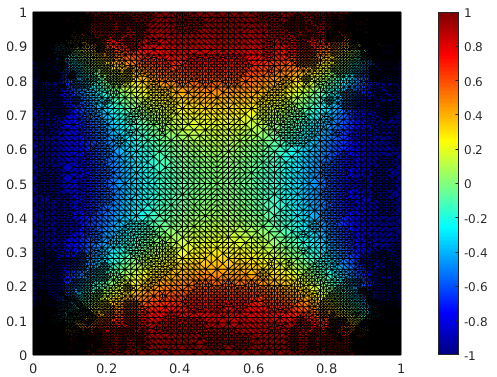}\label{uD1}} 
	\subfloat[]{\includegraphics[width=3.7cm,height=3.4cm]{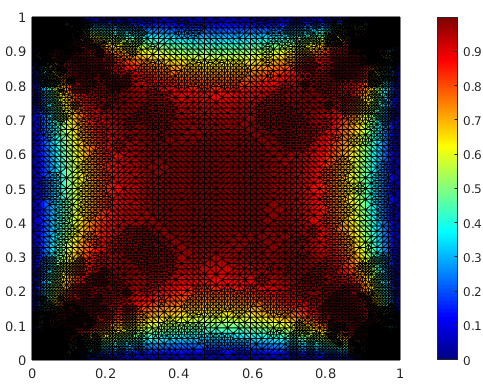}\label{vD1}}
	\subfloat[]{\includegraphics[width=3.7cm,height=3.4cm]{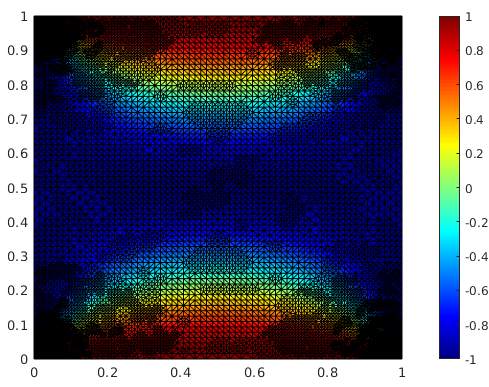}\label{uR1}} 
	\subfloat[]{\includegraphics[width=3.7cm,height=3.4cm]{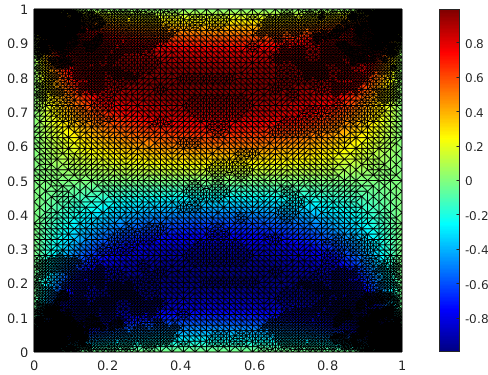}\label{vR1}}
	\caption{Adaptive mesh refinements: \protect \subref{uD1} $u_h$, \protect \subref{vD1} $v_h$ for D1 solution. Adaptive mesh refinements: \protect\subref{uR1} $u_h$, \protect\subref{vR1} $v_{h}$ for R1 solution of Example \ref{Luo_exampple}. 
	}
	\label{Adaptive_solution_D1}
\end{figure}  
	\begin{figure}[H]
	\centering 
	\subfloat[]{\includegraphics[width=7cm,height=6cm]{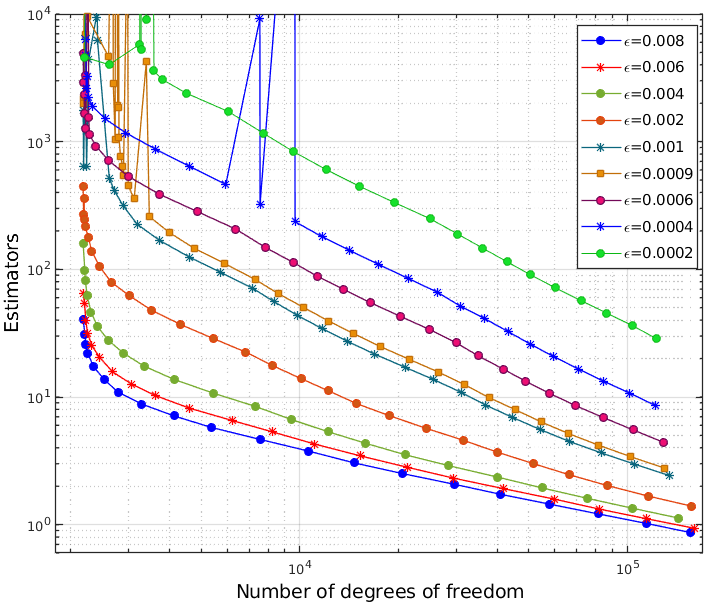}\label{Estimators}}
	\hspace{0.4cm}
	\subfloat[]{\includegraphics[width=7cm,height=6cm]{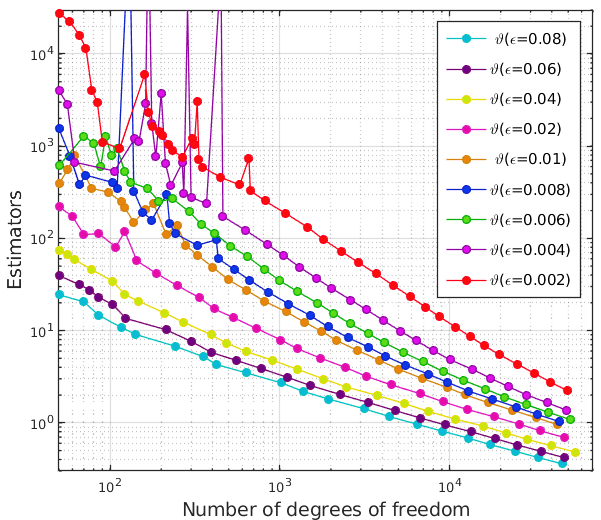}\label{Estimators_dG}}	
	\caption{ Ndof vs estimators plot for various  values of $\epsilon$ in square domain Example \ref{Luo_exampple} for \protect \subref{Estimators}  Nitsche's method and \protect \subref{Estimators_dG} dGFEM. 
	}
	\label{Error_estimators}
\end{figure}  
\begin{figure}[H]
	\centering
	\subfloat[]{\includegraphics[width=4.8cm,height=3.7cm]{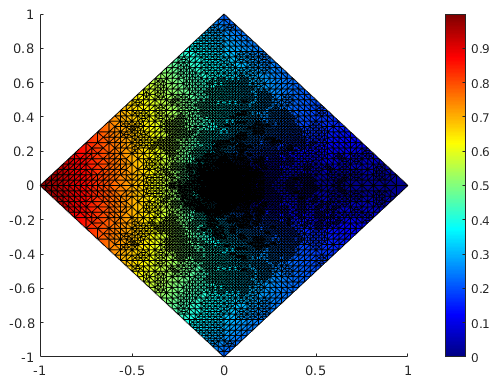}\label{Slit_domain_triangulation_Nitsche}}
	\hspace{0.3 cm}	
	\subfloat[]{\includegraphics[width=4.8cm,height=3.7cm]{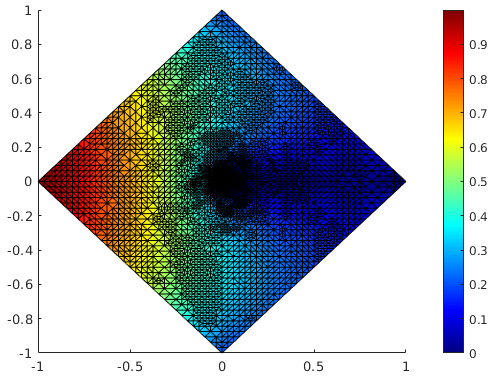}\label{Slit_domain_triangulation_dG}} 		
	\caption{ Adaptive mesh refinements: \protect\subref{Slit_domain_triangulation_Nitsche} $u_h$ for Nitsche's method with $\epsilon=0.6$ . \protect\subref{Slit_domain_triangulation_dG} $u_h$ for dGFEM with $\epsilon$=$1.$  }
	\label{Slit_omain}
\end{figure}
\begin{figure}[H]
	\centering
		\subfloat[]{\includegraphics[width=6.8cm,height=5.5cm]{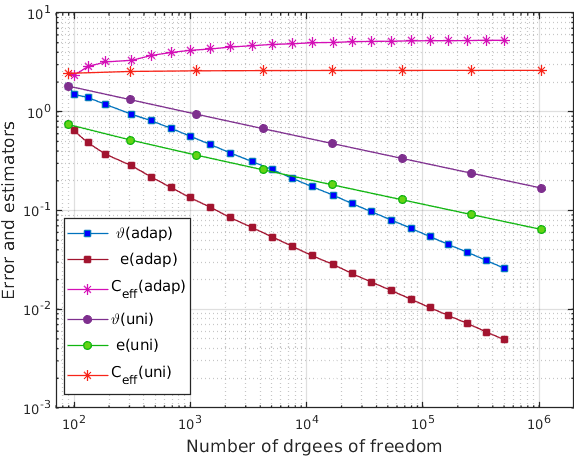}\label{Slit_domain_Error_estimator_Nitsche}}
	\hspace{0.4cm}
	\subfloat[]{\includegraphics[width=6.8cm,height=5.5cm]{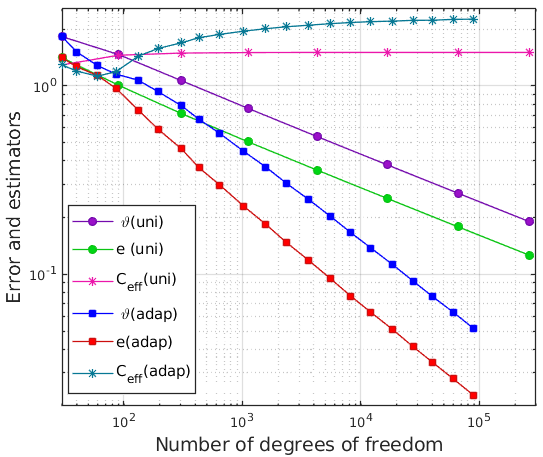}\label{Slit_domain_Error_estimator}}	
	\hspace{-0.1cm}		
	\caption{ Ndof versus $e,$ $\vartheta$ and $C_{eff}$ for \protect\subref{Slit_domain_Error_estimator_Nitsche}Nitsche's method with $\epsilon=0.6$ and \protect\subref{Slit_domain_Error_estimator}  dGFEM with $\epsilon=1$. }
	\label{Slit_domain_dG_solution_estimator_Nitsche}
\end{figure}

\section{Conclusions}\label{sec:conclusions} This manuscript focuses on {\it  a priori} and {\it a posteriori} error analysis for solutions with milder regularity than $\h^2$, and such solutions of lesser regularity are relevant, for example, in polygonal domains or domains with re-entrant corners that have boundary conditions of lesser regularity. We use Nitsche's method for our analysis; the {\it a priori} error analysis relies on medius analysis and these techniques are  extended to  dGFEM. In \cite{DGFEM},  $h$-$\epsilon$ dependent error estimates for $\h^2(\Omega)$ regular solutions are obtained, and this follows from an $\epsilon$ independent bound for the exact solution $\vertiii{\Psi}_2$, as established in \cite{Bethuel}. It is not clear if such estimates are feasible for exact solutions with milder regularity,  $\h^{1+\alpha}(\Omega),$ $0 < \alpha \le 1$,  since we do not have $\epsilon$-independent bounds for $\vertiii{\Psi}_{1+\alpha}$  at hand.
\noindent It may be possible to obtain such bounds for certain model problems, which would allow $h-\epsilon$ dependent estimates. The methods in this paper will extend to modelling problems with weak anchoring or surface energies, which would translate to a Robin-type boundary condition; some dynamical models e.g. Allen-Cahn type evolution equations, stochastic versions of the Ginzburg-Landau system (\ref{continuous nonlinear strong form}); modelling problems for composite material, such as ferronematics, which have both nematic and polar order etc. The overarching aim is to propose optimal estimates for the discretization parameter and number of degrees of freedom, for systems of second-order elliptic partial differential equations with lower order polynomial non-linearities, as a function of the model parameters e.g. $\epsilon$, and use these estimates for powerful new computational algorithms.
		\section*{Acknowledgements}
		R.M. gratefully acknowledges support from institute Ph.D. fellowship and  N.N. gratefully acknowledges the support by  DST SERB MATRICS grant MTR/2017/000 199. A.M acknowledges support from the DST-UKIERI and British Council funded project on "Theoretical and Experimental Studies of Suspensions of Magnetic Nanoparticles, their Applications and Generalizations" and support from IIT Bombay, and a Visiting Professorship from the University of Bath.

	\bibliographystyle{amsplain}
\bibliography{ReferencesLDG}
\begin{appendices}
\section{Appendix}
\noindent This section discusses the proofs of the local efficiency results in Lemmas \ref{lem 3.1}-\ref{local efficiency for L2 estimate}. The local cut off functions play an  important role to establish the local efficiency results. Consider the interior bubble function \cite{Oden, Verfurth} $\widehat{b}_T = 27\widehat{\lambda}_1 \widehat{\lambda}_2 \widehat{\lambda}_3 $ supported on a reference triangle $\widehat{T}$ with the barycentric coordinate functions $\widehat{\lambda}_1 , \widehat{\lambda}_2, \widehat{\lambda}_3$. For $T \in \mathcal{T}, $  
let $\mathcal{F}_T: \widehat{T} \rightarrow T  $ be a continuous, affine and invertible transformation. Define the bubble function on the element $T$ by $b_T = \widehat{b}_T \circ \mathcal{F}_T^{-1} $. Three edge bubble functions on the reference triangle $\widehat{T}$ are given by $\widehat{b}_1= 4 \widehat{\lambda}_2 \widehat{\lambda}_3$, $\widehat{b}_2= 4 \widehat{\lambda}_1 \widehat{\lambda}_3$ and $\widehat{b}_3= 4 \widehat{\lambda}_1 \widehat{\lambda}_2$. On the edge $E$ of any triangle $T \in \mathcal{T}$, define the edge bubble function to be $b_E: =  \widehat{b}_E \circ  \mathcal{F}_T^{-1}$, where $\widehat{b}_E$ is the corresponding edge bubble function on $\widehat{T}$. Here, $b_E$ is supported on the pair of triangles sharing the edge $E.$ 
\begin{lem}\cite{Oden,Verfurth}\label{lem 3.4}
	Let $\widehat{P} \subset H^1(\widehat{T})$ be a finite dimensional subspace on the reference triangle $\widehat{T}$ and consider $P= \{\widehat{v} \circ \mathcal{F}_T^{-1}:\widehat{v} \in  \widehat{P} \}$ to be the finite dimensional space of functions defined on $T$. Then the following inverse estimates hold for all  $v \in P$, 
	\begin{align}
	\norm{v}^2_{L^2(T)} \lesssim \int_T b_T v^2 \dx \lesssim \norm{v}^2_{L^2(T)}, \quad 
	\norm{ v}_{L^2(T)} \lesssim \norm{b_Tv}_{L^2(T)} +	 h_T\norm{\nabla (b_Tv)}_{L^2(T)} \lesssim \norm{ v}_{L^2(T)}. \label{b_T estimates} 
	\end{align}
	Let $E \subset  \partial T$ be an edge and $b_E$ be the corresponding edge bubble function supported on the patch of triangles $\omega_E$ sharing the edge $E$. Let $P(E)$ be the finite dimensional space of functions defined on $E$ obtained by mapping $\widehat{P}(\widehat{E})\subset H^1(\widehat{E}).$ Then for all $v \in P(E)$,
	\begin{align} 
	\norm{v}^2_{L^2(E)} \lesssim \int_E b_E v^2 \dx \lesssim \norm{v}^2_{L^2(E)}, \quad\,\,\, h_E^{ -\frac{1}{2}}	\norm{b_Ev}_{L^2(\omega_E)}+ h_E^{ \frac{1}{2}}\norm{\nabla (b_Ev)}_{L^2(\omega_E)} \lesssim \norm{ v}_{L^2(E)}, \label{b_E estimates}
	\end{align}
	where the hidden constants in $"\lesssim"$ are independent of $h_T$ and $h_E$.
\end{lem}
\begin{proof}[\textbf{Proof of Lemma \ref{lem 3.1}}]
	$(i)$	Let $T \in \mathcal{T}$ be arbitrary and $b_T$ be the interior bubble function supported on the triangle $T$.
		Choose $\displaystyle 
	\boldsymbol{\rho}_T:=
	\left\{
	\begin{array}{l}
	\big(-\Delta \Phi_{h}+2\epsilon^{-2}(\abs{\Phi_{h}}^2 -1)\Phi_{h}\big)b_T \quad \text{in } T\\
	0 \quad \text{in } \Omega \setminus  T
	\end{array}
	\right .
	$, utilize  \eqref{b_T estimates}, \eqref{continuous nonlinear} with $\Phi:=\boldsymbol{\rho}_T$ and apply an integration by parts for the first term (which is a zero term) on the right-hand side below  to obtain
	\begin{align}\label{lem 4.4.1}
	&	\vertiii{\boldsymbol{ \eta}_T}_{0,T}^2
	\lesssim \int_T \big(-\Delta \Phi_{h}+2\epsilon^{-2}(\abs{\Phi_{h}}^2 -1)\Phi_{h}\big) \cdot \boldsymbol{\rho}_T \dx 
	\notag\\&
	=  A_T(\Phi_{h}-\Psi, \boldsymbol{\rho}_T ) +(B_T(\Phi_{h},\Phi_{h},\Phi_{h},\boldsymbol{\rho}_T )-B_T(\Psi,\Psi,\Psi,\boldsymbol{\rho}_T ))+C_T(\Phi_{h}-\Psi,\boldsymbol{\rho}_T ) .
	\end{align}	
	Together with H\"older's inequality,  Lemma \ref{boundedness} and \eqref{b_T estimates},  the terms on the right-hand side of \eqref{lem 4.4.1} are estimated as 
	\begin{align}
	A_T(\Phi_{h}-\Psi, \boldsymbol{\rho}_T )\lesssim \vertiii{\nabla (\Phi_{h}-\Psi)}_{0,T}\vertiii{\nabla \boldsymbol{\rho}_T}_{0,T}\lesssim \vertiii{\Psi-\Phi_{h}}_{1,T} h_T^{-1}\vertiii{\boldsymbol{ \eta}_T}_{0,T}.\label{estimation_A}
	\\
	C_T(\Phi_{h}-\Psi,\boldsymbol{\rho}_T )\lesssim \epsilon^{-2} \vertiii{\Phi_h-\Psi}_{0,T}\vertiii{\boldsymbol{\rho}_T}_{0,T}\lesssim \epsilon^{-2} \vertiii{\Psi-\Phi_{h}}_{0,T} \vertiii{\boldsymbol{ \eta}_T}_{0,T},\label{estimation_C}
	\end{align}
	\begin{align}
	B_T(\Phi_{h},\Phi_{h},\Phi_{h},\boldsymbol{\rho}_T )-B_T(\Psi,\Psi,\Psi,\boldsymbol{\rho}_T )
	\lesssim &\epsilon^{-2}\vertiii{\Psi-\Phi_{h}}_{1,T}(\vertiii{\Psi-\Phi_{h}}_{1,T} (\vertiii{\Phi_h}_{1,T}+ \vertiii{\Psi}_{1,T})\notag \\& + \vertiii{\Psi}_{1,T}^2) h_T^{-1}\vertiii{\boldsymbol{ \eta}_T}_{0,T}.
	\label{3.33}
	\end{align}
	A combination of the above three displayed estimates in \eqref{lem 4.4.1} plus Lemma \ref{lem 3.4} establishes
	\begin{align}\label{estimate eta_T} 
	h_T\vertiii{\boldsymbol{ \eta}_T}_{0,T}  \lesssim	\vertiii{\Psi - \Phi_{h}}_{h, T} (1+ \epsilon^{-2} (1+ \vertiii{\Psi}_{1, T}^2 +\vertiii{\Psi-\Phi_{h}}_{h, T} (\vertiii{\Phi_h}_{1, T}+ \vertiii{\Psi}_{1,T}) )).
	\end{align}
	To find the estimate corresponding to $\boldsymbol{ \eta}_E$, consider the edge bubble function $b_E$ supported on the patch of triangles $\omega_E$ sharing the edge $E$. 
	Define 
	$\displaystyle
	\boldsymbol{\rho}_E:=\left\{
	\begin{array}{l}
	[\nabla \Phi_h \nu]b_E \quad \text{in } \omega_E\\
	0 \quad \text{in } \Omega \setminus \omega_E
	\end{array}
	\right.
	$
	and  use \eqref{b_E estimates}, $[\boldsymbol{\rho}_E] =  0$ for $E \in \mathcal{E}_h^i $ and an integration by parts to obtain
	\begin{align}\label{3.11.1}
	\vertiii{\boldsymbol{ \eta}_E}_{0,E}^2 & \lesssim \int_{E}[\nabla\Phi_{h} \nu] \cdot \boldsymbol{\rho}_E\ds 
	= \int_{E } [\nabla\Phi_{h} \nu]  \cdot  \{\ \boldsymbol{\rho}_E \} \ds +\int_{E} \{ \nabla\Phi_{h} \nu \} \cdot [\boldsymbol{\rho}_E] \ds \notag
	\\& =
	\sum_{T \in \omega_E} \int_T (\Delta \Phi_{h} \cdot \boldsymbol{\rho}_E + \nabla \Phi_{h} \cdot \nabla \boldsymbol{\rho}_E) \dx .
	\end{align} 
Add and subtract $\sum_{T \in \omega_E} \int_T2\epsilon^{-2}(\abs{\Phi_{h}}^2 -1)\Phi_{h} \cdot  \boldsymbol{\rho}_E \dx
	$ in the right-hand side of \eqref{3.11.1} to rewrite the expression with the help of $\boldsymbol{ \eta}_T= \Delta \Phi_{h}-2\epsilon^{-2}(\abs{\Phi_{h}}^2 -1)\Phi_{h}$ (with a $-\Delta \Phi_{h}=0$ added). The expression \eqref{continuous nonlinear} with $\Phi =  \boldsymbol{\rho}_E$, a re-grouping of terms and H\"older's inequality lead to 
	\begin{align}\label{lem 4.4.11}
	\vertiii{\boldsymbol{ \eta}_E}_{0,E}^2   &\lesssim 
	(\sum_{T \in \omega_E} \vertiii{\boldsymbol{ \eta}_T}_{0,T}^2)^{\frac{1}{2}}(\sum_{T \in \omega_E} \vertiii{\boldsymbol{\rho}_E }_{0,T}^2)^{\frac{1}{2}} +	\sum_{T \in \omega_E} (A_T(\Phi_{h}-\Psi, \boldsymbol{\rho}_E) +C_T(\Phi_{h}-\Psi, \boldsymbol{\rho}_E)\notag\\&\quad + (B_T( \Phi_{h}, \Phi_{h},  \Phi_{h},\boldsymbol{\rho}_E)-B_T( \Psi, \Psi, \Psi,\boldsymbol{\rho}_E))) .
	\end{align} 
	A combination of H\"older's inequality, Lemma \ref{boundedness} and \eqref{b_E estimates} yields
	\begin{align}
	\sum_{T \in \omega_E} A_T(\Phi_{h}- \Psi, \boldsymbol{\rho}_E)  \lesssim   \sum_{T \in \omega_E}\vertiii{\nabla(\Psi -\Phi_{h})}_{0,T} \vertiii{ \nabla \boldsymbol{\rho}_E}_{0,T}\lesssim h_E^{-\frac{1}{2}} \vertiii{ \boldsymbol{ \eta}_E}_{0,E} \vertiii{\nabla(\Psi -\Phi_{h})}_{0,\omega_E},\label{A_E}
	\\
	\sum_{T \in \omega_E} C_T(\Phi_{h}- \Psi, \boldsymbol{\rho}_E)\lesssim \epsilon^{-2}\sum_{T \in \omega_E} \vertiii{\Psi-\Phi_{h}}_{0,T} \vertiii{ \boldsymbol{\rho}_E}_{0,T}  \lesssim \epsilon^{-2}h_E^{\frac{1}{2}}\vertiii{ \boldsymbol{ \eta}_E}_{0,E}  \vertiii{\Psi-\Phi_{h}}_{0,\omega_E},\label{C_E}
	\end{align}	
		\begin{align}\label{lem 4.4.14}
	\sum_{T \in \omega_E} (B_T( \Phi_{h}, \Phi_{h},  \Phi_{h},\boldsymbol{\rho}_E)-B_T( \Psi, \Psi,  \Psi,\boldsymbol{\rho}_E) ) \lesssim & \epsilon^{-2} h_E^{-\frac{1}{2}} \vertiii{ \boldsymbol{ \eta}_E}_{0,E}\sum_{T \in \omega_E}\vertiii{\Psi-\Phi_{h}}_{1, T}(\vertiii{\Psi-\Phi_{h}}_{1, T} (\vertiii{\Phi_h}_{1, T}  \notag \\&+\vertiii{\Psi}_{1, T})+ \vertiii{\Psi}_{1, T}^2).
	\end{align}	
	The estimate of $\vertiii{\boldsymbol{ \eta}_T}_{0,T}$ in \eqref{estimate eta_T} and \eqref{b_E estimates}  together with the above three displayed estimates in \eqref{lem 4.4.11} lead to 
	\begin{align}\label{estimate eta_E}
	h_E^{\frac{1}{2}}  \vertiii{\boldsymbol{ \eta}_E}_{0,E} \lesssim \sum_{T \in \omega_E}	\vertiii{\Psi - \Phi_{h}}_{h, T} (1+ \epsilon^{-2} (1+ \vertiii{\Psi}_{1, T}^2 +\vertiii{\Psi-\Phi_{h}}_{h, T} (\vertiii{\Phi_h}_{1}+ \vertiii{\Psi}_{1, T}) )).
	\end{align}
	A combination of \eqref{estimate eta_T}	 and \ref{estimate eta_E} completes the proof of $(i)$ in Lemma \ref{lem 3.1}.
	
	\medskip	
	\noindent $(ii)$	For  $\Phi_{h}={ \rm I}_h \Psi$ in \eqref{3.33},  Lemma \ref{boundedness}$(v)$ and \eqref{b_T estimates} yield
	\begin{align}\label{estimation_B}
	B_T({\rm I}_{h} \Psi,{\rm I}_{h} \Psi, {\rm I}_{h} \Psi, \boldsymbol{\rho}_T)- B_T( \Psi, \Psi,  \Psi, \boldsymbol{\rho}_T) &\lesssim \epsilon^{-2}\vertiii{ \Psi}_{1+\alpha,T}^3 (h_T^{2\alpha}  \vertiii{\nabla \boldsymbol{\rho}_T}_{0,T}+ h_T^{1+\alpha}\vertiii{\boldsymbol{\rho}_T}_{0,T})\notag \\& \lesssim \epsilon^{-2}  \vertiii{ \Psi}_{1+\alpha,T}^3(h_T^{2\alpha}  + h_T^{2+\alpha} )h_T^{-1}\vertiii{\boldsymbol{ \eta}_T}_{0,T} .
	\end{align}	
	Substitute \eqref{estimation_A}, \eqref{estimation_C}, \eqref{estimation_B}  in \eqref{lem 4.4.1}   and utilize Lemma \ref{Interpolation estimate} to arrive at 
	\begin{align}\label{estimate eta_T1}
	h_T\vertiii{\boldsymbol{ \eta}_T}_{0,T}& \lesssim \vertiii{\nabla ({\rm I}_{h} \Psi-\Psi)}_{0,T}+ \epsilon^{-2} \vertiii{{\rm I}_{h} \Psi-\Psi}_{0,T}+  \epsilon^{-2} h_T^{2\alpha}  \vertiii{ \Psi}_{1+\alpha}^3 
	\notag \\& \lesssim h_T^{\alpha} (1+ \epsilon^{-2}h_T^{\alpha}(1+ \vertiii{ \Psi}_{1+\alpha}^2)) \vertiii{ \Psi}_{1+\alpha}.
	\end{align}
	A choice of  $\Phi_h = {\rm I}_h \Psi $ in \eqref{lem 4.4.14}, Lemma \ref{boundedness}$(v)$ and \eqref{b_E estimates} yield
	\begin{align}\label{B_E}
	&\sum_{T \in \omega_E}(	B_T({\rm I}_{h} \Psi,{\rm I}_{h} \Psi, {\rm I}_{h} \Psi, \boldsymbol{\rho}_E)- B_T( \Psi, \Psi,  \Psi, \boldsymbol{\rho}_E)) \lesssim \epsilon^{-2}\sum_{T \in \omega_E}\vertiii{ \Psi}_{1+\alpha, T}^3 (h_T^{2\alpha}  \vertiii{\nabla \boldsymbol{\rho}_E}_{0,T}+ h_T^{1+\alpha}\vertiii{\boldsymbol{\rho}_E}_{0,T})\notag \\& 
	\qquad \lesssim \epsilon^{-2} h_E^{-\frac{1}{2}}\vertiii{\boldsymbol{ \eta}_E}_{0,E} \sum_{T \in \omega_E} \vertiii{ \Psi}_{1+\alpha,T}^3(h_T^{2\alpha}  +h_E h_T^{2+\alpha} ) .
	\end{align}	
	Substitute \eqref{A_E}, \eqref{C_E}, \eqref{B_E} in \eqref{lem 4.4.11} and employ Lemma \ref{Interpolation estimate} to obtain
	\begin{align} \label{estimate eta_E1}
	h_E^{\frac{1}{2}}	\vertiii{\boldsymbol{ \eta}_E}_{0,E}
	\lesssim \sum_{T \in \omega_E} h_T^{\alpha} (1+ \epsilon^{-2}h_T^{\alpha}(1+ \vertiii{ \Psi}_{1+\alpha}^2)) \vertiii{ \Psi}_{1+\alpha}.
	\end{align}
	A combination of  \eqref{estimate eta_T1} and \eqref{estimate eta_E1}  concludes the proof of $(ii)$ in Lemma \ref{lem 3.1}.
\end{proof}
\noindent The proof of Lemma \ref{local efficiency for discrete infsup} (resp. \ref{local efficiency for L2 estimate}) follows analgous to the proof of Lemma \ref{lem 3.1}  with the choice of  
$$\boldsymbol{\rho}_T:=
\left\{
\begin{array}{ll@{:}}
(\Delta (\textrm{I}_h \boldsymbol \xi) +2 \epsilon^{-2}(\abs{\textrm{I}_h \Psi}^2  \Theta_{h} + 2 (\textrm{I}_h \Psi \cdot  \Theta_{h}) \textrm{I}_h \Psi -\Theta_{h}) )b_T \text{ in } T\\
0 \quad \text{in } \Omega \setminus  T
\end{array}
\right.
\text{ and }
\boldsymbol{\rho}_E:=\left\{
\begin{array}{l}
[\nabla(\textrm{I}_h\boldsymbol \xi) \nu] b_E \text{ in } \omega_E\\
0 \quad \text{in } \Omega \setminus  \omega_E
\end{array}
\right..
$$
$$ \bigg(\text {resp. } 
\boldsymbol{\rho}_T:=
\left\{
\begin{array}{l}
\big(G_h +\Delta ( \textrm{I}_{h}\boldsymbol{\chi}) -2\epsilon^{-2} (\abs{\textrm{I}_{h}\Psi}^2 \textrm{I}_{h}\boldsymbol{\chi} +2 (\Ihpsi \cdot \Ichi)\Ihpsi- \Ichi)\big)b_T \quad \text{in } T\\
0 \quad \text{in } \Omega \setminus  T
\end{array}
\right . $$
$$\text{ and }  \boldsymbol{\rho}_E:=\left\{
\begin{array}{l}
[ \nabla(\textrm{I}_h\boldsymbol{\chi}) \nu ]b_E \quad \text{in } \omega_E\\
0 \quad \text{in } \Omega \setminus  \omega_E
\end{array}
\right. \bigg).
$$
\end{appendices}

\end{document}